\documentclass[10pt,a4paper]{article}
\usepackage{amssymb}
\usepackage{amsmath}
\usepackage{graphicx}
\usepackage[T1]{fontenc} 
\usepackage[english]{babel}
\usepackage{amsthm}
\usepackage{tabulary}
\usepackage{booktabs}
\usepackage{mathtools}
\usepackage{tikz}
\usepackage{comment}
\usepackage{pst-node}
\usepackage{tikz-cd}
\usepackage{shuffle}
\usepackage{comment}
\usepackage{xcolor}

\usepackage{difftrees}
\tikzset{
  ncone/.pic={
	\draw (0,0)--(0,0.2);
  }
}

\tikzset{
  nctwo/.pic={
    \draw (0,0)--(0,0.2);
	\draw (0.1,0)--(0.1,0.2);
  }
}

\tikzset{
  nctwoW/.pic={
    \draw (0,0.2)--(0,0)--(0.1,0)--(0.1,0.2);
  }
}

\tikzset{
  nctwoWW/.pic={
    \draw (0,0.2)--(0,0)--(0.2,0)--(0.2,0.2);
  }
}

\tikzset{
  ncthreeWW/.pic={
    \draw (0,0.2)--(0,0)--(0.3,0)--(0.3,0.2);
	\draw (0.2,0)--(0.2,0.2);
  }
}

\tikzset{
  ncthree/.pic={
    \draw (0,0)--(0,0.2);
	\draw (0.1,0)--(0.1,0.2);
	\draw (0.2,0)--(0.2,0.2);
  }
}

\tikzset{
  ncthreeW/.pic={
    \draw (0,0.2)--(0,0)--(0.2,0)--(0.2,0.2);
	\draw (0.1,0)--(0.1,0.2);
  }
}

\tikzset{
  ncfour/.pic={
    \draw (0,0)--(0,0.2);
	\draw (0.1,0)--(0.1,0.2);
	\draw (0.2,0)--(0.2,0.2);
	\draw (0.3,0)--(0.3,0.2);
  }
}
\tikzset{
  ncfive/.pic={
    \draw (0,0)--(0,0.2);
	\draw (0.1,0)--(0.1,0.2);
	\draw (0.2,0)--(0.2,0.2);
	\draw (0.3,0)--(0.3,0.2);
	\draw (0.4,0)--(0.4,0.2);
  }
}

\tikzset{
  ncfourW/.pic={
    \draw (0,0.2)--(0,0)--(0.3,0)--(0.3,0.2);
	\draw (0.1,0)--(0.1,0.2);
	\draw (0.2,0)--(0.2,0.2);
  }
}
\tikzset{
  ncfiveW/.pic={
    \draw (0,0.2)--(0,0)--(0.4,0)--(0.4,0.2);
	\draw (0.1,0)--(0.1,0.2);
	\draw (0.2,0)--(0.2,0.2);
	\draw (0.3,0)--(0.3,0.2);
  }
}

\tikzset{
  nconeinsidetwoWW/.pic={
    \path (0,0) pic {nctwoWW}; \path (0.1,0.1) pic {ncone};
  }
}

\tikzset{
  nconeinsidethreeWW/.pic={
    \path (0,0) pic {ncthreeWW}; \path (0.1,0.1) pic {ncone};
  }
}
\tikzset{
  nconeinsidethreerightWW/.pic={
    \draw (0,0.2)--(0,0)--(0.3,0)--(0.3,0.2);
    \draw (0.1,0)--(0.1,0.2); \draw (0.2,0.1)--(0.2,0.3);
  }
}
\tikzset{
  nconeinsidethreeleftWW/.pic={
    \draw (0,0.2)--(0,0)--(0.3,0)--(0.3,0.2);
    \draw (0.2,0)--(0.2,0.2); \draw (0.1,0.1)--(0.1,0.3);
  }
}
\tikzset{
  nctwoWWW/.pic={
    \draw (0,0.2)--(0,0)--(0.3,0)--(0.3,0.2);
  }
}

\tikzset{
  nconeoneinsidetwoWW/.pic={
	\path (0,0) pic {ncone};
    \path (0.1,0) pic {nctwoWW}; 
	\path (0.2,0.1) pic {ncone};
  }
}

\usepackage{forest}
\forestset{
	decor/.style = {
		label/.expanded = {[inner sep = 0.2ex]right:{\scalebox{0.5}{$#1$}}}
	},
	root/.style = {minimum size = 0.1ex},
	decorated/.style = {
		for tree = {
			circle, fill, inner sep = 0.3ex, minimum size = 1.ex,
			grow' = north, l = 0, l sep = 4ex, s sep = 2.9em,
			fit = tight, parent anchor = center, child anchor = center,
			delay = {decor/.option = content, content =}
		}
	},
	default preamble = {decorated, root},
	begin draw/.code={\begin{tikzpicture}[baseline={([yshift=-1ex]current bounding box.center)}]},
}

\newtheorem{theorem}{Theorem}[section]
\newtheorem{corollary}{Corollary}[theorem]
\newtheorem{lemma}[theorem]{Lemma}
\newtheorem{proposition}[theorem]{Proposition}
\newtheorem{definition}[theorem]{Definition}
\newtheorem{example}[theorem]{Example}
\newtheorem{remark}{Remark}

\DeclareMathOperator{\graft}{\triangleright}
\DeclareMathOperator{\dgraft}{\widehat{\triangleright}}
\DeclareMathOperator{\dgraftt}{\widehat{\blacktriangleright}}
\DeclareMathOperator{\dinsert}{\widehat{\diamond}}

\title{Planar Regularity Structures}

\author{Ludwig Rahm\footnote{Department of Mathematical Sciences, Norwegian University of Science and Technology (NTNU), 7491 Trondheim, Norway. \texttt{ludwig.rahm@ntnu.no}.}}

\begin{document}
\maketitle

\begin{abstract}
Branched rough paths, used to solve ODEs on $\mathbb{R}$, have been generalised in two different directions. In one direction, there are regularity structures aimed at solving SPDEs on $\mathbb{R}$. In the other direction, there are planarly branched rough paths to solve ODEs on homogeneous spaces. This paper combines these two directions to construct planar regularity structures, for (S)PDEs on homogeneous spaces.
\end{abstract}

\section{Introduction}
\label{sect:intro}

Geometric and branched rough paths \cite{Gubinelli2010, Lyons1998} were introduced to solve so-called rough differential equations
\begin{align*}
y_t=f(y)dX,
\end{align*}
characterised by being driven by some noise. The unknown $y$ is a map $\mathbb{R} \to \mathbb{R}$. This has been generalised in two different directions. In one direction, rough paths generalise to regularity structures \cite{Hairer2013} to solve for maps $y: \mathbb{R}^d \to \mathbb{R}$. In the other direction, rough paths generalise to planarly branched rough paths \cite{CurryEbrahimi-FardManchonMunthe-Kaas2018} to solve for maps $\mathbb{R} \to M$, where the target space $M$ is a homogeneous space. Considering these two results, there should reasonably be a way to generalise both the initial space and the target space at the same time. Indeed, the ultimate goal of this paper is to construct such a generalisation, which we call \textit{planar regularity structures}.

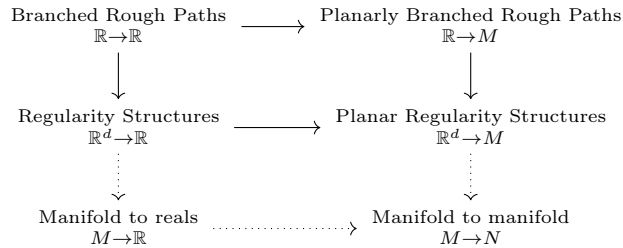
\[\begin{tikzcd}
 \text{Branched Rough Paths} \atop \mathbb{R}\to \mathbb{R}  \arrow{r} \arrow{d} & \text{Planarly Branched Rough Paths} \atop \mathbb{R} \to M \arrow{d} \\
 \text{Regularity Structures}\atop \mathbb{R}^d \to \mathbb{R}  \arrow{r} \arrow[dotted]{d} & \text{Planar Regularity Structures} \atop \mathbb{R}^d \to M \arrow[dotted]{d} \\
\text{Manifold to reals} \atop M \to \mathbb{R} \arrow[dotted]{r}& \text{Manifold to manifold} \atop M \to N 
\end{tikzcd}
\]

Rough paths, as well as regularity structures, have been developed and studied through the formalism of Hopf algebras on rooted trees. Indeed, one can understand the generalisation from branched rough paths to planarly branched rough paths as a Hopf algebra morphism. Starting from the free post-Lie algebra structure on planar rooted trees, and applying the post-Lie version of the Guin--Oudom procedure \cite{EbrahimiFardLundervoldMuntheKaas2014,GuinOudom2008}, one obtains the planar Grossman--Larson product. Dualising this product gives the Munthe-Kaas--Wright Hopf algebra $\mathcal{H}_{MKW}$. A planarly branched rough paths can be seen as a map from $\mathbb{R}^2$ taking values in the character group of the Hopf algebra $\mathcal{H}_{MKW}$. One can then find the Butcher--Connes--Kreimer Hopf algebra $\mathcal{H}_{BCK}$ as a Hopf subalgebra of $\mathcal{H}_{MKW}$. A branched rough path on the other hand can be seen as a map on $\mathbb{R}^2$ that takes values in the character group of $\mathcal{H}_{BCK}$. The Butcher--Connes--Kreimer Hopf algebra can furthermore be obtained as the dual of the non-planar Grossman--Larson product, which one gets from applying the Guin--Oudom procedure to the free pre-Lie algebra structure on non-planar rooted trees. \\
To get to the algebraic structures used in regularity structures, one again starts with the free pre-Lie algebra structure on non-planar rooted trees. This structure is then altered to obtain the so-called \textit{deformed pre-Lie product}. The vector space of non-planar rooted trees is then augmented to furthermore contain the unit vectors in $\mathbb{N}^d$ in its basis. Extending the deformed pre-Lie product, and defining a Lie bracket, turns this larger space into a post-Lie algebra. One can then apply the post-Lie version of the Guin--Oudom procedure to obtain a deformed Grossman--Larson product. One of the relevant structures for regularity structures is the coaction obtained from dualising this deformed Grossman--Larson product \cite{BrunedKatsetsiadis2022}. This is the coaction for re-centering, also called positive renormalisation. The other relevant structure is the coaction for renormalisation, also called negative renormalisation. It was shown in \cite{BrunedManchon2022} that one can use the deformed Grossman--Larson product to define a new pre-Lie product on the non-planar rooted trees. Applying the Guin--Oudom procedure to this new pre-Lie structure yields a Grossman--Larson-like product, which is dual to the coaction for renormalisation.\\
The generalisation from Branched rough paths to regularity structures was studied in \cite{BrunedChevyrevFrizPreiss2017}, where the authors established a bijection between the set of branched rough paths, and a subset of all possible regularity structures. In this sense, one can see branched rough paths as a special case of regularity structures. In this bijection, re-centering corresponds to Connes--Kreimer coproduct on the Butcher--Connes--Kreimer Hopf algebra. The authors furthermore introduced the concept of \textit{translations of rough paths}, which corresponds to renormalisation of regularity structures. The concept of translation was extended to planarly branched rough paths in \cite{Rahm2021RP}.\\

This paper aims to introduce the algebraic structures for planar regularity structures, which generalise both the structures from planarly branched rough paths, and those from regularity structures, at the same time. We will accomplish this by following the procedures used in \cite{BrunedManchon2022} and \cite{BrunedKatsetsiadis2022}, where the authors obtained the algebraic structures for regularity structures by starting with the free pre-Lie algebra structure on non-planar rooted trees. However, we will instead start with the free post-Lie algebra structure on planar rooted trees, and perform the corresponding procedures. We will furthermore show how planarly branched rough paths can be seen as regularity structures.

\smallskip 

The paper is organised as follows. In Section \ref{sect:prelim}, we summarize the definitions and results that the present paper builds upon. In Section \ref{sec::PlanarRoughPathsAreRegularityStructures}, we show how to obtain a regularity structure from a planarly branched rough path. In Section \ref{sec::PlanarRegularityStructures}, we construct the algebraic structures of planar regularity structures. First we construct a re-centering coaction, followed by a renormalisation coaction. Finally we show an important cointeraction property between the coactions.

\medskip 

{\bf{Acknowledgements}}: The author is supported by the Research Council of Norway through project 302831 ''Computational Dynamics and Stochastics on Manifolds" (CODYSMA). \\
This work grew out of discussions with Dominique Manchon during the author's visit to Université Clermont-Auvergne in spring 2022. The author thanks the university for the visit, and thanks Dominique for the helpful discussions both during and after the visit. The author thanks Kurusch Ebrahimi-Fard for helpful discussions.

\section{Preliminaries}
\label{sect:prelim}

We review the theory of rough paths, regularity structures and planarly branched rough paths. We furthermore outline how rough paths can be considered a special case of both regularity structures as well as planarly branched rough paths. We will in later sections combine these two different generalisations of rough paths into planar regularity structures, which will generalise both regularity structures and planarly branched rough paths.

\subsection{Hopf algebras on rooted trees, planar and non-planar}
\label{ssect:trees}

Smooth ordinary differential equations (ODEs) on $\mathbb{R}$ have been studied using Butcher's $B$-series methods \cite{HLW2006}, which can be described in terms of characters on the Butcher--Connes--Kreimer (BCK) Hopf algebra over undecorated non-planar rooted trees \cite{CK1998}. In the generalisation from smooth ODEs to rough differential equations, by means of so-called branched rough paths \cite{Gubinelli2010}, characters in the BCK Hopf algebra over non-planar rooted trees with decorated vertices are considered. Further generalising to rough partial differential equations (PDEs), using the notion of regularity structures \cite{Hairer2013}, characters from a BCK-like Hopf algebra over non-planar rooted trees with decorated vertices and decorated edges appear. There is a similar story for differential equations with a homogeneous space as the target space. Smooth ODEs on a homogeneous space have been studied using so-called Lie--Butcher series (LB-series), a blend of Lie- and Butcher's $B$-series, which can be described in terms of characters on the Munthe-Kaas--Wright (MKW) Hopf algebra over undecorated planar rooted trees \cite{MuntheKaasWright2008}. The MKW Hopf algebra is a generalisation of the BCK Hopf algebra to planar rooted trees. Further generalising to rough ODEs results in the notion of planarly branched rough paths \cite{CurryEbrahimi-FardManchonMunthe-Kaas2018}, which is based on characters over the MKW Hopf algebra over planar rooted trees with decorated vertices. A central goal of this paper is to look at planarly branched rough paths from the viewpoint of regularity structures by constructing a MKW-like Hopf algebra over planar rooted trees with both decorated vertices and decorated edges. 

The purpose of this section is to describe the above mentioned Hopf algebras over planar and non-planar rooted trees. We will begin with the MKW Hopf algebra over planar rooted trees with decorated vertices, and then see how Hopf algebras of non-planar rooted trees can be described as subalgebras.

Consider the set of planar rooted trees
\begin{align*}
	\mathbf{PT}=\Big\{ \Forest{[]}, \scalebox{0.8}{\Forest{[[]]}}, \scalebox{0.8}{\Forest{[[][]]}}, \scalebox{0.8}{\Forest{[[[]]]}}, \scalebox{0.8}{\Forest{[[][][]]}}, \scalebox{0.8}{\Forest{[[[]][]]}}, \scalebox{0.8}{\Forest{[[][[]]]}}, \scalebox{0.8}{\Forest{[[[[]]]]}},\ldots \Big\},
\end{align*}
and the vector space $\mathcal{PT}=\langle \mathbf{PT} \rangle$. A tree $\tau \in \mathbf{PT}$ has a set of vertices, $V(\tau)$, and a set of edges, $E(\tau)$. With the exception of a particular vertex, every $v \in V(\tau)$ has exactly one incoming and arbitrarily many outgoing edges. The special vertex, called the root, is drawn at the bottom of the tree and has no incoming edge. Vertices with no outgoing edges are the leaves of a tree. Elements in $\mathbf{PT}$ are furthermore endowed with an embedding in the plane. This means, for example, that the two following trees are considered to be different
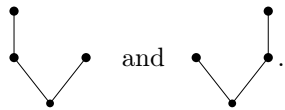
\begin{align*}
	\scalebox{0.8}{\Forest{[[[]][]]}}\quad \text{and} \quad \scalebox{0.8}{\Forest{[[][[]]]}}.
\end{align*}

We say that trees in $\mathbf{PT}$ are decorated by the set $\mathcal{C}$, if for each tree there exists a map from the set of its vertices to $\mathcal{C}$. Decorations are denoted next to the vertices. Similarly, the set of edges can be decorated as well. For instance, the following planar rooted tree has both decorated vertices and decorated edges:
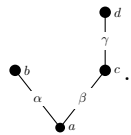
\begin{align*}
	\Forest{[a[b,edge label = {node[midway,fill=white,scale=0.5]{$\alpha$}}]
	[c,edge label = {node[midway,fill=white,scale=0.5]{$\beta$}}
	[d,edge label = {node[midway,fill=white,scale=0.5]{$\gamma$}}]]]}.
\end{align*}

Let $\mathbf{PT}_{\mathcal{C}}$ denote the set of planar rooted trees whose vertices are decorated by $\mathcal{C}$, and let $(Lie(\mathcal{PT}_{\mathcal{C}}),[\cdot,\cdot])$ be the free Lie algebra generated by $\mathbf{PT}_{\mathcal{C}}$. We can endow $\mathcal{PT}_{\mathcal{C}}$ with the left grafting product $\graft$, where $\tau_1 \graft \tau_2$ is the sum over all possible trees obtained by adding to any vertex in the planar embedding of $\tau_2$ a new leftmost edge to the root of $\tau_1$. For example:
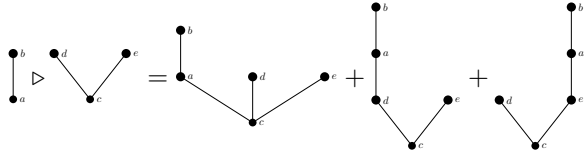
\begin{align*}
	\scalebox{0.8}{\Forest{[a[b]]}}\graft \scalebox{0.8}{\Forest{[c[d][e]]}}
	=\scalebox{0.8}{\Forest{[c[a[b]][d][e]]}}
	+\scalebox{0.8}{\Forest{[c[d[a[b]]][e]]}}
	+\scalebox{0.8}{\Forest{[c[d][e[a[b]]]]}}.
\end{align*}
Extending the left grafting product $\graft$ to $Lie(\mathcal{PT}_{\mathcal{C}})$ by the rules:
\begin{align}
	[x,y]\graft z 
		&= x \graft (y \graft z) - (x \graft y) \graft z - y \graft (x \graft z) + (y \graft x) \graft z \label{eq::PostLieAssociator},\\
	x \graft [y,z]
		&=[x \graft y,z] + [y,x \graft z], \label{eq::PostLieDerivation}
\end{align}
turns $(Lie(\mathcal{PT}_{\mathcal{C}}),\graft,[\cdot,\cdot])$ into the free post-Lie algebra generated by the set $\mathcal{C}$. The rules \eqref{eq::PostLieAssociator} and \eqref{eq::PostLieDerivation} define the two axioms characterising the notion of post-Lie algebra \cite{Foissy2018}\cite{AlKaabiEbrahimiFardManchonMuntheKaas2022}.

Consider now the free non-commutative associative algebra $(\mathcal{OF}_{\mathcal{C}},\cdot)$ generated by $\mathcal{PT}_{\mathcal{C}}$. Its associative product is the concatenation of trees giving \textit{ordered forests}:
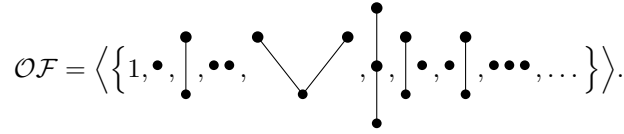
\begin{align*}
	\mathcal{OF}= \Big\langle \Big\{1,\Forest{[]},\Forest{[[]]},\Forest{[]}\Forest{[]},\Forest{[[][]]},\Forest{[[[]]]},\Forest{[[]]}\Forest{[]},\Forest{[]}\Forest{[[]]},\Forest{[]}\Forest{[]}\Forest{[]},\dots  \Big\}  \Big\rangle.
\end{align*}
Note that this is the universal enveloping algebra of the Lie algebra $(Lie(\mathcal{PT}_{\mathcal{C}}),[\cdot,\cdot])$. 
The left grafting product $\graft$ can be extended to $\mathcal{OF}_{\mathcal{C}}$ by the post-Lie version of the Guin--Oudom construction \cite{GuinOudom2008, EbrahimiFardLundervoldMuntheKaas2014}:
\begin{align}
	1 \graft \omega 
		&= \omega, \label{eq::GuinOudom1} \\
	\omega \graft 1 
		&= 0, \label{eq::GuinOudom2} \\
	(\tau_1 \omega) \graft \tau_2 
		&= \tau_1 \graft (\omega \graft \tau_2)-(\tau_1 \graft \omega)\graft \tau_2, \label{eq::GuinOudom3} \\
	\omega \graft (\omega_1 \omega_2) 
		&= (\omega_{(1)}\graft \omega_1)(\omega_{(2)} \graft \omega_2), \label{eq::GuinOudom4}
\end{align}
for $\tau_1,\tau_2 \in Lie(\mathcal{PT}_{\mathcal{C}})$ and  $\omega,\omega_1,\omega_2 \in \mathcal{OF}_{\mathcal{C}}$.  We use Sweedler's notation to denote the deshuffle coproduct 
\begin{align*}
	\Delta_{\shuffle}(\omega) = \omega_{(1)} \otimes \omega_{(2)}.
\end{align*}
This construction gives the free D-algebra $(\mathcal{OF}_{\mathcal{C}},\cdot,\graft)$ \cite{MuntheKaasWright2008}. Recall that a D-algebra $(A,\cdot,\graft)$ consists of a unital associative algebra $(A,\cdot)$ together with a second product $\graft$. Letting $D(A)$ denote the set $\{x \in A : x \graft (ab)=(x\graft a)b+a(x \graft b), \forall a,b \in A  \}$, the D-algebra axioms are given by:
	\begin{align*}
		1 \graft a =& a, \\
		a \graft x \in & D(A), \\
		x \graft (a \graft b)=&(x \cdot a) \graft b + (x \graft a) \graft b,
	\end{align*}
for $x \in D(A)$ and $a,b \in A$. 

Following \cite{GuinOudom2008, EbrahimiFardLundervoldMuntheKaas2014}, we now define a second associative product $\ast$ on $\mathcal{OF}_{\mathcal{C}}$, known as planar Grossman--Larson product:
\begin{align}
	\omega \ast \omega'=\omega_{(1)}(\omega_{(2)} \graft \omega'). \label{eq::GrossmanLarson}
\end{align}
Then $(\mathcal{OF}_{\mathcal{C}},\ast,\Delta_{\shuffle})$ is the dual Hopf algebra to the Munthe-Kaas--Wright (MKW) Hopf algebra $\mathcal{H}_{MKW}^{\mathcal{C}}$, which we will now describe. There exists a bijection between ordered forests and planar rooted trees, denoted $B_+$, which grafts all of the trees in the forest onto a common (undecorated) root. The planar embedding of the tree is given by the non-commutative order of the forest. We denote the inverse map by $B_-$. For example:
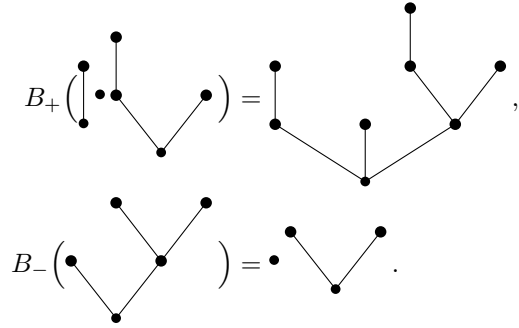
\begin{align*}
	B_+\Big(\Forest{[[]]}\Forest{[]}\Forest{[[[]][]]}\Big)
	&= \Forest{[[[]][][[[]][]]]},\\
	B_-\Big(\Forest{[[][[][]]]}\Big)
	&= \Forest{[]}\Forest{[[][]]}.
\end{align*}
Let $\tau$ be a planar forest and let $c$ be a subset of its edges. We say that $c$ is a \textit{left admissible cut} if:
\begin{itemize}
\item Any path from the root of $\tau$ to a leaf, contains at most one edge in $c$.
\item If $e \in c$, then every other outgoing edge of the vertex that is to the left of $e$ in the planar embedding is also in $c$.
\end{itemize}
For $c$ an admissible cut, we define the pruned part $P^c(\tau)$ and the trunk $T^c(\tau)$. The trunk $T^c(\tau)$ is given by the connected component containing the root, obtained after removing all edges in $c$ from $\tau$. The pruned part $P^c(\tau)$ is obtained by first concatenating connected components that were cut off from the same vertex in $T^c(\tau)$ into a forest, and then taking the shuffle product of these forests. We now define the MKW coproduct $\Delta_{MKW}: \mathcal{OF}_{\mathcal{C}} \to \mathcal{OF}_{\mathcal{C}}$ by:
\begin{align} 
\label{eq::MKW}
	\Delta_{MKW}(\omega)=\sum_{c \text{ left admis.~cut of }B_+(\omega)} P^c(B_+(\omega)) \otimes B_-(T^c(B_+(\omega))).
\end{align}
We can think of the MKW coproduct as taking left admissible cuts of a forest, where we are allowed to cut between the different trees of the forest. The formula achieves this by grafting the forest onto a root using the $B_+$ operation, then making the cut, and finally removing the added root from the trunk. For example:
\begin{align*}
	\Delta_{MKW}\Big(\Forest{[a]}\Forest{[b[c][d]]}\Big)
	&=1 \otimes \Forest{[a]}\Forest{[b[c][d]]} + \Forest{[a]}\Forest{[b[c][d]]}\otimes 1 
		+ \Forest{[a]} \otimes \Forest{[b[c][d]]}\\
	&+ \Forest{[c]}\otimes \Forest{[a]}\Forest{[b[d]]}
		+ \Forest{[a]}\shuffle \Forest{[c]}\otimes \Forest{[b[d]]} 
		+ \Forest{[c]}\Forest{[d]}\otimes \Forest{[a]}\Forest{[b]} 
		+ \Forest{[a]}\shuffle (\Forest{[c]}\Forest{[d]})\otimes \Forest{[b]}.
\end{align*}
The MKW Hopf algebra with vertex decoration $\mathcal{C}$, as used in planarly branched rough paths, is defined as $\mathcal{H}_{MKW}^{\mathcal{C}}=(\mathcal{OF}_{\mathcal{C}},\shuffle,\Delta_{MKW})$, where $\shuffle$ is the shuffle product. The case of LB-series and undecorated trees is obtained by forgetting the vertex decorations. Recall $Lie(\mathcal{PT}_{\mathcal{C}})$, the free Lie algebra generated by $\mathcal{PT}_{\mathcal{C}}$, and let $S_{\mathcal{C}}=(S(Lie(\mathcal{PT}_{\mathcal{C}})\times \mathcal{C}),\centerdot)$ be the free symmetric algebra generated by $Lie(\mathcal{PT}_{\mathcal{C}})\times \mathcal{C}$. The important concepts of substitution of LB-series \cite{LundMuntheK2013,MuntheKFllesdal2018} and translations of planarly branched rough paths \cite{Rahm2021RP}, which will later be generalised to renormalisation of regularity structures, can be described by two coactions, $\rho_S$ respectively $\rho_T$, from $\mathcal{H}_{MKW}$ to $S_{\mathcal{C}} \otimes \mathcal{H}_{MKW}$. They are called cosubstitution and cotranslation, respectively. To describe these coactions, we define the notion of \textit{admissible partition}. Let $\omega$ be a forest and let $\omega_1,\ldots,\omega_n$ be disjoint subforests of $\omega$. Note that there is no order on the subforests, the subscripts only serve to label the subforests for notational purposes. We say that $\omega_1,\ldots,\omega_n$ is an admissible partition if:
\begin{itemize}
\item Either all roots of $\omega_i$ are roots of $\omega$, or if every root of $\omega_i$ is grafted onto the same vertex of $\omega$. Furthermore, the roots of $\omega_i$ are adjacent in the planar embedding of $\omega$.
\item If $e$ is an edge in $\omega$, then every edge $e'$ that is outgoing from the same vertex as $e$ and is to the right of $e$ in the planar embedding of $\omega$, is also in $\omega_i$.
\end{itemize}
If every vertex of $\omega$ belongs to an $\omega_i$, we say that $\omega_1,\ldots,\omega_n$ is a \textit{spanning admissible partition}. For $\omega_1,\ldots,\omega_n$ admissible, we denote by $\omega/_{\textbf{c}}\ \omega_1,\ldots,\omega_n, \textbf{c}=(c_1,\ldots,c_n) \in \mathcal{C}^{\times n}$ the forest obtained by contracting each $\omega_i$ into a single vertex decorated by the $i:th$ component of $\textbf{c}$. We then define the coactions by:
\begin{align} 
\label{eq::CoSubstitution}
	\rho_S(\omega)
	&= \sum_{\omega_1,\dots,\omega_n \atop \text{spanning admiss.~partition}  } \sum_{\textbf{c} \in \mathcal{C}^{\times n}} (\pi(\omega_1),c_1)\centerdot \dots \centerdot (\pi(\omega_n), c_n) \otimes \omega /_{\textbf{c}} \omega_1,\dots,\omega_n, \\ \label{eq::CoTranslation}
\rho_T(\omega)
	&= \sum_{\omega_1,\dots,\omega_n \atop \text{admiss.~partition}  } \sum_{\textbf{c} \in \mathcal{C}^{\times n}} (\pi(\omega_1),c_1)\centerdot \dots \centerdot (\pi(\omega_n), c_n) \otimes \omega /_{\textbf{c}} \omega_1,\dots,\omega_n,
\end{align}
where $\pi:\mathcal{OF}_{\mathcal{C}} \to Lie(\mathcal{PT}_{\mathcal{C}})$ is the projection onto Lie polynomials. We illustrate this by an example. Consider the decorated tree
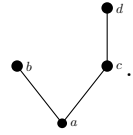
\begin{align*}
\Forest{[a[b][c[d]]]}.
\end{align*}
The spanning admissible partitions are as follows, where each $\omega_i$ is identified by vertex colours:
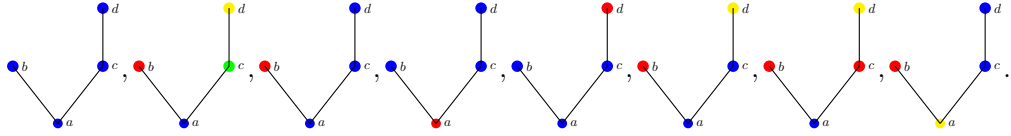
\begin{align*}
\Forest{[a,fill=blue[b,fill=blue][c,fill=blue[d,fill=blue]]]},\Forest{[a,fill=blue[b,fill=red][c,fill=green[d,fill=yellow]]]},\Forest{[a,fill=blue[b,fill=red][c,fill=blue[d,fill=blue]]]},\Forest{[a,fill=red[b,fill=blue][c,fill=blue[d,fill=blue]]]},\Forest{[a,fill=blue[b,fill=blue][c,fill=blue[d,fill=red]]]},\Forest{[a,fill=blue[b,fill=red][c,fill=blue[d,fill=yellow]]]},\Forest{[a,fill=blue[b,fill=red][c,fill=red[d,fill=yellow]]]},\Forest{[a,fill=yellow[b,fill=red][c,fill=blue[d,fill=blue]]]}.
\end{align*}
Hence the coaction $\rho_S$ evaluates to:
\begin{align*}
	\lefteqn{\rho_S\Big(\Forest{[a[b][c[d]]]}\Big)
	= \sum_{c_1 \in \mathcal{C}} \Big(\Forest{[a[b][c[d]]]},c_1\Big) \otimes \Forest{[c_1]} 
		+ \sum_{\textbf{c} \in \mathcal{C}^{\times 4}} \Big(\Forest{[a]},c_1\Big) \centerdot \Big(\Forest{[b]},c_2\Big) \centerdot \Big(\Forest{[c]},c_3\Big) \centerdot \Big(\Forest{[d]},c_4\Big) \otimes \Forest{[c_1[c_2][c_3[c_4]]]}}\\
	&+ \sum_{\textbf{c} \in \mathcal{C}^{\times 2}} \Big(\Forest{[b]},c_1\Big) \centerdot \Big(\Forest{[a[c[d]]]},c_2\Big) \otimes \Forest{[c_2[c_1]]}
		+ \sum_{\textbf{c} \in \mathcal{C}^{\times 2}} \Big(\Forest{[a]},c_1\Big) \centerdot \Big([\Forest{[b]},\Forest{[c[d]]}],c_2\Big) \otimes \Forest{[c_1[c_2]]}\\
	&+ \sum_{\textbf{c} \in \mathcal{C}^{\times 2}} \Big(\Forest{[d]},c_1\Big) \centerdot \Big(\Forest{[a[b][c]]},c_2\Big) \otimes \Forest{[c_2[c_1]]} 
		+ \sum_{\textbf{c} \in \mathcal{C}^{\times 3}} \Big(\Forest{[b]},c_1\Big) \centerdot \Big(\Forest{[d]},c_2\Big) \centerdot \Big(\Forest{[a[c]]},c_3\Big) \otimes \Big(\Forest{[c_3[c_1][c_2]]}+\Forest{[c_3[c_2][c_1]]} \Big)\\
	&+ \sum_{\textbf{c} \in \mathcal{C}^{\times 3}} \Big(\Forest{[a]},c_1\Big) \centerdot \Big([\Forest{[b]},\Forest{[c]}],c_2\Big) \centerdot \Big(\Forest{[d]},c_3\Big) \otimes \Forest{[c_1[c_2[c_3]]]} 
		+ \sum_{\textbf{c} \in \mathcal{C}^{\times 3}} \Big(\Forest{[a]},c_1\Big) \centerdot \Big(\Forest{[b]},c_2\Big) \centerdot \Big(\Forest{[c[d]]},c_3\Big) \otimes \Forest{[c_1[c_2][c_3]]}.
\end{align*}
We will now describe the Butcher--Connes--Kreimer (BCK) Hopf algebra defined over non-planar rooted trees. Recall that the latter have not been endowed with a planar embedding. The set of non-planar rooted trees is denoted by:
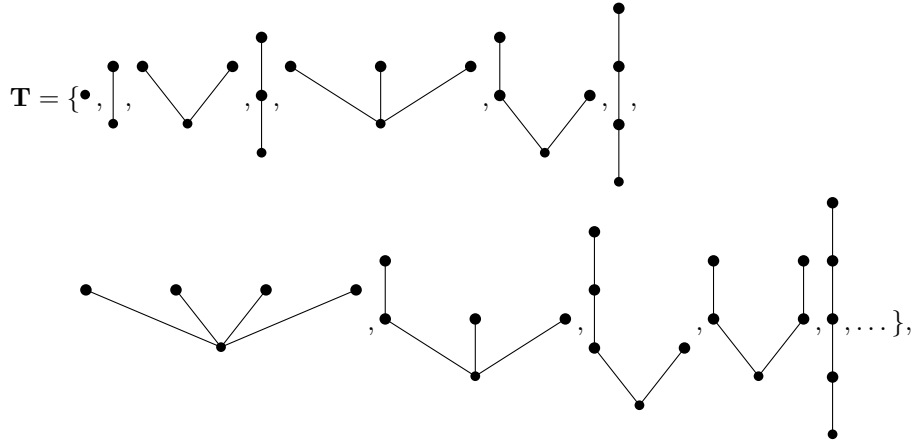
\begin{align*}
	\mathbf{T}=\{& \Forest{[]},\Forest{[[]]},\Forest{[[][]]},\Forest{[[[]]]},\Forest{[[][][]]},\Forest{[[[]][]]},\Forest{[[[[]]]]},\\
&\Forest{[[][][][]]},\Forest{[[[]][][]]},\Forest{[[[[]]][]]},\Forest{[[[]][[]]]},\Forest{[[[[[]]]]]},\dots \},
\end{align*}
and spans the vector space $\mathcal{T}$. The symmetric algebra $\mathcal{F}$ generated by $\mathcal{T}$ is spanned by elements from the set $\mathbf{F}$ of forests. We define the injection $\Omega: \mathcal{F} \to \mathcal{OF}$ as the sum over all ways to endow a non-planar forest with a planar embedding, for example:
\begin{align*}
	\Omega\Big(\Forest{[a[b[c]][d]]}\Forest{[e[f]]}\Big)
	&= \Forest{[a[b[c]][d]]}\Forest{[e[f]]}
	+\Forest{[a[d][b[c]]]}\Forest{[e[f]]}
	+\Forest{[e[f]]}\Forest{[a[b[c]][d]]}
	+\Forest{[e[f]]}\Forest{[a[d][b[c]]]}.
\end{align*}
The map $\Omega$ is in fact a Hopf algebra morphism from the BCK Hopf algebra $\mathcal{H}_{BCK}=\{\mathcal{F},\cdot,\Delta_{CK} \}$ to the MKW Hopf algebra \cite{MuntheKaasWright2008}. Hence the image of $\Omega$ is a Hopf subalgebra of $\mathcal{H}_{MKW}$, isomorphic to $\mathcal{H}_{BCK}$. The coproduct $\Delta_{CK}$ is given by \textit{admissible cuts}, which are subsets of edges that contain at most one edge from any path from the root to a leaf. Note that they differ from left admissible cuts (on planar trees), in that they have no requirement with respect to the planar embedding. We can then describe $\Delta_{CK}$ by:
\begin{align*}
	\Delta_{CK}(\omega)=\sum_{c \text{ admiss.~cut of }B_+(\omega)} P^c(B_+(\omega)) \otimes B_-(T^c(B_+(\omega))).
\end{align*}
Compare with equation \eqref{eq::MKW} that defined the MKW coproduct. We also obtain the non-planar versions of the two coactions $\rho_S$, $\rho_T$ in this way, i.e., by forgetting the planarity in the definition of admissible partitions. Denote them by $\rho_S^{np},\rho_T^{np}: \mathcal{F}_{\mathcal{C}} \to S(\mathcal{T}_{\mathcal{C}}\times \mathcal{C}) \otimes \mathcal{F}_{\mathcal{C}}$. They can be described by contracting subtrees, where we have no planar requirements on these subtrees, compare to \eqref{eq::CoSubstitution} and \eqref{eq::CoTranslation}:
\begin{align*}
	\rho_S^{np}(\omega)
	&= \sum_{\omega_1,\dots,\omega_n \atop \text{spanning disjoint subtrees}  } \sum_{\textbf{c} \in \mathcal{C}^{\times n}} (\omega_1,c_1)\centerdot \dots \centerdot (\omega_n, c_n) \otimes \omega /_{\textbf{c}}\ \omega_1,\dots,\omega_n, \\
	\rho_T^{np}(\omega)
	&= \sum_{\omega_1,\dots,\omega_n \atop \text{disjoint subtrees}  } \sum_{\textbf{c} \in \mathcal{C}^{\times n}} (\omega_1,c_1)\centerdot \dots \centerdot (\omega_n, c_n) \otimes \omega /_{\textbf{c}}\ \omega_1,\dots,\omega_n.
\end{align*}
An alternative way to arrive at the BCK Hopf algebra is to consider the free pre-Lie algebra structure on non-planar rooted trees. Recall that a pre-Lie algebra is a post-Lie algebra with abelian Lie bracket, i.e., they are characterised by the (left) pre-Lie axiom
\begin{align*}
	x \curvearrowright (y \curvearrowright z) - (x \curvearrowright y) \curvearrowright z 
	- y \curvearrowright (x \curvearrowright z) + (y \curvearrowright x) \curvearrowright z = 0,
\end{align*}
obtained by setting the Lie bracket to zero in \eqref{eq::PostLieAssociator} and \eqref{eq::PostLieDerivation}. The free pre-Lie product $\tau_1 \curvearrowright \tau_2$ of non-planar rooted trees is given as the sum over all ways to add an edge from any vertex of $\tau_2$ to the root of $\tau_1$. As in the post-Lie algebra case, we can extend the grafting product from $\mathcal{T}_{\mathcal{C}}$ to its universal enveloping algebra using the Guin--Oudom relations \eqref{eq::GuinOudom1}-\eqref{eq::GuinOudom4}. The universal enveloping algebra of $\mathcal{T}_{\mathcal{C}}$ is $\mathcal{F}_{\mathcal{C}}$. We then arrive at the dual Hopf algebra to $\mathcal{H}_{BCK}$ via equation \eqref{eq::GrossmanLarson}.

\medskip

We now sketch how the BCK-like Hopf algebra used in regularity structures is constructed following the recent work of Bruned and Katsetsiadis  \cite{BrunedKatsetsiadis2022}. The setting is the one of non-planar rooted trees where both vertices and edges are decorated by elements from $\mathbb{N}^d$, the space of these trees is denoted by $\mathfrak{T}$. A non-planar rooted tree is called \textit{planted} if its root is undecorated and has exactly one outgoing edge. Let $I_{\alpha}$ be the operator that grafts an input tree $\tau$ onto an undecorated root, by using an edge decorated by $\alpha$, for example:
\begin{align*}
	I_{\alpha}\Big(\Forest{[a[b,edge label = {node[midway,fill=white,scale=0.5]{$\beta$}}][c,edge label = {node[midway,fill=white,scale=0.5]{$\gamma$}}[d,edge label = {node[midway,fill=white,scale=0.5]{$\delta$}}]]]}\Big)
	=\Forest{[[a,edge label = {node[midway,fill=white,scale=0.5]{$\alpha$}}[b,edge label = {node[midway,fill=white,scale=0.5]{$\beta$}}][c,edge label = {node[midway,fill=white,scale=0.5]{$\gamma$}}[d,edge label = {node[midway,fill=white,scale=0.5]{$\delta$}}]]] ]}.
\end{align*}
Then any planted tree can be written as $I_{\alpha}(\tau)$, for some $\alpha \in \mathbb{N}^d$ and some decorated tree $\tau$. Planted trees can be endowed with the pre-Lie product defined by
\begin{align*}
	I_{\alpha}(\tau_1) \curvearrowright I_{\beta}(\tau_2) := I_{\beta}(\tau_1 \curvearrowright_{\alpha} \tau_2),
\end{align*}
where $\curvearrowright_{\alpha}$ means grafting $\tau_1$ onto $\tau_2$ using an edge decorated by $\alpha$. This is then deformed into the pre-Lie product
\begin{align*}
	I_{\alpha}(\tau_1) \widehat{\curvearrowright} I_{\beta}(\tau_2) 
	= \sum_{v \in N_{\tau_2}} \sum_{\ell \in \mathbb{N}^d} {n_v \choose \ell}I_{\beta}( \tau_1 \curvearrowright^v_{\alpha-\ell}(\uparrow_v^{-\ell} \tau_2 ) ),
\end{align*}
where the set of vertices of $\tau_2$ is denoted $N_{\tau_2}=V(\tau_2)$. The operation $\curvearrowright^v_{\alpha-\ell}$ means grafting onto the vertex $v$, $\uparrow_v^{-\ell}$ is defined by subtracting $\ell$ from the decoration of the vertex $v$, and $n_v$ means the decoration of the vertex $v$. For example:
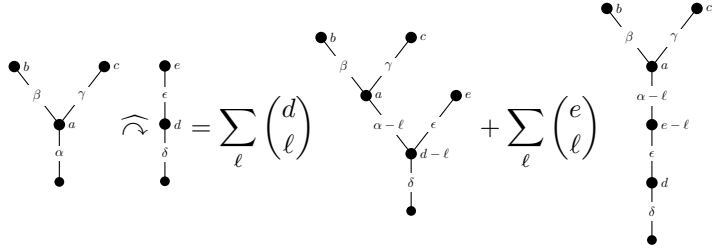
\begin{align*}
	\Forest{[ [a,edge label = {node[midway,fill=white,scale=0.5]{$\alpha$}}[b,edge label = {node[midway,fill=white,scale=0.5]{$\beta$}}][c,edge label = {node[midway,fill=white,scale=0.5]{$\gamma$}}]] ]} \widehat{\curvearrowright}\ \Forest{[ [d,edge label = {node[midway,fill=white,scale=0.5]{$\delta$}}[e,edge label = {node[midway,fill=white,scale=0.5]{$\epsilon$}}]] ]} =
 \sum_{\ell} {d \choose \ell} \Forest{[ [d-\ell,edge label = {node[midway,fill=white,scale=0.5]{$\delta$}}[a,edge label = {node[midway,fill=white,scale=0.5]{$\alpha-\ell$}}[b,edge label = {node[midway,fill=white,scale=0.5]{$\beta$}}][c,edge label = {node[midway,fill=white,scale=0.5]{$\gamma$}}]][e,edge label = {node[midway,fill=white,scale=0.5]{$\epsilon$}}]] ]}
 + \sum_{\ell} {e \choose \ell} \Forest{[ [d,edge label = {node[midway,fill=white,scale=0.5]{$\delta$}}[e-\ell,edge label = {node[midway,fill=white,scale=0.5]{$\epsilon$}}[a,edge label = {node[midway,fill=white,scale=0.5]{$\alpha-\ell$}}[b,edge label = {node[midway,fill=white,scale=0.5]{$\beta$}}][c,edge label = {node[midway,fill=white,scale=0.5]{$\gamma$}}]]]] ]}.
\end{align*}
Now consider the units $X^i, i=\{0,\dots,0,1,0,\dots,0\}\in \mathbb{N}^d$ and let $\mathcal{V}$ be the vector space that has planted trees and $X^i$'s as basis. Define the Lie bracket $[\cdot,\cdot]_0$:
\begin{align*}
[X^i,X^j]_0=&0,\\
[I_{\alpha}(\tau_1),I_{\beta}(\tau_2)]_0=&0,\\
[I_{\alpha}(\tau_1),X^i]=&I_{\alpha-i}(\tau_1),
\end{align*}
and extend $\widehat{\curvearrowright}$ to $\mathcal{V}$ by:
\begin{align*}
X^i \widehat{\curvearrowright} I_{\alpha}(\tau) =& \sum_{v \in N_{\tau}} I_{\alpha}( \uparrow_v^i \tau ),\\
I_{\alpha}(\tau) \widehat{\curvearrowright} X^i =&0,\\
X^i \widehat{\curvearrowright} X^j =& 0.
\end{align*}
Then $(\mathcal{V},\widehat{\curvearrowright},[\cdot,\cdot]_0)$ is a post-Lie algebra \cite{BrunedKatsetsiadis2022}. The product $\widehat{\curvearrowright}$ can be extended to the universal enveloping algebra by \eqref{eq::GuinOudom1}-\eqref{eq::GuinOudom4}, and one obtains a Hopf algebra by defining the product $\ast_+$ according to equation \eqref{eq::GrossmanLarson}. This is the dual Hopf algebra to the Connes--Kreimer--like Hopf algebra used in regularity structures.

Lastly, we want to sketch how the cotranslation (renormalisation) coaction was constructed for regularity structures in \cite{BrunedManchon2022}. We define a product $\dinsert$ by:
\begin{align*}
\tau_1 \dinsert_v \tau_2 =& (P_v(\tau_2) \ast_+ \tau_1) \curvearrowright_v T_v(\tau_2),\\
\tau_1 \dinsert \tau_2 =& \sum_{v \text{ vertex of } \tau_2} \tau_1 \dinsert_v \tau_2,
\end{align*}
where $P_v(\tau_2)$ is the subtree of $\tau_2$ that has $v$ as a root, and $T_v(\tau_2)$ is the tree obtained by removing all branches attached to $v$ in $\tau_2$ and setting the decoration of $v$ to zero. The product $\curvearrowright_v$ means identifying the root of the left argument with the vertex $v$ in the right argument. This is a pre-Lie product and we can extend it to the symmetric algebra by \eqref{eq::GuinOudom1}-\eqref{eq::GuinOudom4}. This gives an action $S(\mathfrak{T}^-)\otimes \mathfrak{T} \to \mathfrak{T}$, where $\mathfrak{T}^-$ is the subspace of trees that has negative regularity. The renormalisation coaction is the dual of this action.

\subsection{Rough paths in a combinatorial Hopf algebra, and their translations}
\label{ssect:Translations}

We will in later sections link the concept of translations of rough paths with that of renormalisation of regularity structures. Because of this, we find it useful to recall some results following mainly the article \cite{Rahm2021RP}. The notion of translation of rough path was first considered in the 2017 work \cite{BrunedChevyrevFrizPreiss2017}, and further developed in the recent article \cite{BellingeriFrizPaychaPreiss2021}. Rough paths from the viewpoint of combinatorial Hopf algebras were considered in \cite{CurryEbrahimi-FardManchonMunthe-Kaas2018}, which is the setting we will be working in. See also \cite{TapiaZambotti2020}. The relevant definitions are as follows:

\begin{definition}
\label{def:combHopfAlg}
A combinatorial Hopf algebra $(\mathcal{H},\odot,\Delta,\eta,\epsilon)$ is a graded connected Hopf algebra $\mathcal{H} = \oplus_{n=0}^{\infty}\mathcal{H}_n$ over a field $\mathbb{K}$ of characteristic zero, together with a basis $\mathcal{B}=\cup_{n \geq 0} \mathcal{B}_n$ of homogeneous elements, such that:
\begin{enumerate}
\item There exists two positive constants $B$ and $C$ such that the dimension of $\mathcal{H}_n$ is bounded by $BC^n$.
\item The structure constants $c_{x y}^{z}$ and $c_{z}^{x y}$ of the product respectively the coproduct, defined for all elements $x,y,z \in \mathcal{B}$ by
\begin{align*}
x \odot y =& \sum_{z \in \mathcal{B}} c_{x y}^{z} z, \\
\Delta(z) =& \sum_{x,y \in \mathcal{B}} c_{z}^{x y} x \otimes y,
\end{align*}
are non-negative integers.
\end{enumerate}
We furthermore say that $\mathcal{H}$ is non-degenerate if $\mathcal{B} \cap \text{Prim}(\mathcal{H})=\mathcal{B}_1$.
\end{definition}

\begin{remark}
The above definition of a combinatorial Hopf algebras was given in \cite{CurryEbrahimi-FardManchonMunthe-Kaas2018}. and is useful for our specific setting. In general, there is no consensus on what an appropriate definition of a combinatorial Hopf algebra should be.
\end{remark}

\begin{definition} \label{def::RoughPath}
Let $\mathcal{H}=\oplus_{n\geq 0} \mathcal{H}_n$ be a commutative graded Hopf algebra with unit $1$, and let $\gamma \in (0,1]$. Suppose that $\mathcal{H}$ is endowed with a basis $\mathcal{B}$ making it combinatorial and non-degenerate in the sense of Definition \ref{def:combHopfAlg}. A $\gamma$-regular $\mathcal{H}$-rough path is a two-parameter family $\mathbb{X}=(\mathbb{X}_{st})_{s,t \in \mathbb{R}}$ of linear forms on $\mathcal{H}$ such that $\langle \mathbb{X}_{st},1\rangle =1$ and: 
\begin{enumerate}
\item For any $s,t \in \mathbb{R}$ and any $x,y \in \mathcal{H}$, the following identity holds
\begin{align*}
\langle \mathbb{X}_{st},x \odot y \rangle 
= \langle \mathbb{X}_{st},x \rangle\langle \mathbb{X}_{st}, y \rangle.
\end{align*}
\item For any $s,t,u \in \mathbb{R}$, Chen's identity holds
\allowdisplaybreaks
\begin{align*}
\mathbb{X}_{su} \ast \mathbb{X}_{ut}=\mathbb{X}_{st},
\end{align*}
where $\ast$ is the convolution product for linear forms on $\mathcal{H}$, defined in terms of the coproduct on $\mathcal{H}$.
\item For any $n \geq 0$ and any $x \in \mathcal{B}_n$, we have estimates
\begin{align*}
\sup_{s \neq t} \frac{|\langle \mathbb{X}_{st},x \rangle |}{|t-s|^{\gamma |x|}} < \infty,
\end{align*}
where $|x|=n$ denotes the degree of the element $x \in \mathcal{B}_n$.
\end{enumerate}
\end{definition}

A rough path in the BCK Hopf algebra is called a \textit{branched rough path}, these are used to solve rough differential equations
\begin{align}
	dY_t=f_0(Y_t)dt + \sum_{i=1}^n f_i(Y_t)dX_t^i, \quad Y_0=y_0 \label{eq::BranchedDiffEq}
\end{align}
on $\mathbb{R}^d$. If the driving path $X$ was smooth, one could solve the equation with Picard iteration
\begin{align*}
	Y_t=Y_0+\sum_{n=1}^{\infty} \sum_{1 \leq i_1 \leq \cdots \leq i_n\leq n} \,
	\big(\idotsint\limits_{s\leq t_1\leq \cdots\leq t_n\leq t} 
	f_{i_n}\cdots f_{i_1}dX_{t_1}^{i_1}\cdots dX_{t_n}^{i_n}\big).
\end{align*}
If the path $X$ is sufficiently rough, the above integrals are not defined. The main idea of rough path theory is to obtain a notion of solution by postulating the values of the first few integrals. The properties we require in Definition \ref{def::RoughPath} generalise the properties we expect a notion of integration to have. Indeed, we think of the linear forms $\mathbb{X}_{st}$ as encoding integrals, and we think of the parameters $s,t$ as the integration limits. In branched rough paths, the rooted tree structure of the elements in the BCK Hopf algebra reflects the structure of nested integrals. We, for example, have the following, where we think of the right side as being defined by the left side:
\begin{align*}
	\langle \mathbb{X}_{st}, \Forest{[3[1][2]]}\rangle =: \int_s^t X^{i_1}_{t_1}X^{i_2}_{t_2}dX^{i_3}_{t_1}.
\end{align*}
A rough path in the MKW Hopf algebra is called \textit{planarly branched rough path}. We will take a closer look at these in a later section. Note that, because the BCK Hopf algebra is a Hopf subalgebra of the MKW Hopf algebra, branched rough paths will be a special case of planarly branched rough paths.\\

Recall that the definition of a combinatorial Hopf algebra comes with a choice of a basis $\mathcal{B}$, this basis will also play a role in the concept of translations. Let $(\mathcal{H},\odot,\Delta,\eta,\epsilon)$ be a non-degenerate combinatorial Hopf algebra with basis $\mathcal{B}$ and $\mathcal{B}_1=\{e_0,\dots,e_{m-1}\}$. Let $(\mathcal{H}^{\ast},\ast,\Delta_{\odot})$ denote the graded dual space, with convolution product $\ast$ dual to $\Delta$ and coproduct $\Delta_{\odot}$ dual to $\odot$. Let $(\overline{\mathcal{H}^{\ast}},\ast,\Delta_{\odot})$ be the completion of $\mathcal{H}^{\ast}$ with respect to the grading, together with the continuously extended product $\ast$, and $\Delta_{\odot}$ extended continuously to a map $\overline{\mathcal{H}^{\ast}} \to \mathcal{H}^{\ast}\overline{\otimes} \mathcal{H}^{\ast}$. 

A translation is a family of maps $T_\textbf{v}$, indexed by the $m$ primitive elements $\textbf{v}=(v_0,\dots,v_{m-1} ) \in P(\overline{\mathcal{H}^{\ast}})^{\times m}$, such that $T_{\textbf{v}}(e_i)=e_i+v_i$ for all $e_i \in \mathcal{B}_1$ and such that some additional conditions are fulfilled. In the specific context of regularity structures, we will work with two special assumptions:
\begin{enumerate}
\item The basis element $e_0$ corresponds to time, meaning that our rough paths $\mathbb{X}_{st}$ will satisfy the identity:
\begin{align}
\langle \mathbb{X}_{st},e_0\rangle = t-s. \label{eq::ZeroIsTime}
\end{align}
\item We will only translate in the time direction, meaning that our index vectors $\textbf{v}$ will be of the form:
\begin{align}
\textbf{v}=(v_0,0,\dots,0). \label{eq::TranslateInTimeDirection}
\end{align}
\end{enumerate}

\begin{definition}
A rough path that satisfies the assumption \eqref{eq::ZeroIsTime} is called a time-augmented rough path, and a translation map $T_{\bf{v}}$ that satisfies the assumption \eqref{eq::TranslateInTimeDirection} is called a time-translation.
\end{definition}

We now give the full definition of a translation as defined in \cite{Rahm2021RP}.

\begin{definition} \label{def::Translation}
A family of algebra morphisms $T_{\bf{v}}: \overline{\mathcal{H}^{\ast}} \to \overline{\mathcal{H}^{\ast}}$ is a translation if
\begin{enumerate}
\item $T_{\bf{v}}(e_i)=e_i+v_i$ for every $e_i \in \mathcal{B}_1$ and some ${\bf{v}}=(v_0,\ldots,v_{m-1})$, $v_i \in \overline{\mathcal{H}^{\ast}}$ primitive.

\item $T_{\bf{v}} \circ T_{\bf{u}} = T_{\bf{v}} + T_{\bf{v}}({\bf{u}})$, where $T_{\bf{v}}({\bf{u}})=(T_{\bf{v}}(u_0), \ldots, T_{\bf{v}}(u_{m-1}))$.

\item For each $\mathcal{H}$-rough path $\mathbb{X}_{st}$, the pointwise translation $T_{\bf{v}}(\mathbb{X}_{st})$ is a $\mathcal{H}$-rough path:
\begin{enumerate}

\item $T_{\bf{v}}$ maps characters to characters.

\item $T_{\bf{v}}$ is a morphism with respect to the convolution product of $\overline{\mathcal{H}^*}$.

\item The bound \begin{align*}
\sup_{s \neq t} \frac{|\langle T_{\bf{v}}(\mathbb{X}_{st}),x \rangle |}{|t-s|^{\gamma |x|}} < \infty
\end{align*}
holds.
\end{enumerate}
\item \label{item:coaction} There exists a coaction $\rho_T : \mathcal{H} \to S(P(\mathcal{H}^{\ast}) \times \mathcal{B}_1)\otimes \mathcal{H}$ such that $\langle T_{\bf{v}}(\chi),x\rangle = \langle e^{\bf{v}} \otimes \chi,\rho_T(x)\rangle$.
\end{enumerate}
\end{definition}

\begin{remark}
	In the case of planarly branched rough paths, the coaction $\rho_T$ in point $4$ of the definition is the cotranslation coaction defined in equation \eqref{eq::CoTranslation}. For branched rough paths, the coaction is $\rho_T^{np}$.
\end{remark}

The space $S(P(\mathcal{H}^{\ast})\times \mathcal{B}_1)$ appears in point $4$ of the definition. We will not go into details on this space. Instead, we will remark that, when $T_{\textbf{v}}$ is a time-translation, the map $\rho_T$ can be restricted to a map $\rho_T^0: \mathcal{H}\to S(P(\mathcal{H}^{\ast})) \otimes \mathcal{H}$ such that $\langle T_{\textbf{v}}(\chi),x\rangle = \langle e^{\textbf{v}}\otimes \chi, \rho_T^0(x)\rangle$, where $e^{\textbf{v}} \in S(P(\mathcal{H}^{\ast}))$ is the character generated by $v_0 \in P(\mathcal{H}^{\ast})$. We call the map $\rho_T$ \textit{cotranslation} and the map $\rho_T^0$ \textit{time-cotranslation}. We will later need the following important property of (time-)cotranslation, called \textit{cointeraction}:

\begin{proposition}[\cite{Rahm2021RP}]
The coaction $\rho_T$ satisfies the identity
\begin{align*}
	m^{1,3}(\rho_T \otimes \rho_T)\Delta = (id \otimes \Delta)\rho_T,
\end{align*}
where
\begin{align*}
	m^{1,3}(a \otimes b \otimes c \otimes d)=ac\otimes b \otimes d.
\end{align*}
\end{proposition}

The cointeraction property also holds for $\rho_T^0$.

\subsection{Planarly branched rough paths}
\label{ssect:planarly}

Branched rough paths were generalized to planarly branched rough paths in \cite{CurryEbrahimi-FardManchonMunthe-Kaas2018}, to solve differential equations
\begin{align}
	dY_{st}=\sum_{i=1}^n f_i(Y_{st})dX_t^i, \quad Y_{ss}=y_0, \label{eq::PlanarlyBranchedEquation}
\end{align}
where $Y_{s \cdot}: \mathbb{R} \to \mathcal{M}$, $f_i : \mathcal{M} \to T\mathcal{M}$, $X_{\cdot}^i : \mathbb{R} \to \mathbb{R}$, for $\mathcal{M}$ a homogeneous space. Note, in the case of $\mathcal{M}$ being $\mathbb{R}^d$, we have branched rough paths. As a homogeneous space, $\mathcal{M}$ has a transitive action by Lie group $G$:
\begin{align*}
	G \times \mathcal{M} \to & \mathcal{M} \\
	(g,x) \mapsto & g.x
\end{align*}
Let $\mathfrak{g}$ be the Lie algebra of $G$. Then we can represent vector fields by Lie algebra valued maps:
\begin{align*}
f:& \mathcal{M} \to  \mathfrak{g},\\
\#f(x)=&\frac{d}{dt}|_{t=0}\exp(tf(x)).x \in T_x \mathcal{M},
\end{align*}
and equation \eqref{eq::PlanarlyBranchedEquation} can be rewritten in the form
\begin{align*}
dY_{st}=\sum_{i=1}^d \#f_i(Y_{st})dX_t^i, \quad Y_{ss}=y_0,
\end{align*}
where $f_i$ are now Lie algebra valued. The space $C^{\infty}(\mathcal{M},\mathfrak{g})$ can then be endowed with the product:
\begin{align*}
(f \graft g)(x)=\frac{d}{dt}|_{t=0}g(\exp(tf(x)).x),
\end{align*}
and the pointwise Lie bracket:
\begin{align*}
[f,g](x)=[f(x),g(x)].
\end{align*}
These two products turn $C^{\infty}(\mathcal{M},\mathfrak{g})$ into a post-Lie algebra, i.e., the axioms \eqref{eq::PostLieAssociator} and \eqref{eq::PostLieDerivation} are satisfied. We now extend the post-Lie algebra over $C^{\infty}(\mathcal{M},\mathfrak{g})$, to a D-algebra over $C^{\infty}(\mathcal{M},U(\mathfrak{g}))$, in order to describe higher order differential operators. Let $V \in \mathfrak{g}$ and $f \in C^{\infty}(\mathcal{M},U(\mathfrak{g}))$, define the Lie derivative:
\begin{align*}
	V[f](x)=\frac{d}{dt}|_{t=0} f(\exp(tV).x),
\end{align*}
and extend it to $U(\mathfrak{g})$ by
\begin{align*}
	VW[f]=V[W[f]],
\end{align*}
where $V \in \mathfrak{g}$ and $W \in U(\mathfrak{g})$. Let the associative product on $C^{\infty}(\mathcal{M},U(\mathfrak{g}))$ be defined pointwise:
\begin{align*}
	(fg)(x)=f(x)g(x).
\end{align*}
For $f,g \in C^{\infty}(\mathcal{M},U(\mathfrak{g}))$, we define
\begin{align*}
	(f \graft g)(x)=(f(x)[g])(x).
\end{align*}
Then $( C^{\infty}(\mathcal{M},U(\mathfrak{g})),\cdot,\graft  )$ is a D-algebra \cite{MuntheKaasWright2008}. Because $\mathcal{OF}_{\mathbf{f}}$ is the free D-algebra, there exists a D-algebra morphism $F_{\textbf{f}} : \mathcal{OF}_{\textbf{f}} \to C^{\infty}(\mathcal{M},U(\mathfrak{g}))$ generated by mapping a single-vertex tree decorated by $i$ to $f_i \in C^{\infty}(\mathcal{M},\mathfrak{g})$. We can now define the second associative product $\ast$ on $C^{\infty}(\mathcal{M},U(\mathfrak{g}))$ by
\begin{align*}
	f \ast g=f_{(1)}(f_{(2)} \graft g).
\end{align*}
This product describes the composition of differential operators:
\begin{align*}
	\# (f \ast g )= \# f \circ \# g.
\end{align*}
The solution $Y_{st}$ to equation \eqref{eq::PlanarlyBranchedEquation} can now be represented by the element $\mathbb{Y}_{st} \in C^{\infty}(\mathcal{M},U(\mathfrak{g}))$ satisfying
\begin{align*}
	\Psi(Y_{st})=(\mathbb{Y}_{st} \graft \Psi)(y_0),
\end{align*}
for all $\Psi \in C^{\infty}(\mathcal{M},U(\mathfrak{g}))$. Using this representation, we can define what we mean by a formal solution to equation \eqref{eq::PlanarlyBranchedEquation}.

\begin{definition}
A formal solution to the equation \eqref{eq::PlanarlyBranchedEquation} is given by
\begin{align*}
	\mathbb{Y}_{st}=\sum_{x \in \mathbf{OF}_{ \{1,\dots,n\} }} \langle \mathbb{X}_{st}, x \rangle F_{\textbf{f}}(x),
\end{align*}
where $\mathbb{X}_{st}$ is a rough path in $\mathcal{H}_{MKW}^{ \{1,\dots,n\} }$ such that
\begin{align*}
	\langle \mathbb{X}_{st}, \Forest{[i]} \rangle = X_t^i - X_s^i.
\end{align*}
\end{definition}

We will now describe translations of planarly branched rough paths. The primitive elements in $\overline{\mathcal{H}_{MKW}^{ \{1,\dots,n\} }}$ form the free post-Lie algebra $( Lie(\mathcal{PT}_{ \{1,\dots,n\} }),\graft,[\cdot,\cdot] )$. Let $\textbf{v}=(v_1,\dots,v_n)$ be $n$ primitive elements and define
\begin{align*}
	T_{\textbf{v}}(\Forest{[i]})=&\Forest{[i]}+v_i, \\
	T_{\textbf{v}}(x \graft y)=&T_{\textbf{v}}(x)\graft T_{\textbf{v}}(y),\\
	T_{\textbf{v}}([x,y])=&[ T_{\textbf{v}}(x), T_{\textbf{v}}(y) ],\\
	T_{\textbf{v}}(a \ast b)=&T_{\textbf{v}}(a) \ast T_{\textbf{v}}(b),
\end{align*}
for $x,y \in Lie(\mathcal{PT}_{ \{1,\dots,n\} })$ and $a,b \in \mathcal{OF}_{ \{1,\dots,n\} }$. Then $T_{\textbf{v}}$ is a translation for planarly branched rough paths and the dual coaction is given by $\rho_T$ defined in equation \eqref{eq::CoTranslation}.

\begin{proposition}[\cite{Rahm2021RP}]
	Let $T_{\bf{v}}$ be translation of planarly branched rough paths. $\mathbb{Y}_{st}$ is a formal solution to
	\begin{align*}
		d\mathbb{Y}_{st}=\sum_{i=1}^n f_i(\mathbb{Y}_{st})d(T_{\bf{v}}(\mathbb{X}_t^i))
	\end{align*}
if and only if $\mathbb{Y}_{st}$ is a formal solution to
\begin{align*}
	d\mathbb{Y}_{st}=\sum_{i=1}^n F_{\textbf{f}}(v_i)(\mathbb{Y}_{st})d\mathbb{X}_t^i.
\end{align*}
\end{proposition}

\subsection{Regularity Structures}
\label{ssect:regstruc}

Regularity structures were introduced by M.~Hairer \cite{Hairer2013} to solve rough partial differential equations:
\begin{align}
	u=K \ast F(u, \xi), \label{eq::DiffEqConvolutionForm}
\end{align}
for unknown $u: \mathbb{R}^d \to \mathbb{R}$, $K$ a convolution kernel, $\xi$ some noise and $F$ a non-linearity. Solving the equation using regularity structures requires three main steps:
\begin{enumerate}
	\item First, one constructs a \textit{regularity structure}; a vector space that will allow us to extend the notion of a Taylor expansion, by also including terms involving the noise $\xi$ and in general other functions, rather than only polynomials. This algebraic construction should be sufficiently rich to allow us to express equation \eqref{eq::DiffEqConvolutionForm}.
	\item Formally solve the algebraic equation formulated in step $1$.
	\item Give meaning to the formal solution obtained in step $2$. This is typically done by renormalisation, and requires postulating the value of finitely many ill-defined products. Note that this is similar to postulating the values of integrals as done for rough paths.
\end{enumerate}
The focus of this paper will mainly be on steps $1$ and $3$. It was already shown in Section 4.3 of \cite{Hairer2013} that one can construct a regularity structure from any connected, graded, commutative Hopf algebra. In \cite{BrunedHairerZambotti2019}, the authors constructed a Connes--Kreimer-like Hopf algebra over non-planar rooted trees with decorated vertices and decorated edges, whose regularity structure is useful to express many concrete differential equations. They furthermore obtained a renormalisation procedure for step $3$, by constructing a coaction on this Hopf algebra.  Recall that we constructed the Hopf algebra and the coaction in Section \ref{ssect:trees}. We now give the definition of a regularity structure.

\begin{definition} \label{def::RegularityStructure}
A regularity structure $\mathcal{R}=(A,T,G)$ consists of
\begin{itemize}
\item An index set $A \subset \mathbb{R}$ that is bounded from below and locally finite.
\item A graded vector space $T=\oplus_{\alpha \in A}T_{\alpha}$, with each $T_{\alpha}$ a Banach space. Furthermore, $T_0 \simeq \mathbb{R}$, and we denote the unit vector by $1$.
\item A group $G$ of linear operators acting on $T$, such that for every $\Gamma \in G,$ every $\alpha \in A$ and every $a \in T_{\alpha}$, one has
\begin{align*}
\Gamma a - a \in \oplus_{\beta < \alpha} T_{\beta}.
\end{align*}
\end{itemize}
\end{definition}

We think of the index set $A$ as being the set regularities of the functions/distributions we want to build, and the vector space $T$ as formal Taylor series. We will soon define a \textit{model} for a regularity structure. If an element in $T$ is a formal Taylor series, then a model maps the formal Taylor series to a concrete expansion around a given point. The group $G$ will describe how to re-expand around a different point, often called re-centering or positive renormalisation.

\begin{definition} \label{def::model}
A model for a given regularity structure $\mathcal{R}=(A,T,G)$ on $\mathbb{R}^d$ consists of:
\begin{itemize}
\item A map $\Gamma: \mathbb{R}^d \times \mathbb{R}^d \to G$ such that $\Gamma_{xx}=Id$ and such that $\Gamma_{xy}\Gamma_{yz}=\Gamma_{xz}$ for all $x,y,z \in \mathbb{R}^d$.
\item A collection of continuous linear maps $\Pi_x: T \to S'(\mathbb{R}^d)$ such that $\Pi_y = \Pi_x \Gamma_{xy}$ for every $x,y \in \mathbb{R}^d$.
\end{itemize}

\end{definition}
One typically also requires some analytical bounds on the model, see \cite{Hairer2013} for details. For the purpose of this paper, we will not be considering this.

\subsection{Branched rough paths as a special case of regularity structures}
\label{ssect:roughregstr}

In the language of Section \ref{ssect:Translations}, branched rough paths are rough paths in the BCK Hopf algebra $\mathcal{H}_{BCK}$, which was described in Section \ref{ssect:trees}. Recall from Section \ref{ssect:Translations} that branched rough paths are used to solve differential equations of the form
\begin{align}
dY_t=f_0(Y_t)dt + \sum_{i=1}^n f_i(Y_t)dX_t^i, \quad Y_0=y_0.
\end{align}
Letting $K$ denote the Heaviside kernel, this equation can be written in the form \eqref{eq::DiffEqConvolutionForm} as
\begin{align*}
Y_t=Y_0+K\ast \sum_{i=1}^n f_i(Y_t)X_t^i,
\end{align*}
hence the construction from \cite{BrunedHairerZambotti2019} can be applied to the problem. We will in this subsection outline how, in reference \cite{BrunedChevyrevFrizPreiss2017}, a special case of the construction in the work \cite{BrunedHairerZambotti2019} was recovered from the rough path setting.

\smallskip

Let $\mathcal{B}$ denote the space of non-planar rooted trees whose edges are decorated by elements from the set $\{0,\dots,n\}$, where we will use undecorated edges to represent the $0$ decoration, and where edges decorared by elements from $\{1,\dots,n\}$ are adjacent to leaves. Then there is an isomorphism $\phi$ between $\mathcal{T}_{ \{0,\dots,n\} }$ and $\mathcal{B}$. We will draw trees in $\mathcal{B}$ in blue, to distinguish them from trees in $\mathcal{T}_{\{0,\dots,n\}}$. The isomorphism is given by replacing every non-zero decorated vertex in $\mathcal{T}_{\{0,\dots,n\}}$ with two vertices connected by an edge of the same decoration as the decoration of the original vertex.

\begin{example} The isomorphism between $\mathcal{T}_{\{0,\dots,n\}}$ and $\mathcal{B}$.
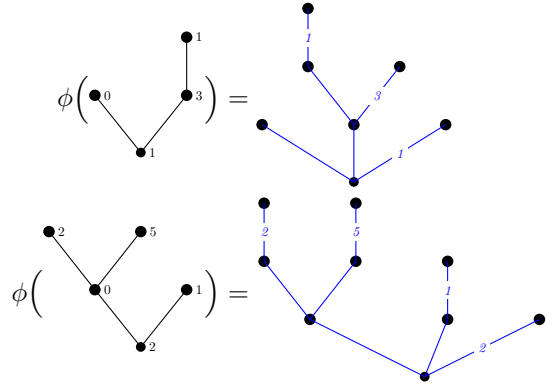
\begin{align*}
\phi\Big(\Forest{[1[0][3[1]]]}\Big) 
&= {\color{blue} \Forest{[[][[[,edge label = {node[midway,fill=white,scale=0.5]{1}}]][,edge label = {node[midway,fill=white,scale=0.5]{3}}]][,edge label = {node[midway,fill=white,scale=0.5]{1}}]]}} \\
\phi\Big(\Forest{[2[0[2][5]][1]]}\Big)
&= {\color{blue} \Forest{[[[[,edge label = {node[midway,fill=white,scale=0.5]{2}}]][[,edge label = {node[midway,fill=white,scale=0.5]{5}}]]] [[,edge label = {node[midway,fill=white,scale=0.5]{1}}]] [,edge label = {node[midway,fill=white,scale=0.5]{2}}]]}  }
\end{align*}
\end{example}

\medskip

Let $\alpha_i$ denote the regularity of the path $X^i$ from equation \eqref{eq::BranchedDiffEq} and define the index set $A=\{0\}\cup_i \alpha_i \mathbb{N} \cup_i (\alpha_i \mathbb{N}-1)$. We associate a degree in $A$ to each tree $\tau \in\mathcal{B}$, as the sum over all edges in $\tau$ where an edge decorated by $i$ contributes $\alpha_i-1$ and undecorated edges contribute $1$.

Recall the root-adding map $B_+$ that sends forests to trees, and the root-removing map $B_-$ that sends  trees to forests. Define on $\mathcal{B}$ the tree product $\tau_1\tau_2:=B_+(B_-(\tau_1)B_-(\tau_2))$, i.e., by identifying $\tau_1\tau_2$ with the tree obtained by joining the roots of $\tau_1$ and $\tau_2$. We write $\mathcal{B}_-$ for the free commutative associative algebra generated by trees of negative degree, where we interpret the product as the forest product. Write $\mathcal{B}_+$ for the vector space spanned by trees of positive degree where no decorated edge is adjacent to the root, together with the tree product.

Using the isomorphism $\phi$ between $\mathcal{T}_{ \{0,\dots,n\} }$ and $\mathcal{B}$, any tree $\tau \in \mathcal{B}_+$ can be written as $\tau= B_+(\phi(\omega))$, where $\omega \in \mathcal{F}$. Furthermore, if $\tau$ is an arbitrary tree in $\mathcal{B}$, then $B_+(\tau)\in \mathcal{B}_+$ and can be written in this form. We use this fact to construct a regularity structure in the following theorem.

\begin{theorem}[\cite{BrunedChevyrevFrizPreiss2017}] \label{thm::ModelFromRoughPath}
Let $\mathbb{X}_{st}$ be a branched rough path over the (time-augmented) path $(t,X_t^1,\dots,X_t^m)$ that satisfies
\begin{align*}
\langle \mathbb{X}_{st},\tau \graft \bullet \rangle = \int_s^t \langle \mathbb{X}_{su},\tau \rangle du,
\end{align*}
for $\tau \in \mathcal{T}$. Define a map $\Pi: \mathbb{R}^2 \times \mathcal{B} \to \mathbb{R}$ by
\begin{align*}
\Pi_s( B_+ (\phi(\tau)) )(t)=&\langle \mathbb{X}_{st},\tau\rangle,\\
\Pi_s( \tau )(t)=&\frac{d}{dt}\Pi_s(B_+ (\tau) )(t).
\end{align*}
Define a map $\Gamma : \mathbb{R}^2 \times \mathcal{B}^{\ast} \to \mathcal{B}^{\ast}$ by
\begin{align*}
\Gamma_{st}(B_+(\phi(\tau))) =& (\mathbb{X}_{st} \otimes B_+ \circ \phi)\Delta_{CK}(\tau),\\
\Gamma_{st}(\tau) =& (\mathbb{X}_{st} \otimes \phi)(\Delta_{CK}(\phi^{-1}(\tau)) - \phi^{-1}(\tau) \otimes 1 ).
\end{align*}
Then $(A,\mathcal{B},G)$ is a regularity structure and $(\Pi,\Gamma)$ is a model over the regularity structure, where $G$ is the group generated by $\Gamma_g = (g\otimes id)\Delta^+$for $g$ in $\mathcal{B}_+^{\ast}$.
\end{theorem}

Both in defining $\Pi$ and defining $\Gamma$, we need two equalities. One equality to define the maps on trees with undecorated root, and a second equality to define what happens for trees with decorated root. Note that the second equalities also makes sense for trees where the root is not decorated. It is shown in \cite{BrunedChevyrevFrizPreiss2017} that the two equalities are consistent with each other in these cases. This construction of $\Gamma$ furthermore agrees with the more general construction:
\begin{align*}
\Gamma_{st}=(\gamma_{ts} \otimes id)\Delta^+,
\end{align*}
where $\Delta^+$ is the Connes--Kreimer-like coproduct used for regularity structures, and $\gamma_{st}=B_+(\phi(\mathbb{X}_{st}))$. We sketched the construction of the dual product $\ast_+$ to $\Delta_+$ at the end of \ref{ssect:trees}.

We will now describe how the renormalisation coaction can be obtained from the cotranslation coaction. Let $\ell \in \mathcal{B}_-^{\ast}$ be a character and denote by $M_{\ell}$ renormalisation by $\ell$, in the sense of regularity structures, this can be described by a coaction $\Delta^-$:
\begin{align*}
M_{\ell}(\tau)=(\ell \otimes id)\Delta^-(\tau).
\end{align*}
For the general case where regularity structure trees have decorated edges and vertices, we sketched the construction of the dual product to $\Delta^-$ at the end of Section \ref{ssect:trees}. One uses the renormalisation map to obtain a renormalised model by:
\begin{align*}
\Pi_s^{M_{\ell}}=&\Pi_s M_{\ell},\\
\Gamma_{st}{M_\ell}=&(\gamma_{st}M_{\ell} \otimes id)\Delta^+.
\end{align*}

Let $\rho_T^{np}: \mathcal{T}_{ \{0,\dots,n\} } \to \mathcal{F}_{\{0,\dots,n\}} \otimes \mathcal{T}_{\{0,\dots,n\}}$ be the time-cotranslation coaction for branched rough paths. Let $\pi^-$ be projection onto trees of negative degree. Denote $\delta^-=(\pi^- \otimes id)\rho_T^{np}$. Then:
\begin{align*}
\Delta^-(\phi(\tau))=&(\phi \otimes \phi)\delta^-(\tau),\\
M_{\ell}(\phi(\tau))=&\phi(T^{\ast}_{\log(\ell)}(\tau)),
\end{align*}
where $T^{\ast}$ is dual to translation, and $\log(\ell)$ is the primitive element that generates the character $\ell$.

\section{Planarly branched rough paths as regularity structures} 
\label{sec::PlanarRoughPathsAreRegularityStructures}

In this section we will generalize the construction from \cite{BrunedChevyrevFrizPreiss2017}, that saw branched rough paths as a special case of regularity structures, and apply it to planarly branched rough paths.\\ 

Let $Lie(\mathcal{PT})$ be the Lie algebra over the vector space of planar rooted trees with vertices decorated by $\{0,\dots,n\}$. Let $\mathcal{PB}$ denote the vector space of planar rooted trees whose edges are decorated by the set $\{0,\dots,n\}$, where we will use undecorated edges to represent the $0$ decoration, and where edges decorated by $\{1,\dots,n\}$ are adjacent to leaves. Let $Lie(\mathcal{PB})$ denote the Lie algebra over $\mathcal{PB}$. We will draw trees in $Lie(\mathcal{PB})$ in blue, to differentiate them from trees in $Lie(\mathcal{PT})$. There is an isomorphism $\varphi$ between $Lie(\mathcal{PT})$ and $Lie(\mathcal{PB})$ given by replacing every non-zero vertex in $Lie(\mathcal{PT})$ with $2$ vertices connected by a rightmost edge of the same decoration as the original vertex.

\begin{example} The isomorphism between $Lie(\mathcal{PT})$ and $Lie(\mathcal{PB})$.
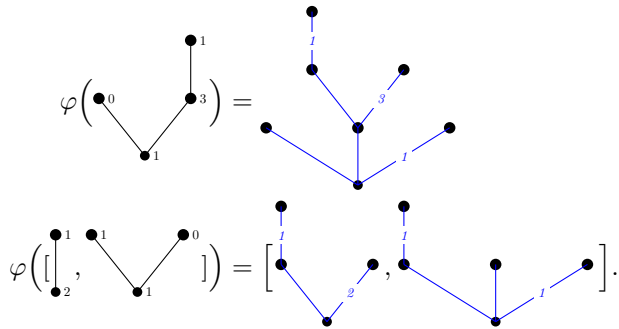
\begin{align*}
\varphi\Big(\Forest{[1[0][3[1]]]}\Big) 
&= {\color{blue} \Forest{[[][[[,edge label = {node[midway,fill=white,scale=0.5]{1}}]][,edge label = {node[midway,fill=white,scale=0.5]{3}}]][,edge label = {node[midway,fill=white,scale=0.5]{1}}]]}} \\
\varphi\Big( [\Forest{[2[1]]},\Forest{[1[1][0]]} ]\Big)
&= \Big[{\color{blue}\Forest{[[[,edge label = {node[midway,fill=white,scale=0.5]{1}}]][,edge label = {node[midway,fill=white,scale=0.5]{2}}]]}   },{\color{blue} \Forest{[[[,edge label = {node[midway,fill=white,scale=0.5]{1}}]][][,edge label = {node[midway,fill=white,scale=0.5]{1}}]]} } \Big].
\end{align*}
\end{example}

Define a map $B_+ : Lie(\mathcal{PB})\to \mathcal{PB}$ by $B_+(\tau)=\tau \graft {\color{blue} \Forest{[]} }$, where $\graft$ is left grafting. Let $\mathcal{OF}$ denote the planar forests with vertices decorated by $\{0,\dots,n\}$ and let $\mathcal{OB}$ denote the planar forests with edges decorated by $\{0,\dots,n\}$. Extend $\varphi: \mathcal{OF} \to \mathcal{OB}$ to be a shuffle morphism. Define the tree product on $\mathcal{PB}$ by $\tau_1\tau_2$ being the tree obtained by joining the roots of $\tau_1$ and $\tau_2$ and shuffling the branches. Extend $B_+: \mathcal{OB} \to \mathcal{PB}$ by $B_+(\tau_1 \shuffle \cdots \shuffle \tau_n)=B_+(\tau_1)\cdots B_+(\tau_n)$, where the product on the right side is the tree product. This amounts to the natural extension of grafting everything on a new root, and arranging the planar order of the branches according to the order of the roots in the forest.\\

\begin{example}
The extended maps $B_+$ and $\varphi$.
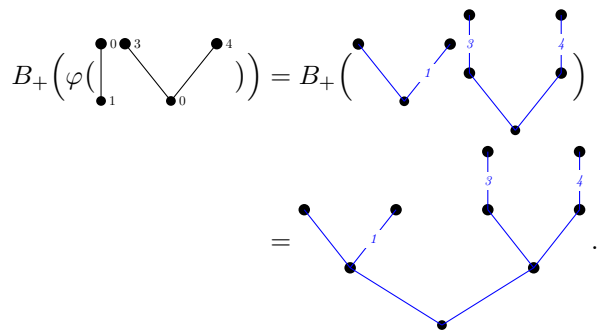
\begin{align*}
B_+\Big(\varphi\big( \Forest{[1[0]]}\Forest{[0 [3] [4]]} \big)\Big)
&=B_+\Big({\color{blue} \Forest{[[][,edge label = {node[midway,fill=white,scale=0.5]{1}}]]}\Forest{[ [[,edge label = {node[midway,fill=white,scale=0.5]{3}}]] [[,,edge label = {node[midway,fill=white,scale=0.5]{4}}]]]}}\Big)\\
&={\color{blue}\Forest{[[[][,edge label = {node[midway,fill=white,scale=0.5]{1}}]]  [ [[,edge label = {node[midway,fill=white,scale=0.5]{3}}]] [[,,edge label = {node[midway,fill=white,scale=0.5]{4}}]]]   ]}}.
\end{align*}
\end{example}

Let $\alpha_i$ denote the regularity of the path $X$ from equation \eqref{eq::PlanarlyBranchedEquation} and define the index set $A=\{0\}\cup_i \alpha_i \mathbb{N} \cup_i (\alpha_i \mathbb{N}-1)$. We associate a degree in $A$ to each tree $\tau \in Lie(\mathcal{PB})$, as the sum over all edges in $\tau$ where decorated edges contribute $\alpha_i-1$ and undecorated edges contribute $1$. We write $S(\mathcal{PB})_-$ for the free commutative associative algebra generated by trees of negative degree. Write $\mathcal{PB}_+$ for the vector space spanned by trees of positive degree, together with the tree product. Denote by $\mathcal{G}_+$ the group of characters over $\mathcal{PB}_+$.

\begin{remark}
The space $S(\mathcal{PB})_-$ does not contain any Lie brackets. One should think of this in terms of negative trees representing distributions and positive trees representing vector fields. The fact that Lie brackets appear for positive trees reflects the existence of the Jacobi--Lie bracket of vector fields, which does not exist for distributions.\\ 
In the dictionary between rough path trees and regularity structure trees, an extra root is added via $B_+$ to the regularity structure trees compared to the rough path forests. The "trees" in $\mathcal{PB}_+$ are the trees with only one edge adjacent to the root and the "Lie polynomials" are those trees where the edges adjacent to the root form a Lie polynomial.
\end{remark}

Define a coaction $\Delta^+: \mathcal{PB} \to \mathcal{PB}_+ \otimes \mathcal{PB}$ by
\begin{align*}
\Delta^+( B_+(\varphi(\omega)) )=(B_+\circ \varphi \otimes B_+ \circ \varphi  )\Delta_{MKW}(\omega),
\end{align*}
for $\omega \in \mathcal{OF}$. This defines $\Delta^+$ for all input trees that have no decorated edge adjacent to the root. For input trees $\omega \in \mathcal{PB}$ that have a decorated edge adjacent to the root, define:
\begin{align*}
	\Delta^+(\omega)=(id \otimes B_-)(\Delta^+(B_+(\omega))-B_+(\omega) \otimes {\color{blue}\Forest{[]}}).
\end{align*}
One can describe this coaction with left admissible cuts, where the cut edges are undecorated. The pruned parts are shuffled and then grafted onto a common root.

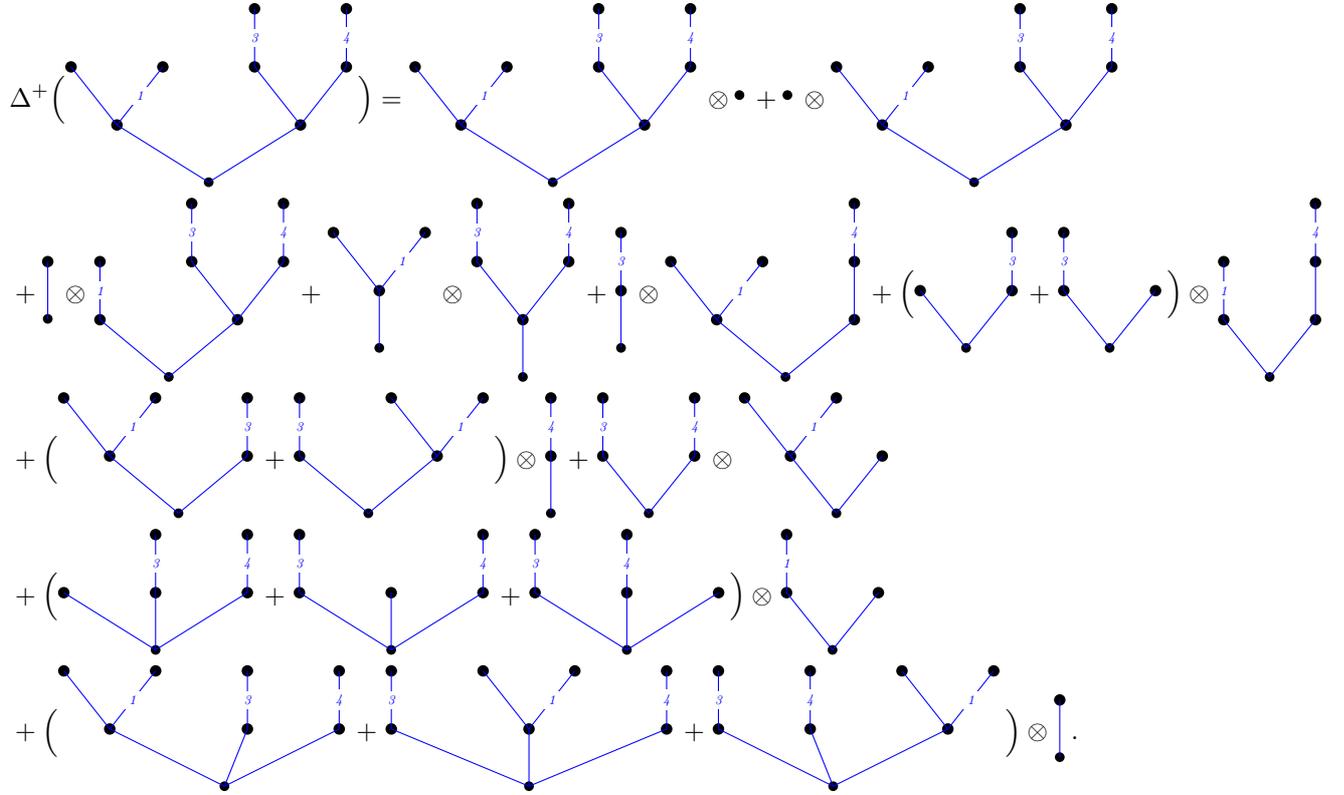
\begin{example}
	\begin{align*}
		&\Delta^+\Big({\color{blue}\Forest{[[[][,edge label = {node[midway,fill=white,scale=0.5]{1}}]]  [ [[,edge label = {node[midway,fill=white,scale=0.5]{3}}]] [[,,edge label = {node[midway,fill=white,scale=0.5]{4}}]]]   ]}}\Big) 
		={\color{blue}\Forest{[[[][,edge label = {node[midway,fill=white,scale=0.5]{1}}]]  [ [[,edge label = {node[midway,fill=white,scale=0.5]{3}}]] [[,,edge label = {node[midway,fill=white,scale=0.5]{4}}]]]   ]}} \otimes {\color{blue}\Forest{[]}} + {\color{blue}\Forest{[]}} \otimes {\color{blue}\Forest{[[[][,edge label = {node[midway,fill=white,scale=0.5]{1}}]]  [ [[,edge label = {node[midway,fill=white,scale=0.5]{3}}]] [[,,edge label = {node[midway,fill=white,scale=0.5]{4}}]]]   ]}}\\
		&+ {\color{blue} \Forest{[[]]} } \otimes {\color{blue}\Forest{[[[,edge label = {node[midway,fill=white,scale=0.5]{1}}]]  [ [[,edge label = {node[midway,fill=white,scale=0.5]{3}}]] [[,,edge label = {node[midway,fill=white,scale=0.5]{4}}]]]   ]}} 
		+ {\color{blue} \Forest{[[[][,edge label = {node[midway,fill=white,scale=0.5]{1} }]]]} }\otimes {\color{blue}\Forest{[[ [[,edge label = {node[midway,fill=white,scale=0.5]{3}}]] [[,,edge label = {node[midway,fill=white,scale=0.5]{4}}]]]   ]}}
		+ {\color{blue} \Forest{[[[,edge label = {node[midway,fill=white,scale=0.5]{3}}]]]} } \otimes {\color{blue}\Forest{[[[][,edge label = {node[midway,fill=white,scale=0.5]{1}}]]  [[[,edge label = {node[midway,fill=white,scale=0.5]{4}}]]]   ]}}
		+\Big({\color{blue} \Forest{[ [] [[,edge label = {node[midway,fill=white,scale=0.5]{3}}]] ]}} 
		+ {\color{blue}\Forest{[ [[,edge label = {node[midway,fill=white,scale=0.5]{3}}]][] ]} }\Big)\otimes {\color{blue}\Forest{[[[,edge label = {node[midway,fill=white,scale=0.5]{1}}]]  [ [[,,edge label = {node[midway,fill=white,scale=0.5]{4}}]]]   ]}}\\
		&+\Big({\color{blue} \Forest{[ [[][,edge label={node[midway,fill=white,scale=0.5]{1}}]] [[,edge label = {node[midway,fill=white,scale=0.5]{3}}]] ]} } + {\color{blue} \Forest{ [ [[,edge label = {node[midway,fill=white,scale=0.5]{3}}]] [[][,edge label = {node[midway,fill=white,scale=0.5]{1}}]] ] } } \Big) \otimes {\color{blue}\Forest{[[[,,edge label = {node[midway,fill=white,scale=0.5]{4}}]]]   ]}} + {\color{blue} \Forest{[[[,edge label = {node[midway,fill=white,scale=0.5]{3}}]][[,edge label = {node[midway,fill=white,scale=0.5]{4}}]]]} } \otimes {\color{blue}\Forest{[[[][,edge label = {node[midway,fill=white,scale=0.5]{1}}]]  []   ]}}\\
		&+\Big({\color{blue} \Forest{ [ [] [[,edge label = {node[midway,fill=white,scale=0.5]{3}}]][[,edge label = {node[midway,fill=white,scale=0.5]{4}}]]  ] } } + {\color{blue} \Forest{[ [[,edge label = {node[midway,fill=white,scale=0.5]{3}}]] [] [[,edge label = {node[midway,fill=white,scale=0.5]{4}}]] ]} } + {\color{blue} \Forest{[[[,edge label = {node[midway,fill=white,scale=0.5]{3}}]][[,edge label = {node[midway,fill=white,scale=0.5]{4}}]][]]} } \Big)\otimes {\color{blue}\Forest{[[[,edge label = {node[midway,fill=white,scale=0.5]{1}}]]  []   ]}}\\
		&+\Big({\color{blue} \Forest{[ [[][,edge label = {node[midway,fill=white,scale=0.5]{1}}]] [[,edge label = {node[midway,fill=white,scale=0.5]{3}}]][[,edge label = {node[midway,fill=white,scale=0.5]{4}}]]]} }+{\color{blue} \Forest{[[[,edge label = {node[midway,fill=white,scale=0.5]{3}}]] [[][,edge label={node[midway,fill=white,scale=0.5]{1}}]] [[,edge label = {node[midway,fill=white,scale=0.5]{4}}]]]} } + {\color{blue} \Forest{[[[,edge label = {node[midway,fill=white,scale=0.5]{3}}]][[,edge label = {node[midway,fill=white,scale=0.5]{4}}]] [[][,edge label = {node[midway,fill=white,scale=0.5]{1}}]] ]} }\Big) \otimes {\color{blue} \Forest{[[]]} }.
	\end{align*}
\end{example}

Let $g \in \mathcal{G}_+$, the group of characters over $\mathcal{PB}_+$, and define
\begin{align*}
\Gamma_g = (g \otimes id)\Delta^+.
\end{align*}
Denote $\mathcal{G}=\{\Gamma_g : g \in \mathcal{G}_+  \}$. Note that $\mathcal{G}$ is a group with product $\Gamma_g\Gamma_h=\Gamma_{h \circ g}$, where $\circ$ is the product in $\mathcal{G}_+$, that is dual to $\Delta^+$.

\begin{lemma}
The triple $(A,\mathcal{PB},\mathcal{G})$ is a regularity structure.
\end{lemma}

We are now ready to define maps $\Pi: \mathbb{R}^2 \times \mathcal{PB} \to \mathbb{R}$ and $\Gamma: \mathbb{R}^2 \times \mathcal{PB}^{\ast} \to \mathcal{PB}^{\ast}$. Let $\mathbb{X}_{st}$ be a planarly branched rough path over the time-augmented path $(t,X_t^1,\dots,X_t^m)$. Let $\omega \in \mathcal{OF}$, define $\Pi$ by
\begin{align*}
\Pi_s(B_+(\varphi(\omega)))(t)=&\langle \mathbb{X}_{st},\omega\rangle.
\end{align*}
This defines $\Pi$ for all input trees where no decorated edge is adjacent to the root. Let $\omega \in \mathcal{PB}$ have a decorated edge adjacent to the root and define
\begin{align*}
	\Pi_s(\omega)(t)=\frac{d}{dt}\Pi_s(B_+(\omega))(t).
\end{align*}

\begin{proposition}
	If the time-augmented planarly rough path $\mathbb{X}_{st}$ satisfies 
	\begin{align}
		\langle \mathbb{X}_{st},\omega \graft \Forest{[0]} \rangle = \int_s^t \langle \mathbb{X}_{su},\omega \rangle du, \label{eq::EdgesAreIntegration}
	\end{align}
	then the model satisfies
	\begin{align*}
		\frac{d}{dt}\Pi_x( B_+(\varphi(\omega))(t)=\Pi_x(\varphi(\omega))(t),
	\end{align*}
	for all $\omega \in \mathcal{OF}$.
\end{proposition}

\begin{proof}
	Suppose first that there exists an $\omega'$ such that $\varphi(\omega)=B_+(\varphi(\omega'))$, then:
	\begin{align*}
		\frac{d}{dt}\Pi_x( B_+(\varphi(\omega))(t)=&\frac{d}{dt}\Pi_x\bigl( B_+(B_+(\varphi(\omega'))) \bigr)\\
		=&\frac{d}{dt}\Pi_x\bigl(B_+(\varphi(B_+(\omega')   )  )  \bigr)\\
		=&\frac{d}{dt}\langle \mathbb{X}_{xt},\omega' \graft \bullet \rangle \\
		=&\frac{d}{dt}\int_x^t \langle \mathbb{X}_{xu},\omega' \rangle du \\
		=&\langle \mathbb{X}_{xt},\omega' \rangle \\
		=&\Pi_x\bigl(B_+(\varphi(\omega') )  \bigr)(t)\\
		=&\Pi_x\bigl(\varphi(B_+(\omega'))  \bigr)(t)\\
		=&\Pi_x\bigl(\varphi(\omega)\bigr)(t).
	\end{align*}
	Suppose that no such $\omega'$ exists, then the proposition is by definition.
\end{proof}

Note that the integral in \eqref{eq::EdgesAreIntegration} is not a rough integral, as it is a standard integral against the time-component of the path rather than any of the rough components. Now define $\Gamma$ by
\begin{align*}
\Gamma_{st}=(\gamma_{ts} \otimes id)\Delta^+,
\end{align*}
where
\begin{align*}
\gamma_{st}=&B_+(\varphi(\mathbb{X}_{st})).
\end{align*}
These definitions of $(\Pi,\Gamma)$ generalise those from Theorem \ref{thm::ModelFromRoughPath}. 

\begin{lemma}\label{lemma::ModelOnIdentityIsConstant}
	If the rough path $\mathbb{X}_{st}$ is time-augmented, i.e. it satisfies
\begin{align*}
	\langle \mathbb{X}_{st},\Forest{[0]}\rangle = t-s,
\end{align*}
then:
\begin{align*}
	\Pi_x({\color{blue}\Forest{[]}})(t)=1.
\end{align*}
\end{lemma}

\begin{proof}
	\begin{align*}
		\Pi_x({\color{blue}\Forest{[]}})(t)=&\frac{d}{dt}\Pi_x({\color{blue} \Forest{[[]]} })(t)\\
		=&\frac{d}{dt}\langle \mathbb{X}_{xt},\Forest{[0]}\rangle \\
		=&\frac{d}{dt}(t-s) \\
		=1.
	\end{align*}
\end{proof}

\begin{theorem}
If the rough path is time-augmented and satisfies \eqref{eq::EdgesAreIntegration}, then the above defined maps $(\Pi,\Gamma)$ are a model for the regularity structure $(A,\mathcal{PB},\mathcal{G})$.
\end{theorem}

\begin{proof}
We have to show that $\Gamma_{xy}\Gamma_{yz}=\Gamma_{xz}$, that $\Pi_x\Gamma_{xy}=\Pi_y$. The first conditions follows immediately from Chen's identity on rough paths. For the second identity, first let $\omega \in \mathcal{OF}$, then:
\begin{align*}
\Pi_x\Gamma_{xy}(B_+(\varphi(\omega)))(t)=&\Pi_x( (\mathbb{X}_{yx} \otimes B_+ \circ \varphi)\Delta_{MKW}(\omega) )(t)\\
=&\langle \mathbb{X}_{xt},( \mathbb{X}_{yx} \otimes id)\Delta_{MKW}(\omega)\rangle \\
=&\langle \mathbb{X}_{yx}\otimes \mathbb{X}_{xt},\Delta_{MKW}(\omega)\rangle \\
=&\langle \mathbb{X}_{yt},\omega\rangle\\
=&\Pi_y(B_+(\varphi(\omega)))(t),
\end{align*}
which shows that the identity holds for any tree that has no decorated edges adjacent to the root. For a general tree:
\begin{align*}
	\Pi_x\Gamma_{xy}(\varphi(\omega))(t)=&\frac{d}{dt}\Pi_x(B_+(\Gamma_{xy}\varphi(\omega)  )  )(t)\\
	=&\frac{d}{dt}\Pi_x( \Gamma_{xy}(B_+(\varphi(\omega)))-\langle\mathbb{X}_{yx},B_+(\omega)\rangle {\color{blue} \Forest{[]}}  )(t)\\
	=&\frac{d}{dt}\Pi_x( \Gamma_{xy}(B_+(\varphi(\omega))))(t)\\
	=&\frac{d}{dt}\Pi_y(B_+(\varphi(\omega)))(t)\\
	=&\Pi_y(\varphi(\omega))(t),
\end{align*}
where we used Lemma \ref{lemma::ModelOnIdentityIsConstant}.
\end{proof}

We now generalise the renormalisation construction from Section \ref{ssect:roughregstr}. Let $\rho_0: Lie(\mathcal{PT}) \to S(Lie(\mathcal{PT}))\otimes Lie(\mathcal{PT})$ be time-cotranslation in planarly branched rough paths. Let $\pi^-$ be projection onto trees of negative degree and denote $\rho_0^-=(\pi^- \otimes id)\rho_0 : Lie(\mathcal{PT})\to S(\mathcal{PT}_-) \otimes Lie(\mathcal{PT})$. Define $\Delta^-: \mathcal{PB} \to S(\mathcal{PB}_-) \otimes \mathcal{PB}$ by $\Delta^- = (\varphi \otimes \varphi)\rho_0^-$. For a character $\ell$ in the dual algebra $\mathcal{PB}_-^{\ast}$, define renormalisation by $\ell$ as $M_{\ell}=(\ell \otimes id)\Delta^-$. For a model $(\Pi,\Gamma)$, define the renormalised model $(\Pi^{M_{\ell}},\Gamma^{M_{\ell}})$ by
\begin{align*}
\Pi_s^{M_{\ell}}=&\Pi_s M_{\ell},\\
\Gamma_{st}^{M_{\ell}}=&(\gamma_{ts}M_{\ell}\otimes id)\Delta^+.
\end{align*}

\begin{proposition} \label{prop::RoughPathRenormalisedModelIsModel}
The renormalised model is a model.
\end{proposition}

\begin{proof}
First we prove that $\Gamma_{xy}^{M_{\ell}}\Gamma_{yz}^{M_{\ell}}=\Gamma_{xz}^{M_{\ell}}$ for an input tree with no decorated edge adjacent to the root. By definition, we can rewrite:
\begin{align*}
\Gamma_{xy}^{M_{\ell}}\Gamma_{yz}^{M_{\ell}}(B_+(\varphi(\omega)))=&\Gamma_{xy}^{M_{\ell}}(\gamma_{zy}M_{\ell} \otimes id)\Delta^+(B_+(\varphi(\omega)))\\
=&(\gamma_{zy}M_{\ell}\otimes \gamma_{yx}M_{\ell}\otimes id)(id \otimes \Delta^+)\Delta^+(B_+(\varphi(\omega))).
\end{align*}
Then using the definition of $\Delta^+$, denoting $\varphi^{-1}(\log(\ell))=\hat{\ell}$, and noting that $\gamma_{yz}M_\ell \circ B_+ \circ \varphi = \mathbb{X}_{yz}T_{\hat{\ell}}^{\ast}$, we get:
\begin{align*}
=&(\mathbb{X}_{zy}T_{\hat{\ell}}^{\ast}\otimes \mathbb{X}_{yx}T_{\hat{\ell}}^{\ast}\otimes B_+ \circ \varphi)(id \otimes \Delta_{MKW})\Delta_{MKW}(\omega)\\
=&(\mathbb{X}_{zy}T_{\hat{\ell}}^{\ast}\otimes \mathbb{X}_{yx}T_{\hat{\ell}}^{\ast}\otimes B_+ \circ \varphi)(\Delta_{MKW}\otimes id)\Delta_{MKW}(\omega)\\
=&(\exp(\hat{\ell}) \otimes \mathbb{X}_{zy} \otimes \exp(\hat{\ell})\otimes  \mathbb{X}_{yx} \otimes B_+ \circ \varphi)(\rho_0^- \otimes \rho_0^- \otimes id)(\Delta_{MKW} \otimes id)\Delta_{MKW}(\omega).
\end{align*}
Now we use the character property of $\exp(\hat{\ell})$:
\begin{align*}
=&(\exp(\hat{\ell})\otimes \mathbb{X}_{zy}\otimes  \mathbb{X}_{yx} \otimes B_+ \circ \varphi)m^{1,3}(\rho_0^- \otimes \rho_0^- \otimes id)(\Delta_{MKW} \otimes id)\Delta_{MKW}(\omega),
\end{align*}
and the cointeraction property between $\rho_0^-$ and $\Delta_{MKW}$:
\begin{align*}
=&(\exp(\hat{\ell})\otimes \mathbb{X}_{zy}\otimes  \mathbb{X}_{yx} \otimes B_+ \circ \varphi)(id \otimes \Delta_{MKW} \otimes id)(\rho_0^- \otimes id)\Delta_{MKW}(\omega)\\
=&\exp(\hat{\ell}) \otimes \mathbb{X}_{zy}\star \mathbb{X}_{yx}\otimes id)(\rho_0^- \otimes B_+ \circ \varphi)\Delta_{MKW}(\omega)\\
=&(\mathbb{X}_{zx}T_{\hat{\ell}}^{\ast}\otimes B_+ \circ \varphi)\Delta_{MKW}(\omega)\\
=&\Gamma_{xz}^{M_{\ell}}(B_+(\varphi(\omega))).
\end{align*}
For an arbitrary input argument, we use the definitions to write:
\begin{align*}
	&\Gamma_{xy}^{M_{\ell}}\Gamma_{yz}^{M_{\ell}}(\varphi(\omega))\\
	=&\Gamma_{xy}^{M_{\ell}}(\gamma_{zy}M_{\ell} \otimes id)(id \otimes B_-)(\Delta^+(B_+(\varphi(\omega))) - B_+(\varphi(\omega))\otimes {\color{blue}\Forest{[]}} )\\
	=&(\gamma_{zy}M_{\ell} \otimes \gamma_{yx}M_{\ell}\otimes B_-)(id \otimes \Delta ^+ - id \otimes id \otimes {\color{blue}\Forest{[]}})(\Delta^+(B_+(\varphi(\omega))) - B_+(\varphi(\omega))\otimes {\color{blue}\Forest{[]}} ).
\end{align*}
Then translating everything into rough path language gives:
\begin{align*}
	=&(\mathbb{X}_{zy}T_{\hat{\ell}}^{\ast}\otimes \mathbb{X}_{yx}T_{\hat{\ell}}^{\ast} \otimes \varphi )(id \otimes \Delta_{MKW} - id \otimes id \otimes 1)(\Delta_{MKW}(\omega)-\omega \otimes 1)\\
	=&(\exp(\hat{\ell})\otimes \mathbb{X}_{zy}\otimes \exp(\hat{\ell}) \otimes \mathbb{X}_{yx}\otimes \varphi)(\rho_0^- \otimes \rho_0^- \otimes id)(id \otimes \Delta_{MKW} - id \otimes id \otimes 1)(\Delta_{MKW}(\omega)-\omega \otimes 1)\\
	=&(\exp(\hat{\ell}) \otimes \mathbb{X}_{zy}\otimes \exp(\hat{\ell})\otimes \mathbb{X}_{yx}\otimes \varphi)(\rho_0^- \otimes \rho_0^- \otimes id)( (id \otimes \Delta_{MKW})\Delta_{MKW}(\omega) - \Delta_{MKW}(\omega)\otimes 1).
	\end{align*}
	We now use the character property and the cointeraction property:
	\begin{align*}
	=&(\exp(\hat{\ell}) \otimes \mathbb{X}_{zy}\otimes \mathbb{X}_{yx}\otimes \varphi)m^{1,3}(\rho_0^- \otimes \rho_0^- \otimes id)( (\Delta_{MKW} \otimes id)\Delta_{MKW}(\omega) - \Delta_{MKW}(\omega)\otimes 1 )\\
	=&(\exp(\hat{\ell}) \otimes \mathbb{X}_{zy}\otimes \mathbb{X}_{yx}\otimes \varphi)(id \otimes \Delta_{MKW} \otimes id)( (\rho_0^- \otimes id)\Delta_{MKW}(\omega) - \rho_0^-(\omega)\otimes 1)\\
	=&(\exp(\hat{\ell})\otimes \mathbb{X}_{zy} \star \mathbb{X}_{yx}\otimes \varphi)( (\rho_0^- \otimes id)\Delta_{MKW}(\omega) - \rho_0^-(\omega)\otimes 1)
	\end{align*}
	We now factor out the $\rho_0^-$ in the rightmost term, and translate everything back to regularity structure language using the definitions.
	\begin{align*}
	=&(\exp(\hat{\ell}) \otimes \mathbb{X}_{zx}\otimes \varphi)(\rho_0^- \otimes id )(\Delta_{MKW}(\omega) - \omega \otimes 1 )\\
	=&(\mathbb{X}_{zx}T_{\hat{\ell}}^{\ast} \otimes \varphi)(\Delta_{MKW}(\omega) - \omega \otimes 1)\\
	=&(\gamma_{zx}M_{\ell}\otimes B_-)(\Delta^+(B_+(\varphi(\omega))) -B_+(\varphi(\omega))\otimes {\color{blue}\Forest{[]}} )\\
	=&\Gamma_{xz}^{M_{\ell}}(\varphi(\omega)).
\end{align*}
Next we prove $\Pi_x^{M_{\ell}}\Gamma_{xy}^{M_{\ell}}=\Pi_y$. First if the input tree has no decorated edges adjacent to the root, we can rewrite:
\begin{align*}
\Pi_x^{M_{\ell}}\Gamma_{xy}^{M_{\ell}}(B_+(\varphi(\omega)))(t)=&\Pi_x^{M_{\ell}}(\gamma_{yx}M_{\ell}\otimes id)\Delta^+(B_+(\varphi(\omega)))(t)\\
=&(\mathbb{X}_{yx}T_{\hat{\ell}}^{\ast}\otimes \mathbb{X}_{xt}T_{\hat{\ell}}^{\ast})\Delta_{MKW}(\omega)\\
=&( \exp(\hat{\ell}) \otimes \mathbb{X}_{yx} \otimes \exp(\hat{\ell})  \otimes \mathbb{X}_{xt})(\rho_0^- \otimes \rho_0^-)\Delta_{MKW}(\omega).
\end{align*}
Then we can use the character property and cointeraction:
\begin{align*}
=&(\exp(\hat{\ell}) \otimes \mathbb{X}_{yx}\otimes \mathbb{X}_{xt})m^{1,3}(\rho_0^-\otimes \rho_0^-)\Delta_{MKW}(\omega)\\
=&(\exp(\hat{\ell}) \otimes \mathbb{X}_{yx}\otimes \mathbb{X}_{xt})(id \otimes \Delta_{MKW})\rho_0^-(\omega)\\
=&(\exp(\hat{\ell}) \otimes \mathbb{X}_{yt})\rho_0^-(\omega)\\
=&\mathbb{X}_{yt}T_{\hat{\ell}}^{\ast}(\omega)\\
=&\Pi_y^{M_{\ell}}(B_+(\varphi(\omega)))(t).
\end{align*}
Finally, if the input tree is arbitrary, we use the definitions to write:
\begin{align*}
	\Pi_x^{M_{\ell}}\Gamma_{xy}^{M_{\ell}}(\varphi(\omega))(t)=&\Pi_x^{M_{\ell}}(\gamma_{yx}M_{\ell}\otimes B_-)(\Delta^+(B_+(\varphi(\omega))) -B_+(\varphi(\omega)) \otimes {\color{blue}\Forest{[]}} )(t)\\
	=&(\gamma_{yx}M_{\ell} \otimes \frac{d}{dt}\Pi_x B_+ M_{\ell} B_- )(\Delta^+(B_+(\varphi(\omega))) -B_+(\varphi(\omega)) \otimes {\color{blue}\Forest{[]}} )(t)\\
	=&(\gamma_{yx}M_{\ell} \otimes \frac{d}{dt}\Pi_x M_{\ell} )(\Delta^+(B_+(\varphi(\omega))) -B_+(\varphi(\omega)) \otimes {\color{blue}\Forest{[]}} )(t).
	\end{align*}
	We then translate into the language of rough paths:
	\begin{align*}
	=&(\mathbb{X}_{yx}T_{\hat{\ell}}^{\ast} \otimes \frac{d}{dt}\mathbb{X}_{xt}T_{\hat{\ell}})( \Delta_{MKW}(\omega) - \omega \otimes 1 )\\
	=&\frac{d}{dt}(\exp(\hat{\ell}) \otimes \mathbb{X}_{yx}\otimes \exp(\hat{\ell}) \otimes \mathbb{X}_{xt})(\rho_0^- \otimes \rho_0^-)( \Delta_{MKW}(\omega) - \omega \otimes 1 ).
	\end{align*}
	Use the character property and cointeraction:
	\begin{align*}
	=&\frac{d}{dt}(\exp(\hat{\ell}) \otimes \mathbb{X}_{yx}\otimes \mathbb{X}_{xt})m^{1,3}(\rho_0^- \otimes \rho_0^-)( \Delta_{MKW}(\omega) - \omega \otimes 1 )\\
	=&\frac{d}{dt}(\exp(\hat{\ell}) \otimes \mathbb{X}_{yx}\otimes \mathbb{X}_{xt})(id \otimes \Delta_{MKW})\rho_0^-(\omega)-\frac{d}{dt}(\exp(\hat{\ell}) \otimes \mathbb{X}_{yx}\otimes \mathbb{X}_{xt})(\rho_0^-(\omega)\otimes 1).
	\end{align*}
	Now we note that $\frac{d}{dt}\langle \mathbb{X}_{xt},1\rangle =0$ and simplify:
	\begin{align*}
	=&\frac{d}{dt}(\exp(\hat{\ell}) \otimes \mathbb{X}_{yt} )\rho_0^-(\omega)\\
	=&\frac{d}{dt}\mathbb{X}_{yt}T_{\hat{\ell}}^{\ast}(\omega)\\
	=&\frac{d}{dt}\Pi_y^{M_{\ell}}(B_+(\varphi(\omega)))\\
	=&\Pi_y^{M_{\ell}}(\varphi(\omega)),
\end{align*}
where we used that $B_+M_{\ell}=M_{\ell}B_+$, which follows from that the edge added by $B_+$ can not be part of a subtree with negative regularity.
\end{proof}

We conclude the section by giving the combinatorial description of $\Delta^-$.

\begin{definition} \label{def::AdmissiblePartitionRough}
	Let $\tau \in \mathcal{PB}$ be a tree and let $\tau_1,\dots,\tau_n$ be a set of disjoint (not necessarily spanning) subtrees of $\tau$. We say that $\tau_1,\dots,\tau_n$ is $\mathcal{PB}_-$-admissible if the following conditions hold:
	\begin{enumerate}
		\item If $e$ is an edge in $\tau_i$ and $e'$ is an edge in $\tau$ that is outgoing from the same vertex as $e$, and if $e'$ is to the right of $e$ in the planar embedding, then $e'$ is in $\tau_i$.
		\item If $e=(v_1,v_2)$ is an edge in $\tau$ with a non-zero decoration then either both $v_1$ and $v_2$ belong to the same subtree, or neither $v_1$ or $v_2$ belong to any subtree.
		\item Each $\tau_i$ has negative regularity.
	\end{enumerate}
For $\tau_1,\dots,\tau_n$ $\mathcal{PB}_-$-admissible, denote by $\tau / \tau_1 \dots \tau_n$ the tree obtained by contracting each $\tau_i$ into a single vertex.
\end{definition}

\begin{proposition}
	Let $\tau \in \mathcal{PB}$, then
	\begin{align*}
		\Delta^-(\tau)=\sum_{\tau_1,\dots,\tau_n \atop \mathcal{PB}_-\text{-admissible}} \tau_1 \dots \tau_n \otimes \tau/\tau_1 \dots \tau_n,
	\end{align*}
where the product in the left tensor is the free symmetric product in $S(\mathcal{PB}_-)$.
\end{proposition}

\begin{example} Let $\frac{1}{2} - \epsilon < \alpha < \frac{1}{2}$ for all positive $\epsilon$, then:
	\begin{align*}
		\Delta^-({\color{blue}\Forest{[ [[][,edge label = {node[midway,fill=white,scale=0.5]{2}}]][[,edge label = {node[midway,fill=white,scale=0.5]{3}}]] [,edge label = {node[midway,fill=white,scale=0.5]{1}}] ]}})=& {\color{blue} \Forest{[[,edge label={node[midway,fill=white,scale=0.5]{1}}]]} } \otimes {\color{blue}\Forest{[ [[][,edge label = {node[midway,fill=white,scale=0.5]{2}}]][[,edge label = {node[midway,fill=white,scale=0.5]{3}}]] ]}} + {\color{blue} \Forest{[[,edge label = {node[midway,fill=white,scale=0.5]{2}}]]} } \otimes {\color{blue}\Forest{[ [[]][[,edge label = {node[midway,fill=white,scale=0.5]{3}}]] [,edge label = {node[midway,fill=white,scale=0.5]{1}}] ]}} + {\color{blue} \Forest{[[,edge label = {node[midway,fill=white,scale=0.5]{3}}]]} } \otimes {\color{blue}\Forest{[ [[][,edge label = {node[midway,fill=white,scale=0.5]{2}}]][] [,edge label = {node[midway,fill=white,scale=0.5]{1}}] ]}}\\
		+&{\color{blue} \Forest{[[[,edge label = {node[midway,fill=white,scale=0.5]{3}}]][,edge label = {node[midway,fill=white,scale=0.5]{1}}]]} } \otimes {\color{blue}\Forest{[ [[][,edge label = {node[midway,fill=white,scale=0.5]{2}}]] ]}}+{\color{blue} \Forest{[[,edge label = {node[midway,fill=white,scale=0.5]{1}}]]} \Forest{[[,edge label = {node[midway,fill=white,scale=0.5]{2}}]]} } \otimes {\color{blue}\Forest{[ [[]][[,edge label = {node[midway,fill=white,scale=0.5]{3}}]] ]}}+{\color{blue}\Forest{[[,edge label={node[midway,fill=white,scale=0.5]{1}}]]}\Forest{[[,edge label = {node[midway,fill=white,scale=0.5]{3}}]]}} \otimes {\color{blue}\Forest{[ [[][,edge label = {node[midway,fill=white,scale=0.5]{2}}]][] ]}}\\
		+&{\color{blue} \Forest{[[,edge label = {node[midway,fill=white,scale=0.5]{2}}]]}\Forest{[[,edge label = {node[midway,fill=white,scale=0.5]{3}}]]} } \otimes {\color{blue}\Forest{[ [[]][] [,edge label = {node[midway,fill=white,scale=0.5]{1}}] ]}} + {\color{blue} \Forest{[[[,edge label = {node[midway,fill=white,scale=0.5]{3}}]][,edge label = {node[midway,fill=white,scale=0.5]{1}}]]} \Forest{[[,edge label = {node[midway,fill=white,scale=0.5]{2}}]]} } \otimes {\color{blue} \Forest{[[[]]]} }.
	\end{align*}
\end{example}

\section{Planar regularity structures}
\label{sec::PlanarRegularityStructures}

We now want to extend the notions from the previous section, to define a concept of planar regularity structures and their renormalisation. To achieve this, we will adapt the constructions based on the pre-Lie structure of non-planar trees to a construction based on the post-Lie structure of planar trees. We will need the following notation:

Let $\mathcal{I}=\{G_1,\dots,G_k \}$ be a set of convolution kernels and let $\Xi = \{X_t^1,\dots,X_t^n \}$ be a set of noises. Denote $\mathcal{D}=\mathcal{I}\cup \Xi \times \mathbb{N}^d$. Let $\mathfrak{T}$ denote the space of planar trees where the edges are decorated by $\mathcal{D}$ and the vertices are decorated by $\mathbb{N}^d$, such that each vertex has at most one outgoing noise edge and such that vertices with an incoming noise edge has no outgoing edges. We will say that a vertex is a \textit{noise vertex} if it has an incoming noise edge. We denote by $\tilde{\mathfrak{T}}$ the subset of trees that has zero root decoration, which we call \textit{planted forests}. To each tree $\tau \in \mathfrak{T}$, we define the regularity $|\tau|$ as the sum over all edge decorations and all vertex decorations; vertex decorations contribute the sum of its components, the $\mathbb{N}^d$ part of an edge decoration subtracts the sum of its components, a noise decoration $X_t^i$ contributes $\alpha_i -1$ where $\alpha_i$ is the regularity of the path $X_t^i$, a convolution kernel $G_i$ contributes the regularity of the convolution kernel. The space $\mathfrak{T}$ will be the model space for the regularity structure we are constructing, and the index set will be the set of possible regularities. Following the planarly branched rough path example, we construct the group of linear operators by considering a Hopf algebra in the next subsection. To construct this Hopf algebra, we use the natural generalisation of the non-planar construction from \cite{BrunedManchon2022} and \cite{BrunedKatsetsiadis2022}.

\subsection{The positive Hopf algebra} \label{ssect:PositiveHopf}

We follow the constructions in \cite{BrunedManchon2022} and \cite{BrunedKatsetsiadis2022}, but adapted for planar trees.\\

Let $\mathfrak{T}^+$ denote the subspace of $\mathfrak{T}$ generated by trees of positive regularity that has no noise edge at the root. Let $\tilde{\mathfrak{T}}^+$ denote the subspace of $\mathfrak{T}^+$ spanned by planted forests. As before, we will draw trees in $\mathfrak{T}^+$ using the colour blue to indicate that they are regularity structure trees, as opposed to rough path trees. We will write $I_a(\omega)$ for the tree obtained by grafting the ordered forest $\omega$ onto a zero-decorated vertex using edges decorated by $(I,a)\in \mathcal{D}$. We say that $I_a(\tau)$ is a planted tree if $\tau$ is a tree. Let $\tilde{\mathcal{V}}$ denote the free Lie algebra generated by planted trees. We will see $\tilde{\mathcal{V}}$ as a subspace of $\mathfrak{T}^+$ by identifying the generic elements $[I^i_a(\tau_1),I^j_b(\tau_2)]$ with a difference of trees as follows:
\begin{align} \label{eq::ForestAsTree}
[I^i_a(\tau_1),I^j_b(\tau_2)] = {\color{blue} \Forest{ [[\tau_1,edge label = {node[midway,fill=white,scale=0.5]{$d$}}] [\tau_2,edge label = {node[midway,fill=white,scale=0.5]{$d'$}}] ] } } - {\color{blue} \Forest{ [[\tau_2,edge label = {node[midway,fill=white,scale=0.5]{$d'$}}] [\tau_1,edge label = {node[midway,fill=white,scale=0.5]{$d$}}] ] } },
\end{align}
where a vertex decorated by $\tau_i$ symbolically represents a subtree $\tau_i$ that could be bigger than a single vertex, and where we write $d=(I^i,a),d'=(I^j,b)$ for the edge decorations. We define a product on the subspace of planted trees by:
\begin{align*}
I^i_a(\tau_1)\graft I^j_b(\tau_2)=I^j_b(\tau_1 \graft_{I^i_a} \tau_2),
\end{align*}
where $\graft_{I^i_a}$ means left grafting of $\tau_1$ onto $\tau_2$ using an edge decorated by $(I^i,a) \in \mathcal{D}$. We do not graft on noise vertices, as this would result in trees outside of $\mathfrak{T}$.\\

Recall the operator $\uparrow_v^{\ell}$ on non-planar trees from Section \ref{ssect:trees}, we now extend it to planar trees by letting $\uparrow_v^{\ell}\tau$ be the tree obtained by changing the decoration of the vertex $v\in V(\tau)$ from $n_v$ to $n_v+\ell$. Let $\uparrow_v^{-\ell}$ by defined by reducing the decoration. Define $\uparrow^{\ell}\tau = \sum_{v \in N_{\tau}} \uparrow_v^{\ell}\tau$, where $N_{\tau}$ is the set of non-noise vertics in $\tau$. Define the deformed grafting product by
\begin{align*}
I^i_a(\tau_1)\dgraft I^j_b(\tau_2)=\sum_{v \in N_{\tau_2}} \sum_{\ell \in \mathbb{N}^d} {n_v \choose \ell}I^j_b( \tau_1 \graft^v_{I^i_{a-\ell}}(\uparrow_v^{-\ell} \tau_2 ) ),
\end{align*}
where the terms in the sum are understood to be $0$ whenever $a-\ell$ or $n_v - \ell$ has a strictly negative component. We now extend the deformed grafting product to $\tilde{\mathcal{V}}$ such that the post-Lie axioms are satisfied:

\begin{align}
a_{\dgraft}(I^i_a(\omega_1),I^j_b(\omega_2),I^k_c(\omega_3))-a_{\dgraft}(I^j_b(\omega_2),I^i_a(\omega_1),I^k_c(\omega_3))=&[I^i_a(\omega_1),I^j_b(\omega_2)] \dgraft I^k_c(\omega_3), \label{eq::DeformedPostLie1} \\
[I ^i_a(\omega_1) \dgraft I^j_b(\omega_2),I^k_c(\omega_3)] + [I^j_b(\omega_2),I^i_a(\omega_1) \dgraft I^k_c(\omega_3)]=&I^i_a(\omega_1) \dgraft [I^j_b(\omega_2),I^k_c(\omega_3)]. \label{eq::DeformedPostLie2}
\end{align}

Identity \eqref{eq::DeformedPostLie1} amounts to saying that deformed grafting a Lie polynomial means that each added edge is deformed individually, for example:
\begin{align*}
[I^i_a(\tau_1),I^j_b(\tau_2)] \dgraft {\color{blue} \Forest{[[\alpha,edge label = {node[midway,fill=white,scale=0.5]{$d''$}}]]}  }= \sum_{\ell_1,\ell_2} {\color{blue} \Forest{[[\alpha-\ell_1-\ell_2,edge label = {node[midway,fill=white,scale=0.5]{$d''$}}  [\tau_1,edge label = {node[midway,fill=white,scale=0.5]{$d-\ell_1$}}] [\tau_2,edge label = {node[midway,fill=white,scale=0.5]{$d'-\ell_2$}}] ]]} } - \sum_{\ell_1,\ell_2} {\color{blue} \Forest{[[\alpha-\ell_1-\ell_2,edge label = {node[midway,fill=white,scale=0.5]{$d''$}} [\tau_2,edge label = {node[midway,fill=white,scale=0.5]{$d'-\ell_2$}}][\tau_1,edge label = {node[midway,fill=white,scale=0.5]{$d-\ell_1$}}] ]]} },
\end{align*}

where $d-\ell_1 = (I^i,a-\ell_1)$ and $d'-\ell_2=(I^j,b-\ell_2)$. Let
\begin{align*}
\mathcal{V}=\tilde{\mathcal{V}} \oplus \langle \{X^i : i= (0,\dots,0,1,0,\dots,0) \} \rangle
\end{align*}
denote the vector space spanned by Lie polynomials of planted trees and by unit vectors in $\mathbb{N}^d$. Define a product $\dgraftt$ on $\mathcal{V}$ by
\begin{align*}
X^i \dgraftt y=&\uparrow^iy,\\
y \dgraftt X^i=&0,\\
X^i \dgraftt X^j =& 0,\\
y \dgraftt z=&y \dgraft z,
\end{align*}
for $y,z \in \tilde{\mathcal{V}}$. Furthermore define a Lie bracket by
\begin{align*}
[X^i,X^j]_0=&0,\\
[y,z]_0=&[y,z],\\
[y,X^i]_0=&\downarrow^iy,
\end{align*}
where $\downarrow^iy$ sums over the edges in $y$ that are adjacent to the root, and subtracts $i$ from one decoration per term, for example:
\begin{align*}
[[I^j_a(\omega_1),I^k_b(\omega_2)  ],X^i]_0=& \downarrow^i [I^j_a(\omega_1),I^k_b(\omega_2)]\\
=&[I^j_{a-i}(\omega_1),I^k_{b}(\omega_2)]+[I^j_{a}(\omega_1),I^k_{b-i}(\omega_2)].
\end{align*}

Note that this is equivalent to defining $\downarrow^i(I_a^j(\omega))=I_{a-i}^j(\omega)$ and extending to be a derivation for the Lie bracket $[\cdot,\cdot]$.

\begin{lemma} \label{lemma::AlmostDerivation}
The following identity holds:
\begin{align*}
\uparrow^i(y \dgraftt z) = (\uparrow^iy) \dgraftt z + y \dgraftt (\uparrow^i z) - (\downarrow^i y) \dgraftt z.
\end{align*}
\end{lemma}

\begin{proof}
The same proof as for Proposition $3.1$ in \cite{BrunedKatsetsiadis2022} applies when $y$ is a planted tree. If $y$ is a Lie polynomial, we note that we can reduce it to the planted tree case by repeated application of identity \eqref{eq::DeformedPostLie1}.
\end{proof}

\begin{proposition}
$(\mathcal{V},\dgraftt,[\cdot,\cdot]_0)$ is a post-Lie algebra.
\end{proposition}

\begin{proof}
The identities
\begin{align*}
a(x,y,z)-a(y,x,z)=&[x,y]_0\dgraftt z,\\
z \dgraftt [x,y]_0=&[z\dgraftt x,y]_0 + [x,z \dgraftt y]_0
\end{align*}
are clear if $x=X^i$ and $y=X^j$, or if both $x,y$ are in $\tilde{\mathcal{V}}$. They are also clear if $z = X^j$. If $x,z\in \tilde{\mathcal{V}}$ and $y=X^i$, then by Lemma \ref{lemma::AlmostDerivation}:
\begin{align*}
a(x,y,z)-a(y,x,z)=&x \dgraftt \uparrow^i z - \uparrow^i (x \dgraftt z) + (\uparrow^i x)\dgraftt z\\
=&\downarrow^i x \dgraftt z \\
=&[x,y]_0 \dgraftt z.
\end{align*}
Furthermore:
\begin{align*}
z \dgraftt [x,y]_0=& z \dgraftt (\downarrow^i x)\\
=&\downarrow^i (z \dgraftt x) + [x,0]_0\\
=&[z \dgraftt x,y]_0 + [x,z \dgraftt y]_0,
\end{align*}
where we used the identity $z \dgraftt (\downarrow^i x)=\downarrow^i (z \dgraftt x)$, which comes from the fact that $\dgraftt$ does not add any edges adjacent to the root.
\end{proof}

Because $(\mathcal{V},\dgraftt,[\cdot,\cdot]_0)$ is a post-Lie algebra we can apply the planar Guin-Oudom construction from \cite{EbrahimiFardLundervoldMuntheKaas2014} to extend $\dgraftt$ to the Lie enveloping algebra of $\mathcal{V}$. This is a generalisation of the Guin-Oudom construction for pre-Lie algebras from \cite{GuinOudom2008}.

\begin{lemma}
The Lie enveloping algebra of $(\mathcal{V},\dgraftt,[\cdot,\cdot]_0)$ can be represented by the space $\mathfrak{T}^+$.
\end{lemma}

\begin{proof}
It is clear that $X^i$'s commute under the associative product on $\mathcal{U}(\mathcal{V})$, as their commutator Lie bracket is abelian. Let $X^m$ be a commutative polynomial of $X^i$'s whose indices adds up to $m \in( \mathbb{N}^d)^{\times 2}$ and represent this object with a single vertex decoared by $m$. It is furthermore clear by the identification \eqref{eq::ForestAsTree} that we can represent the associative product of two planted trees by merging the roots of the trees letting the branches of the left argument be to the left of the branches of the right argument. It remains to represent the associative product between single vertices and planted trees. Let $X^i\tau$ be represented by the planted tree $\tau$, except with the root decorated by $i$. Note that this implies that $\tau X^i$ is represented by $\tau X^i=X^i\tau - [X^i,\tau]_0 = X^i \tau + \downarrow^i \tau$. Since $\downarrow^i \tau$ is again a planted tree, this means that we can always write any element in $\mathcal{U}(\mathcal{V})$ as a sum of elements of the form $X^m \tau_1 \dots \tau_k$, which we represent by the tree $\tau_1 \dots \tau_k$, whose root has decoration $m$.
\end{proof}

We endow the space $\mathfrak{T}^+$ with the cocommutative coproduct $\Delta_{\shuffle}$, defined by letting $\mathcal{V}$ be primitive elements and then generating as an algebra morphism for the associative product. We extend $\dgraftt$ to the Lie enveloping algebra by:

\begin{align}
1 \dgraftt \omega =& \omega,\label{eq::DAlg1} \\
\tau_1\omega \dgraftt \tau_2 =& \tau_1 \dgraftt (\omega \dgraftt \tau_2) - (\tau_1 \dgraftt \omega) \dgraftt \tau_2,\label{eq::DAlg2}\\
\omega_1 \dgraftt (\omega_2\omega_3) =& ((\omega_1)_{(1)} \dgraftt \omega_2)((\omega_1)_{(2)} \dgraftt \omega_3), \label{eq::DAlg3}
\end{align}

where $\Delta_{\shuffle}(\omega_1)=(\omega_1)_{(1)} \otimes (\omega_1)_{(2)}$ is Sweedler's notation, $\tau_1,\tau_2 \in \mathcal{V}$ and $\omega,\omega_1,\omega_2,\omega_3 \in \mathfrak{T}^+$.

\begin{lemma}
Let $X^m \omega_1, X^n \omega_2$ be two generic elements in $\mathfrak{T}^+$, with $\omega_1,\omega_2$ planted forests. Then:
\begin{align*}
X^m \omega_1 \dgraftt X^n \omega_2 = X^n(X^m \omega_1 \dgraftt \omega_2).
\end{align*}
\end{lemma}

\begin{proof}
This follows by equations \eqref{eq::DAlg2},\eqref{eq::DAlg3} and the fact that anything grafted on an $X^i$ is zero.
\end{proof}

We now want to describe the expression $X^m \omega_1 \dgraftt \omega_2$. To do this, we first extend the operators $\uparrow^m$ and $\downarrow^m$ to arbitrary $m \in \mathbb{N}^d$ and acting on the whole $\tilde{\mathfrak{T}}$. For $\omega \in \tilde{\mathfrak{T}}$, let $\uparrow^m\omega$ be defined by composition $\uparrow^m\omega=(\uparrow^{(1,0,\dots,0)})^{m_1}\circ \dots \circ (\uparrow^{(0,\dots,0,1)})^{m_d}\omega$. Similarly for $\downarrow^m$. Note that this is the same as summing over all ways to increase the decorations of vertices in $\omega$, weighted by the number of ways to do this increase by adding a unit at a time. We illustrate this with an example where $\mathbb{N}^d=\mathbb{N}^1$:
\begin{align*}
\uparrow^2 {\color{blue} \Forest{[k_1[k_2,edge label = {node[midway,fill=white,scale=0.5]{$a$}}]]} }=& \uparrow^1({\color{blue} \Forest{[k_1+1[k_2,edge label = {node[midway,fill=white,scale=0.5]{$a$}}]]} } +{\color{blue} \Forest{[k_1[k_2+1,edge label = {node[midway,fill=white,scale=0.5]{$a$}}]]} }  )\\
=&{\color{blue} \Forest{[k_1+2[k_2,edge label = {node[midway,fill=white,scale=0.5]{$a$}}]]} } + {\color{blue} \Forest{[k_1+1[k_2+1,edge label = {node[midway,fill=white,scale=0.5]{$a$}}]]} } + {\color{blue} \Forest{[k_1+1[k_2+1,edge label = {node[midway,fill=white,scale=0.5]{$a$}}]]} } + {\color{blue} \Forest{[k_1[k_2+2,edge label = {node[midway,fill=white,scale=0.5]{$a$}}]]} }\\
=&\sum_{2=\delta_1+\delta_2} {2 \choose \delta_1,\delta_2}  {\color{blue} \Forest{[k_1+\delta_1[k_2+\delta_2,edge label = {node[midway,fill=white,scale=0.5]{$a$}}]]} }.
\end{align*}
We furthermore have the identity $X^m\dgraftt \omega = \uparrow^m \omega$, which is straightforward to see from equation \eqref{eq::DAlg2}.

\begin{lemma}
The following identity holds:
\begin{align*}
 X^m \omega_1 \dgraftt \omega_2 =& \uparrow^m(\omega_1 \dgraftt \omega_2) - (\uparrow^m \omega_1) \dgraftt \omega_2 \\
=& \omega_1 \dgraftt (\uparrow^m \omega_2) - (\downarrow^m \omega_1 )\dgraftt \omega_2 .
\end{align*}
\end{lemma}

\begin{proof}
If $\omega_2 \in \tilde{\mathcal{V}}$, then the statements follows from equation \eqref{eq::DAlg2} and Lemma \ref{lemma::AlmostDerivation}. If $\omega_2$ is not primitive, then note that iterating equation \eqref{eq::DAlg3} says that we have to split $\omega_2$ into all of its branches and then look at all possible ways to pair up $X^{m_i}$'s and branches of $\omega_1$ with the branches of $\omega_2$, with $\sum X^{m_i}=X^m$. But splitting $X^m$ into $X^{m_i}$'s and summing over all ways to apply these to different branches, is the same as applying $X^m$ to the whole tree.
\end{proof}

\begin{lemma}
The following identity holds:
\begin{align*}
X^m \omega X^n =\sum_{r_1+r_2=n} {n \choose r_2} (X^{m+r_1}\downarrow^{r_2}\omega).
\end{align*}
\end{lemma}

\begin{proof}
We can write $\omega = \tau_1 \dots \tau_k$ and $X^n=X^{n_1}\dots X^{n_p}$, with $\tau_j$'s and $X^{n_q}$'s primitive, then:
\begin{align*}
X^m \omega X^n =& X^m \tau_1 \dots \tau_k X^n \\
=& X^m\tau_1 \dots \tau_{k-1}X^{n_1} \tau_k X^{n_2}\dots X^{n_p} + X^m \tau_1 \dots \tau_{k-1}(\downarrow^{n_1} \tau_k)X^{n_2}\dots X^{n_p}\\
=&X^m \tau_1 \dots \tau_{k-2}X^{n_1} \tau_{k-1} \tau_kX^{n_2}\dots X^{n_p} + X^m \tau_1 \dots \tau_{k-2} (\downarrow^{n_1} \tau_{k-1}) \tau_k X^{n_2}\dots X^{n_p}\\
+& X^m \tau_1 \dots \tau_{k-1}(\downarrow^{n_1} \tau_k)X^{n_2}\dots X^{n_p}\\
=& \dots \\
=& X^{m+n_1}\tau_1 \dots \tau_k  X^{n_2}\dots X^{n_p} + X^m (\downarrow^{n_1}\tau_1 \dots \tau_k  )X^{n_2}\dots X^{n_p}\\
=&X^{m+n_1}\omega X^{n_2}\dots X^{n_p} + X^m(\downarrow^{n_1}\omega)X^{n_2}\dots X^{n_p}.
\end{align*}
Repeat the process for each $X^{n_q}$ to get the statement.
\end{proof}

We now define the deformed Grossman--Larson product $\ast_+$ by:

\begin{align*}
(X^m \omega_1) \ast_+ (X^n\omega_2)=&(X^m \omega_1)_{(1)}\cdot ((X^m \omega_1)_{(2)} \dgraftt X^n\omega_2)
\end{align*}

Now using the above lemmas, we can rewrite the above expression:

\begin{align*}
(X^m \omega_1)_{(1)}\cdot ((X^m \omega_1)_{(2)} \dgraftt X^n\omega_2)=&\sum_{m_1+m_2=m \atop n_1+n_2=n} {n \choose n_2} X^{m_1+n_1}(\downarrow^{n_2}(\omega_1)_{(1)}) \cdot (X^{m_2}(\omega_1)_{(2)}\dgraftt \omega_2  ) \\
=&\sum_{m_1+m_2=m \atop n_1+n_2=n} {n \choose n_2} X^{m_1+n_1}(\downarrow^{n_2}(\omega_1)_{(1)})\cdot ( (\omega_1)_{(2)}\dgraftt (\uparrow^{m_2} \omega_2) - (\downarrow^{m_2}(\omega_1)_{(2)})\dgraftt \omega_2 ).
\end{align*}

Finally, we introduce the notation $\uparrow^m_{N_{\omega_2}}$ to mean $\uparrow^m$ acting only on the vertices of $\omega_2$, including the root. Then we can rewrite:

\begin{align}
(X^m \omega_1) \ast_+ (X^n\omega_2)=\uparrow^m_{N_{\omega_2}} \sum_{n=n_1+n_2} {n \choose n_2} X^{n_1} (\downarrow^{n_2}(\omega_1)_{(1)}) \cdot ( (\omega_1)_{(2)} \dgraftt \omega_2 ). \label{eq::DeformedGrossmanLarson}
\end{align}

\begin{proposition}
$(\mathfrak{T}^+,\ast_+,\Delta_{\shuffle})$ is a Hopf algebra.
\end{proposition}

\begin{proof}
This is a standard result for the enveloping algebra over a post-Lie algebra, see for example \cite{EbrahimiFardLundervoldMuntheKaas2014}.
\end{proof}

Extend the product $\ast_+$ to be an action $\ast_+ : \mathfrak{T}^+ \otimes \mathfrak{T} \to \mathfrak{T}$ in the natural way, meaning that we take $\dgraftt$ to still mean deformed grafting on all non-root non-noise vertices and define $\ast^+$ by equation \eqref{eq::DeformedGrossmanLarson}. Denote the dual coaction by $\Delta+: \mathfrak{T} \to \mathfrak{T}^+ \otimes \mathfrak{T}$. Let $H^+=(\mathfrak{T}^+,\shuffle,\Delta^+)$ be the dual Hopf algebra and let $G^+$ be its group of characters. We define the recentering group for our regularity structure by $\mathcal{G}=\{\Gamma_g=(g \otimes \text{id})\Delta^+ : g \text{ character of }H^+  \}$.

\begin{lemma}
The recentering maps $\Gamma_g : \mathfrak{T} \to \mathfrak{T}$ satisfy
\begin{align*}
|\Gamma_g \omega - \omega | < |\omega|
\end{align*}
for all $\omega \in \mathfrak{T}$ and all $g \in H^+$.
\end{lemma}

\begin{proof}
Since edge decorations contribute negatively to the regularity, and vertex decorations contribute positively, we have $|x \ast_+ y|=|x|+|y|$. Hence, if $\omega$ is homogeneous with regularity $\beta$, we have $\Delta^+(\omega) \in \oplus_{m+n=\beta \atop m \geq 0} \mathfrak{T}^+_m \otimes \mathfrak{T}_n$. The component in $\mathfrak{T}^+_0 \otimes \mathfrak{T}_n$ is exactly $1 \otimes \omega$. Every other component in $\Delta^+(\omega)$ must be of regularity less than $\beta$ in the right tensor. So $\Gamma_g \omega = (g \otimes id)\Delta^+(\omega) = \omega$ + lower regularity terms.
\end{proof}

We would now like to say that $(A,\mathfrak{T},\mathcal{G})$ is a regularity structure, where $A$ is the set of regularities. However, this is not completely accurate as the set $A$ is not bounded from below. This issue was dealt with in \cite{BrunedHairerZambotti2019} in the following way. For each specific equation, one only needs to consider a specific subset of trees. Hence they formulate the concept of trees generated by \textit{rules}, where one is expected to find the rule from the specific equation. If the rule has a property called subcritical, then it only generates finitely many negative trees and one can build a regularity structure by restricting the (co)products to only trees generated by this rule. We will not go into details on rules for planar trees, and instead just define:

\begin{definition} \label{def::PlanarRegularityStructure}
A regularity structure built from restricting $(A,\mathfrak{T},\mathcal{G})$ to a subspace is called a planar regularity structure.
\end{definition}

\begin{proposition}
Restricting the planar regularity structure from Definition \ref{def::PlanarRegularityStructure} to trees with no vertex decorations and edge decorations $\mathcal{D}=\Xi$ recovers the rough path regularity structure from section \ref{sec::PlanarRoughPathsAreRegularityStructures}.
\end{proposition}

\begin{proof}
This follows from the fact that there is no deformation if there are no vertex decorations, and that the rough path regularity structure is obtained by applying the Guin-Oudom construction to non-deformed grafting.
\end{proof}

We are now interested in giving a combinatorial description of the action $\ast_+$ and the coaction $\Delta^+$. This is achieved by interpreting equation \eqref{eq::DeformedGrossmanLarson}. In the action $(X^m \omega_1) \ast_+ (X^n\omega_2)$, some branches of $\omega_1$ are grafted onto non-root vertices of $\omega_2$, where the grafting is done in the deformed sense and where branches grafted on the same vertex respect their planar order. The remaining branches are merged into the root of $\omega_2$ in the sense of the associative product. We can interpret this root-merging also as deformed grafting, but onto the root. The term ${n \choose n_2}X^{n_1} (\downarrow^{n_2}(\omega_1)_{(1)})$ is a generic deformed term. Letting $n_1=n-\ell$ and $n_2=\ell$, we could rewrite it as ${n \choose \ell}X^{n-\ell}(\downarrow^{\ell}(\omega_1)_{(1)})$. We note that this is exactly a deformed grafting, as we are simultaneously subtracting $\ell$ from the root decoration and from the decorations of the grafted edges, together with the appropriate combinatorial coefficient. Hence we can describe this as deformed-grafting $\omega_1$ onto $\omega_2$ where we include the root as a valid vertex to graft on. Finally, the root decoration $X^m$ of $X^m \omega_1$ gets split over all vertices in $\omega_2$ in the sense of $\uparrow^m_{N_{\omega_2}}$. We illustrate this with an example:

\begin{example}
\begin{align*}
{\color{blue} \Forest{[\delta[\beta,edge label = {node[midway,fill=white,scale=0.5]{$b$}}][\gamma,edge label = {node[midway,fill=white,scale=0.5]{$c$}}]]} } \ast_+ {\color{blue} \Forest{[\omega[\alpha,edge label = {node[midway,fill=white,scale=0.5]{$a$}}]]} }=&\sum_{\delta=\delta_1+\delta_2,\ell_1,\ell_2} {\delta \choose \delta_1,\delta_2}( {\omega \choose \ell_1,\ell_2} {\color{blue} \Forest{[\omega+\delta_1-\ell_1-\ell_2[\beta,edge label = {node[midway,fill=white,scale=0.5]{$b-\ell_1$}}][\gamma,edge label = {node[midway,fill=white,scale=0.5]{$c-\ell_2$}}][\alpha+\delta_2,edge label = {node[midway,fill=white,scale=0.5]{$a$}}]]} } + {\omega \choose \ell_1}{\alpha \choose \ell_2} {\color{blue}\Forest{[\omega+\delta_1+\ell_1[\beta,edge label = {node[midway,fill=white,scale=0.5]{$b-\ell_1$}}][\alpha+\delta_2-\ell_2,edge label = {node[midway,fill=white,scale=0.5]{$a$}}[\gamma,edge label = {node[midway,fill=white,scale=0.5]{$c-\ell_2$}}]]]}}\\
+&{\omega \choose \ell_2}{\alpha \choose \ell_1}{\color{blue} \Forest{[\omega+\delta_1-\ell_2[\gamma,edge label = {node[midway,fill=white,scale=0.5]{$c-\ell_2$}}][\alpha+\delta_2-\ell_1,edge label = {node[midway,fill=white,scale=0.5]{$a$}}[\beta,edge label = {node[midway,fill=white,scale=0.5]{$b-\ell_1$}}]]]} } + {\alpha \choose \ell_1,\ell_2}{\color{blue} \Forest{[\omega+\delta_1[\alpha+\delta_2-\ell_1-\ell_2,edge label = {node[midway,fill=white,scale=0.5]{$a$}}[\beta,edge label = {node[midway,fill=white,scale=0.5]{$b-\ell_1$}}][\gamma,edge label = {node[midway,fill=white,scale=0.5]{$c-\ell_2$}}]]]} } ).
\end{align*}
\end{example}

We now want to describe the coaction $\Delta^+$. Let $X^n\omega$ be a generic element in $\mathfrak{T}$. We say that a subset $c$ of edges in $X^n\omega$ is a left admissible cut if:
\begin{enumerate}
\item The set $c$ does not contain any noise edges.
\item Any path from the root of $\omega$ to a leaf of $\omega$ contains at most one edge in $c$.
\item If $e\in c$ is an edge in the cut, then every other edge that is outgoing from the same vertex as $e$ and is to the left of $e$ in the planar embedding, is also in $c$.
\end{enumerate}
For $c$ an admissible cut of $X^n\omega$, we define the pruned part $P^c(X^n \omega)$ and the trunk $T^c(X^n \omega)$. The pruned part $P^c(X^n \omega)$ is the tree obtained by cutting all of the edges in $c$ and grafting these edges on a new root with no decoration. Edges cut off from different vertices are shuffled. The trunk $T^c(X^n \omega)$ is the tree obtained by removing the edges in $c$ from $\omega$ and keeping the part that is connected to the root. For $P^c(X^n \omega), T^c(X^n \omega)$ being the trunk and the pruned part for the same cut $c$, we define the deformed tensor product $P^c(X^n \omega) \hat{\otimes} T^c(X^n \omega)$ as the sum over all tensors $P^c(X^n \omega) \otimes T^n(X^n \omega)$ where we simultaneously increase the decorations of vertices in the right tensor and by the same amount increase the decoration of the edge in the left tensor that was cut off from the vertex. See Example \ref{example::DeltaPlus} below, where in the rightmost term on the top line we add $\ell_1$ to a vertex decoration and to the edge that was cut off from this vertex. Multiplied by the appropriate combinatorial factor, being the reciprocal of what the grafting coefficient would have been. We furthermore sum over all ways to decrease decorations of vertices in the right tensor while increasing the decoration of the root in the left tensor by the same amount, see the root decoration $n$ in Example \ref{example::DeltaPlus}.

\begin{proposition}
The coproduct $\Delta^+$ can be described by left admissible cuts by:
\begin{align*}
	\Delta^+(X^n \omega)=\sum_c P^c(X^n \omega) \hat{\otimes} T^c(X^n \omega).
\end{align*}
\end{proposition}

\begin{proof}
The coaction $\Delta^+(X^n\omega)$, being dual to $\ast_+$, is a weighted sum (over some choice of basis) of elements $x,y$ such that $x \ast_+ y$ contains $X^n\omega$ as a summand. This means that the skeletons of $x,y$ have to be given by left admissible cuts. Then we sum over all ways to decorate these skeletons such that the resulting product has a term with the correct decoration, and multiply by the appropriate weights.
\end{proof}

\begin{example}\label{example::DeltaPlus} In the following computation, note that the sums are finite as we require the left tensor to have positive regularity.
\begin{align*}
\Delta^+({\color{blue} \Forest{[\delta[\beta,edge label = {node[midway,fill=white,scale=0.5]{$b$}}][\gamma,edge label = {node[midway,fill=white,scale=0.5]{$c$}}[\alpha,edge label = {node[midway,fill=white,scale=0.5]{$a$}}]]]} })
=& 1 \otimes {\color{blue} \Forest{[\delta[\beta,edge label = {node[midway,fill=white,scale=0.5]{$b$}}][\gamma,edge label = {node[midway,fill=white,scale=0.5]{$c$}}[\alpha,edge label = {node[midway,fill=white,scale=0.5]{$a$}}]]]} } 
+ {\color{blue} \Forest{[\delta[\beta,edge label = {node[midway,fill=white,scale=0.5]{$b$}}][\gamma,edge label = {node[midway,fill=white,scale=0.5]{$c$}}[\alpha,edge label = {node[midway,fill=white,scale=0.5]{$a$}}]]]} } \otimes 1
+ \sum_{n=n_1+n_2+n_3,\ell_1} \frac{1}{{n \choose n_1,n_2,n_3}{\delta+\ell_1-n_1 \choose \ell_1}}{\color{blue} \Forest{[n[\beta,edge label = {node[midway,fill=white,scale=0.5]{$b+\ell_1$}}]]} } \otimes {\color{blue} \Forest{[\delta+\ell_1-n_1[\gamma-n_2,edge label = {node[midway,fill=white,scale=0.5]{$c$}}[\alpha-n_3,edge label = {node[midway,fill=white,scale=0.5]{$a$}}]]]} }\\
+& \sum_{n=n_1+n_2+n_3,\ell_1} \frac{1}{{n \choose n_1,n_2,n_3}{\gamma+\ell_1 - n_3 \choose \ell_1}} {\color{blue} \Forest{[n[\alpha,edge label = {node[midway,fill=white,scale=0.5]{$a+\ell_1$}}]]} } \otimes {\color{blue} \Forest{[\delta-n_1[\beta-n_2,edge label = {node[midway,fill=white,scale=0.5]{$b$}}][\gamma+\ell_1-n_3,edge label = {node[midway,fill=white,scale=0.5]{$c$}}]]} }\\
+& \sum_{n=n_1+n_2,\ell_1,\ell_2} \frac{1}{{n \choose n_1,n_2} {\delta+\ell_1 - n_1 \choose \ell_1}{\gamma+\ell_2 - n_2 \choose \ell_2}} ({\color{blue} \Forest{[n[\beta,edge label = {node[midway,fill=white,scale=0.5]{$b+\ell_1$}}][\alpha,edge label = {node[midway,fill=white,scale=0.5]{$a+\ell_2$}}]]} }+ {\color{blue}\Forest{[n[\alpha,edge label = {node[midway,fill=white,scale=0.5]{$a+\ell_2$}}][\beta,edge label = {node[midway,fill=white,scale=0.5]{$b+\ell_1$}}]]}}) \otimes {\color{blue} \Forest{[\delta+\ell_1 - n_1[\gamma + \ell_2 - n_2,edge label = {node[midway,fill=white,scale=0.5]{$c$}}]]} }\\
+& \sum_{n,\ell_1,\ell_2} \frac{1}{{\delta+\ell_1+\ell_2-n \choose \ell_1,\ell_2}} {\color{blue} \Forest{[n[\beta,edge label = {node[midway,fill=white,scale=0.5]{$b+\ell_1$}}][\gamma,edge label = {node[midway,fill=white,scale=0.5]{$c+\ell_2$}}[\alpha,edge label = {node[midway,fill=white,scale=0.5]{$a$}}]]]} } \otimes {\color{blue} \Forest{[\delta + \ell_1 + \ell_2 -n]} }
+ \sum_{n=n_1+n_2+n_3+n_4} \frac{1}{{n \choose n_1,n_2,n_3,n_4}} {\color{blue} \Forest{[n]} } \otimes {\color{blue} \Forest{[\delta-n_1[\beta-n_3,edge label = {node[midway,fill=white,scale=0.5]{$b$}}][\gamma-n_4,edge label = {node[midway,fill=white,scale=0.5]{$c$}}[\alpha-n_2,edge label = {node[midway,fill=white,scale=0.5]{$a$}}]]]} }  .
\end{align*}
\end{example}

\subsection{The negative Hopf algebra}

We follow the constructions from \cite{BrunedManchon2022}, but adapted for planar trees.\\

 Let $\mathfrak{T}^-$ denote the subspace of $\mathfrak{T}$ generated by planar trees of negative regularity. We define an action of $\mathfrak{T}^-$ onto $\mathfrak{T}$ by
 \begin{align*}
 \tau_1 \diamond \tau_2 = \sum_{v \in \widetilde{N}_{\tau_2}} \tau_1 \diamond_v \tau_2,
 \end{align*}
 where $\tau_1 \in \mathcal{T}^-, \tau_2 \in \mathcal{T},$ $\widetilde{N}_{\tau_2}$ is the set of vertices in $\tau_2$ that are not adjacent to a noise edge and $\diamond_v$ means insertion of $\tau_1$ into the vertex $v$. By inserting $\tau_1$ into $v$, we mean that we identify the root of $\tau_1$ with the vertex $v$. The outgoing edges from $v$ are interpreted as getting left-grafted onto $\tau_1$, i.e. we sum over all ways to left-graft these edges onto vertices of $\tau_1$ such that the planar order is preserved when two edges are grafted onto the same vertex. We sum over all ways to increase the decoration of vertices in $\tau_1$ such that the sum of increases equals the decoration of $v$.
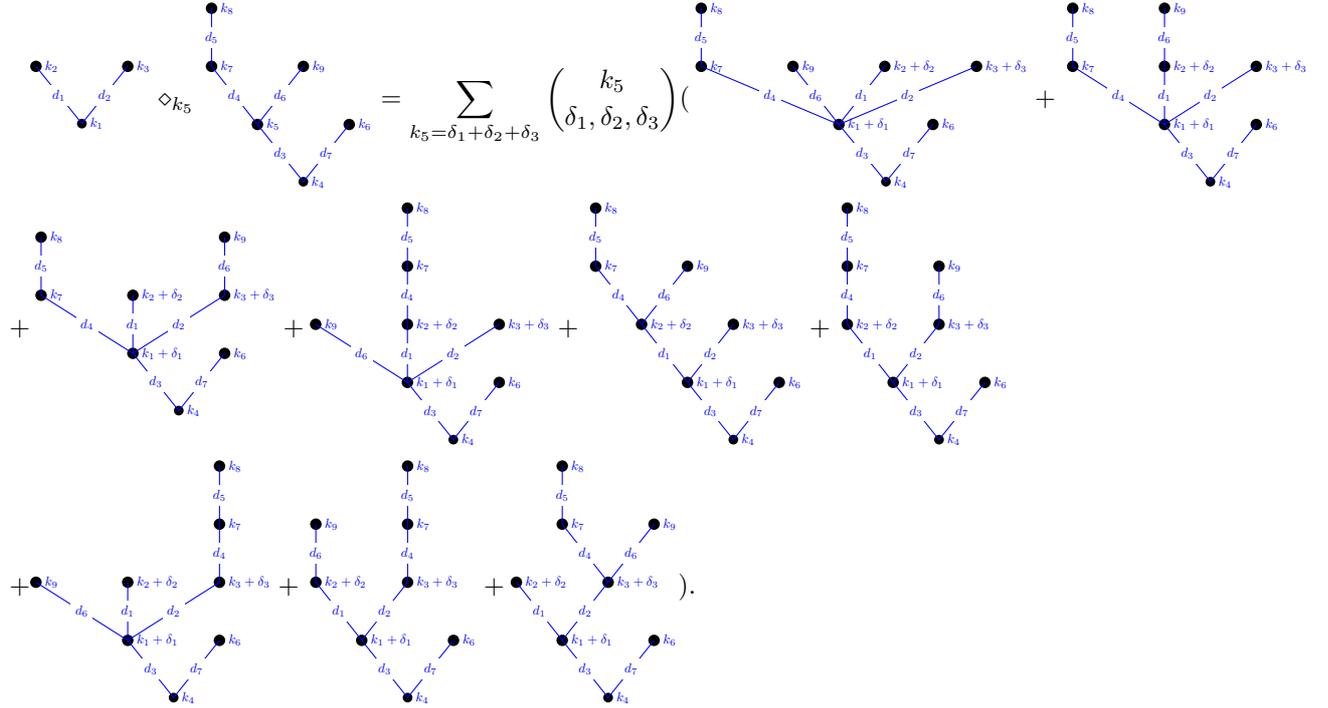
 \begin{example}
 \begin{align*}
 &{\color{blue} \Forest{ [k_1[k_2,edge label = {node[midway,fill=white,scale=0.5]{$d_1$}}][k_3,edge label = {node[midway,fill=white,scale=0.5]{$d_2$}}]] } }
 \diamond_{k_5} {\color{blue} \Forest{[k_4[k_5,edge label = {node[midway,fill=white,scale=0.5]{$d_3$}}[k_7,edge label = {node[midway,fill=white,scale=0.5]{$d_4$}}[k_8,edge label = {node[midway,fill=white,scale=0.5]{$d_5$}}]][k_9,edge label = {node[midway,fill=white,scale=0.5]{$d_6$}}]][k_6,edge label = {node[midway,fill=white,scale=0.5]{$d_7$}}]] } }
 =\sum_{k_5=\delta_1+\delta_2+\delta_3}{k_5 \choose \delta_1,\delta_2,\delta_3}({\color{blue} \Forest{[k_4[k_1+\delta_1,edge label = {node[midway,fill=white,scale=0.5]{$d_3$}}[k_7,edge label = {node[midway,fill=white,scale=0.5]{$d_4$}}[k_8,edge label = {node[midway,fill=white,scale=0.5]{$d_5$}}]][k_9,edge label = {node[midway,fill=white,scale=0.5]{$d_6$}}][k_2+\delta_2,edge label = {node[midway,fill=white,scale=0.5]{$d_1$}}][k_3+\delta_3,edge label = {node[midway,fill=white,scale=0.5]{$d_2$}}]][k_6,edge label = {node[midway,fill=white,scale=0.5]{$d_7$}}]] } }
 +{\color{blue} \Forest{[k_4[k_1+\delta_1,edge label = {node[midway,fill=white,scale=0.5]{$d_3$}}[k_7,edge label = {node[midway,fill=white,scale=0.5]{$d_4$}}[k_8,edge label = {node[midway,fill=white,scale=0.5]{$d_5$}}]][k_2+\delta_2,edge label = {node[midway,fill=white,scale=0.5]{$d_1$}}[k_9,edge label = {node[midway,fill=white,scale=0.5]{$d_6$}}]][k_3+\delta_3,edge label = {node[midway,fill=white,scale=0.5]{$d_2$}}]][k_6,edge label = {node[midway,fill=white,scale=0.5]{$d_7$}}]] } }\\
 +&{\color{blue} \Forest{[k_4[k_1+\delta_1,edge label = {node[midway,fill=white,scale=0.5]{$d_3$}}[k_7,edge label = {node[midway,fill=white,scale=0.5]{$d_4$}}[k_8,edge label = {node[midway,fill=white,scale=0.5]{$d_5$}}]][k_2+\delta_2,edge label = {node[midway,fill=white,scale=0.5]{$d_1$}}][k_3+\delta_3,edge label = {node[midway,fill=white,scale=0.5]{$d_2$}}[k_9,edge label = {node[midway,fill=white,scale=0.5]{$d_6$}}]]][k_6,edge label = {node[midway,fill=white,scale=0.5]{$d_7$}}]] } }
 +{\color{blue} \Forest{[k_4[k_1+\delta_1,edge label = {node[midway,fill=white,scale=0.5]{$d_3$}}[k_9,edge label = {node[midway,fill=white,scale=0.5]{$d_6$}}][k_2+\delta_2,edge label = {node[midway,fill=white,scale=0.5]{$d_1$}}[k_7,edge label = {node[midway,fill=white,scale=0.5]{$d_4$}}[k_8,edge label = {node[midway,fill=white,scale=0.5]{$d_5$}}]]][k_3+\delta_3,edge label = {node[midway,fill=white,scale=0.5]{$d_2$}}]][k_6,edge label = {node[midway,fill=white,scale=0.5]{$d_7$}}]] } }
 +{\color{blue} \Forest{[k_4[k_1+\delta_1,edge label = {node[midway,fill=white,scale=0.5]{$d_3$}}[k_2+\delta_2,edge label = {node[midway,fill=white,scale=0.5]{$d_1$}}[k_7,edge label = {node[midway,fill=white,scale=0.5]{$d_4$}}[k_8,edge label = {node[midway,fill=white,scale=0.5]{$d_5$}}]][k_9,edge label = {node[midway,fill=white,scale=0.5]{$d_6$}}]][k_3+\delta_3,edge label = {node[midway,fill=white,scale=0.5]{$d_2$}}]][k_6,edge label = {node[midway,fill=white,scale=0.5]{$d_7$}}]] } }
 +{\color{blue} \Forest{[k_4[k_1+\delta_1,edge label = {node[midway,fill=white,scale=0.5]{$d_3$}}[k_2+\delta_2,edge label = {node[midway,fill=white,scale=0.5]{$d_1$}}[k_7,edge label = {node[midway,fill=white,scale=0.5]{$d_4$}}[k_8,edge label = {node[midway,fill=white,scale=0.5]{$d_5$}}]]][k_3+\delta_3,edge label = {node[midway,fill=white,scale=0.5]{$d_2$}}[k_9,edge label = {node[midway,fill=white,scale=0.5]{$d_6$}}]]][k_6,edge label = {node[midway,fill=white,scale=0.5]{$d_7$}}]] } }\\
 +&{\color{blue} \Forest{[k_4[k_1+\delta_1,edge label = {node[midway,fill=white,scale=0.5]{$d_3$}}[k_9,edge label = {node[midway,fill=white,scale=0.5]{$d_6$}}][k_2+\delta_2,edge label = {node[midway,fill=white,scale=0.5]{$d_1$}}][k_3+\delta_3,edge label = {node[midway,fill=white,scale=0.5]{$d_2$}}[k_7,edge label = {node[midway,fill=white,scale=0.5]{$d_4$}}[k_8,edge label = {node[midway,fill=white,scale=0.5]{$d_5$}}]]]][k_6,edge label = {node[midway,fill=white,scale=0.5]{$d_7$}}]] } }
 +{\color{blue} \Forest{[k_4[k_1+\delta_1,edge label = {node[midway,fill=white,scale=0.5]{$d_3$}}[k_2+\delta_2,edge label = {node[midway,fill=white,scale=0.5]{$d_1$}}[k_9,edge label = {node[midway,fill=white,scale=0.5]{$d_6$}}]][k_3+\delta_3,edge label = {node[midway,fill=white,scale=0.5]{$d_2$}}[k_7,edge label = {node[midway,fill=white,scale=0.5]{$d_4$}}[k_8,edge label = {node[midway,fill=white,scale=0.5]{$d_5$}}]]]][k_6,edge label = {node[midway,fill=white,scale=0.5]{$d_7$}}]] } }
 +{\color{blue} \Forest{[k_4[k_1+\delta_1,edge label = {node[midway,fill=white,scale=0.5]{$d_3$}}[k_2+\delta_2,edge label = {node[midway,fill=white,scale=0.5]{$d_1$}}][k_3+\delta_3,edge label = {node[midway,fill=white,scale=0.5]{$d_2$}}[k_7,edge label = {node[midway,fill=white,scale=0.5]{$d_4$}}[k_8,edge label = {node[midway,fill=white,scale=0.5]{$d_5$}}]][k_9,edge label = {node[midway,fill=white,scale=0.5]{$d_6$}}]]][k_6,edge label = {node[midway,fill=white,scale=0.5]{$d_7$}}]] } }).
 \end{align*}
 \end{example}
 
We can restrict $\diamond$ to be an action of $\mathfrak{T}^-$ onto itself, then we call it the insertion product.

\begin{lemma}\label{lemma::InsertionIsPreLie}
The algebra $(\mathfrak{T}^-,\diamond)$ is pre-Lie.
\end{lemma}

\begin{proof}
It is straightforward to see that
\begin{align*}
x \diamond (y \diamond z) - (x \diamond y) \diamond z
\end{align*}
is the sum over all possible ways to insert $x$ and $y$ into different vertices of $z$, since the terms where $x$ are inserted into vertices of $y$ cancels. Note that inserting $x$ and $y$ into the same vertex of $z$ is the same as inserting $x$ into the root of $y$, hence it cancels. For the same reason, we have that
\begin{align*}
y \diamond (x \diamond z) - (y \diamond x) \diamond z
\end{align*}
is also a sum over all possible ways to insert $x$ and $y$ into different vertices of $z$. Hence the two expressions are equal, which means that the algebra is pre-Lie.
\end{proof}

We now deform the action $\diamond$ in the same way as was done for non-planar trees in \cite{BrunedManchon2022}. Let the deformed action be defined by:
\begin{align*}
\tau_1 \dinsert \tau_2 = \sum_{v \in \widetilde{N}_{\tau_2}} \tau_1 \dinsert_v \tau_2,
\end{align*}
where $\dinsert_v$ means insertion of $\tau_1$ into the vertex $v$. By inserting $\tau_1$ into $v$, we mean that we identify the root of $\tau_1$ with the vertex $v$. The outgoing edges from $v_1$ are interpreted as getting \textit{deformed} left-grafted onto $\tau_1$, i.e. we sum over all ways to left-graft these edges onto vertices of $\tau_1$ such that the planar order is preserved when two edges are grafted onto the same vertex and we deform each of the edges in this grafting. We sum over all ways to increase the decoration of vertices in $\tau_1$ such that the sum of increases equals the decoration of $v$.

 \begin{example}
\begin{align*}
&{\color{blue} \Forest{ [k_1[k_2,edge label = {node[midway,fill=white,scale=0.5]{$d_1$}}][k_3,edge label = {node[midway,fill=white,scale=0.5]{$d_2$}}]] } }
\dinsert_{k_5} {\color{blue} \Forest{[k_4[k_5,edge label = {node[midway,fill=white,scale=0.5]{$d_3$}}[k_7,edge label = {node[midway,fill=white,scale=0.5]{$d_4$}}[k_8,edge label = {node[midway,fill=white,scale=0.5]{$d_5$}}]][k_9,edge label = {node[midway,fill=white,scale=0.5]{$d_6$}}]][k_6,edge label = {node[midway,fill=white,scale=0.5]{$d_7$}}]] } }
=\sum_{k_5=\delta_1+\delta_2+\delta_3,\ell_1,\ell_2} {k_5 \choose \delta_1,\delta_2,\delta_3} ( {k_1 \choose \ell_1,\ell_2} {\color{blue} \Forest{[k_4[k_1+\delta_1\atop -\ell_1-\ell_2,edge label = {node[midway,fill=white,scale=0.5]{$d_3$}}[k_7,edge label = {node[midway,fill=white,scale=0.5]{$d_4-\ell_1$}}[k_8,edge label = {node[midway,fill=white,scale=0.5]{$d_5$}}]][k_9,edge label = {node[midway,fill=white,scale=0.5]{$d_6-\ell_2$}}][k_2+\delta_2,edge label = {node[midway,fill=white,scale=0.5]{$d_1$}}][k_3+\delta_3,edge label = {node[midway,fill=white,scale=0.5]{$d_2$}}]][k_6,edge label = {node[midway,fill=white,scale=0.5]{$d_7$}}]] } }\\
+&{k_1 \choose \ell_1}{k_2 \choose \ell_2}{\color{blue} \Forest{[k_4[k_1+\delta_1-\ell_1,edge label = {node[midway,fill=white,scale=0.5]{$d_3$}}[k_7,edge label = {node[midway,fill=white,scale=0.5]{$d_4-\ell_1$}}[k_8,edge label = {node[midway,fill=white,scale=0.5]{$d_5$}}]][k_2+\delta_2-\ell_2,edge label = {node[midway,fill=white,scale=0.5]{$d_1$}}[k_9,edge label = {node[midway,fill=white,scale=0.5]{$d_6-\ell_2$}}]][k_3+\delta_3,edge label = {node[midway,fill=white,scale=0.5]{$d_2$}}]][k_6,edge label = {node[midway,fill=white,scale=0.5]{$d_7$}}]] } }
+{k_1 \choose \ell_1}{k_3 \choose \ell_2} {\color{blue} \Forest{[k_4[k_1+\delta_1-\ell_1,edge label = {node[midway,fill=white,scale=0.5]{$d_3$}}[k_7,edge label = {node[midway,fill=white,scale=0.5]{$d_4-\ell_1$}}[k_8,edge label = {node[midway,fill=white,scale=0.5]{$d_5$}}]][k_2+\delta_2,edge label = {node[midway,fill=white,scale=0.5]{$d_1$}}][k_3+\delta_3-\ell_2,edge label = {node[midway,fill=white,scale=0.5]{$d_2$}}[k_9,edge label = {node[midway,fill=white,scale=0.5]{$d_6-\ell_2$}}]]][k_6,edge label = {node[midway,fill=white,scale=0.5]{$d_7$}}]] } }
+{k_1 \choose \ell_2}{k_2 \choose \ell_1}{\color{blue} \Forest{[k_4[k_1+\delta_1-\ell_2,edge label = {node[midway,fill=white,scale=0.5]{$d_3$}}[k_9,edge label = {node[midway,fill=white,scale=0.5]{$d_6-\ell_2$}}][k_2+\delta_2-\ell_1,edge label = {node[midway,fill=white,scale=0.5]{$d_1$}}[k_7,edge label = {node[midway,fill=white,scale=0.5]{$d_4-\ell_1$}}[k_8,edge label = {node[midway,fill=white,scale=0.5]{$d_5$}}]]][k_3+\delta_3,edge label = {node[midway,fill=white,scale=0.5]{$d_2$}}]][k_6,edge label = {node[midway,fill=white,scale=0.5]{$d_7$}}]] } }\\
+&{k_2 \choose \ell_1,\ell_2} {\color{blue} \Forest{[k_4[k_1+\delta_1,edge label = {node[midway,fill=white,scale=0.5]{$d_3$}}[k_2+\delta_2\atop -\ell_1-\ell_2,edge label = {node[midway,fill=white,scale=0.5]{$d_1$}}[k_7,edge label = {node[midway,fill=white,scale=0.5]{$d_4-\ell_1$}}[k_8,edge label = {node[midway,fill=white,scale=0.5]{$d_5$}}]][k_9,edge label = {node[midway,fill=white,scale=0.5]{$d_6-\ell_2$}}]][k_3+\delta_3,edge label = {node[midway,fill=white,scale=0.5]{$d_2$}}]][k_6,edge label = {node[midway,fill=white,scale=0.5]{$d_7$}}]] } }
+{k_2 \choose \ell_1}{k_3 \choose \ell_2}{\color{blue} \Forest{[k_4[k_1+\delta_1,edge label = {node[midway,fill=white,scale=0.5]{$d_3$}}[k_2+\delta_2-\ell_1,edge label = {node[midway,fill=white,scale=0.5]{$d_1$}}[k_7,edge label = {node[midway,fill=white,scale=0.5]{$d_4-\ell_1$}}[k_8,edge label = {node[midway,fill=white,scale=0.5]{$d_5$}}]]][k_3+\delta_3-\ell_2,edge label = {node[midway,fill=white,scale=0.5]{$d_2$}}[k_9,edge label = {node[midway,fill=white,scale=0.5]{$d_6-\ell_2$}}]]][k_6,edge label = {node[midway,fill=white,scale=0.5]{$d_7$}}]] } }
+{k_1 \choose \ell_2}{k_3 \choose \ell_1}{\color{blue} \Forest{[k_4[k_1+\delta_1-\ell_2,edge label = {node[midway,fill=white,scale=0.5]{$d_3$}}[k_9,edge label = {node[midway,fill=white,scale=0.5]{$d_6-\ell_2$}}][k_2+\delta_2,edge label = {node[midway,fill=white,scale=0.5]{$d_1$}}][k_3+\delta_3-\ell_1,edge label = {node[midway,fill=white,scale=0.5]{$d_2$}}[k_7,edge label = {node[midway,fill=white,scale=0.5]{$d_4-\ell_1$}}[k_8,edge label = {node[midway,fill=white,scale=0.5]{$d_5$}}]]]][k_6,edge label = {node[midway,fill=white,scale=0.5]{$d_7$}}]] } }\\
+&{k_2 \choose \ell_2}{k_3 \choose \ell_1}{\color{blue} \Forest{[k_4[k_1+\delta_1,edge label = {node[midway,fill=white,scale=0.5]{$d_3$}}[k_2+\delta_2-\ell_2,edge label = {node[midway,fill=white,scale=0.5]{$d_1$}}[k_9,edge label = {node[midway,fill=white,scale=0.5]{$d_6-\ell_2$}}]][k_3+\delta_3-\ell_1,edge label = {node[midway,fill=white,scale=0.5]{$d_2$}}[k_7,edge label = {node[midway,fill=white,scale=0.5]{$d_4-\ell_1$}}[k_8,edge label = {node[midway,fill=white,scale=0.5]{$d_5$}}]]]][k_6,edge label = {node[midway,fill=white,scale=0.5]{$d_7$}}]] } }
+{k_3 \choose \ell_1,\ell_2}{\color{blue} \Forest{[k_4[k_1+\delta_1,edge label = {node[midway,fill=white,scale=0.5]{$d_3$}}[k_2+\delta_2,edge label = {node[midway,fill=white,scale=0.5]{$d_1$}}][k_3+\delta_3-\ell_1-\ell_2,edge label = {node[midway,fill=white,scale=0.5]{$d_2$}}[k_7,edge label = {node[midway,fill=white,scale=0.5]{$d_4-\ell_1$}}[k_8,edge label = {node[midway,fill=white,scale=0.5]{$d_5$}}]][k_9,edge label = {node[midway,fill=white,scale=0.5]{$d_6-\ell_2$}}]]][k_6,edge label = {node[midway,fill=white,scale=0.5]{$d_7$}}]] } }).
\end{align*}
\end{example}

\begin{lemma} \label{lemma::InsertionAsGrossmanLarson}
The following identity holds:
\begin{align*}
\tau_1 \dinsert_v \tau_2 = (P_v(\tau_2) \ast_+ \tau_1) \graft_v T_v(\tau_2),
\end{align*}
where $P_v(\tau_2)$ is the subtree of $\tau_2$ that has $v$ as a root, and $T_v(\tau_2)$ is the tree obtained by removing all branches attached to $v$ in $\tau_2$ and setting the decoration of $v$ to zero. The product $\graft_v$ means identifying the root of the left argument with the vertex $v$ in the right argument.
\end{lemma}

\begin{proof}
It is clear that the definition of $\dinsert$ agrees with the combinatorial description of $\ast_+$ from Section \ref{ssect:PositiveHopf}.
\end{proof}

\begin{proposition}
The algebra $(\mathfrak{T}^-,\dinsert)$ is pre-Lie.
\end{proposition}

\begin{proof}
The proof for Proposition $3.21$ in \cite{BrunedManchon2022} applies here, by using Lemma \ref{lemma::InsertionAsGrossmanLarson}.
\end{proof}

Because $(\mathfrak{T}^-,\dinsert)$ is pre-Lie, we can use the Guin-Oudom construction \cite{GuinOudom2008} to extend $\dinsert$ to the symmetric algebra $S(\mathfrak{T}^-)$ by:

\begin{align*}
1 \dinsert \omega =& \omega, \\
\tau_1\omega \dinsert \tau_2 =& \tau_1 \dinsert (\omega \dinsert \tau_2) - (\tau_1 \dinsert \omega) \dinsert \tau_2,\\
\omega_1 \dinsert (\omega_2\omega_3) =& ((\omega_1)_{(1)} \dinsert \omega_2)((\omega_1)_{(2)} \dinsert \omega_3),
\end{align*}
for $\tau_1,\tau_2 \in \mathfrak{T}^-$ and $\omega,\omega_1,\omega_2,\omega_3 \in S(\mathfrak{T}^-)$.

\begin{lemma}
The expression
\begin{align*}
\tau_1\dots \tau_n \dinsert \omega,
\end{align*}
with $\tau_i \in \mathfrak{T}^-$ and $\omega \in S(\mathfrak{T^-})$, can be interpreted as a sum over all ways to pair $\tau_i's$ with vertices in $\omega$ and deformed insert each $\tau_i$ into their respective vertex. If the number of $\tau_i's$ exceed the number of vertices in $\omega$, the expression is zero.
\end{lemma}

\begin{proof}
Note that $\tau_1\tau_2 \dinsert \omega = \tau_1 \dinsert (\tau_2 \dinsert \omega) - (\tau_1 \dinsert \tau_2) \dinsert \omega$ means deformed insertion of $\tau_1$ into $\tau_2 \dinsert \omega$, with the terms where $\tau_1$ is deformed inserted into $\tau_1$ removed. Hence both $\tau_1,\tau_2$ are deformed inserted into vertices of $\omega$. Furthermore, they most be inserted into different vertices, because the term in $(\tau_1 \dinsert \tau_2) \dinsert \omega$ where $\tau_1$ is inserted into the root of $\tau_2$ cancels out the term in $\tau_1 \dinsert (\tau_2 \dinsert \omega)$ where both trees are inserted into the same vertex. The same argument can be made inductively for $n \geq 2$.
\end{proof}

We now define a product $\ast_-$ on $S(\mathfrak{T}^-)$ by:

\begin{align*}
\omega_1 \ast_- \omega_2 =& (\omega_1)_{(1)}( (\omega_1)_{(2)} \dinsert \omega_2).
\end{align*}

\begin{lemma}
$(S(\mathfrak{T}^-),\ast_-,\Delta_{\shuffle})$ is a Hopf algebra.
\end{lemma}

\begin{proof}
This is a standard result for the Guin-Oudom construction.
\end{proof}

We call $(S(\mathfrak{T}^-),\ast_-,\Delta_{\shuffle})$ the Hopf algebra for negative renormalisation. 

Recall that $\dinsert$ was also an action $\dinsert: \mathfrak{T}^- \otimes \mathfrak{T} \to \mathfrak{T}$. We now extend this to an action $\dinsert: S(\mathfrak{T}^-) \otimes \mathfrak{T} \to \mathfrak{T}$ by:

\begin{align*}
\tau_1 \omega \dinsert \tau_2 = \tau_1 \dinsert (\omega \dinsert \tau_2) - (\tau_1 \dinsert \omega) \dinsert \tau_2,
\end{align*}
for $\tau_1 \in \mathfrak{T}^-$, $\omega \in S(\mathfrak{T}^-)$ and $\tau_2 \in \mathfrak{T}$. Denote the dual coaction by $\Delta^- : \mathfrak{T} \to S(\mathfrak{T}^-) \otimes \mathfrak{T}$. Now recall Definition \ref{def::AdmissiblePartitionRough} for admissible partitions to describe renormalisation of rough paths. We will use an almost identical definition to describe $\Delta^-$, with the major difference being that we now need to take in account decorations when we contract subtrees.

\begin{definition} \label{def::AdmissiblePartitionReg}
Let $\tau \in \mathfrak{T}$ be a tree and let $\tau_1,\dots,\tau_n$ be a set of disjoint (not necessarily spanning) subtrees of $\tau$. We say that $\tau_1,\dots,\tau_n$ is $\mathfrak{T}^-$-admissible if the following conditions hold:
\begin{enumerate}
\item If $e$ is an edge in $\tau_i$ and $e'$ is an edge in $\tau$ that is outgoing from the same vertex as $e$, and if $e'$ is to the right of $e$ in the planar embedding, then $e'$ is in $\tau_i$.
\item If $e=(v_1,v_2)$ is a noise edge in $\tau$ then either both $v_1$ and $v_2$ belong to the same subtree, or neither $v_1$ or $v_2$ belong to any subtree.
\item Each $\tau_i$ has negative regularity.
\end{enumerate}
For $\tau_1,\dots,\tau_n$ $\mathfrak{T}^-$-admissible, denote by $\tau / \tau_1 \dots \tau_n$ the tree obtained by contracting each $\tau_i$ into a single vertex. For each contraction, sum over all ways to reduce the decoration of vertices in $\tau_i$ and putting the sum of the reductions onto the vertex the subtree was contracted into. Furthermore, for every edge that is outgoing from the subtree, sum over all ways to simultaneously increase the typing of the edge and the decoration of its adjacent vertex in the subtree (such that the regularity of the subtree remains negative). Everything weighted by the appropriate combinatorial factor.
\end{definition}

\begin{proposition}
The coaction $\Delta^-$ can be described by:
\begin{align*}
\Delta^-(\tau)=\sum_{\tau_1,\dots,\tau_n \atop \mathfrak{T}^-\text{-admissible}} \tau_1 \dots \tau_n \otimes \tau/\tau_1 \dots \tau_n.
\end{align*}
\end{proposition}

\begin{proof}
Ignoring the decorations of the tree, the situation is exactly the same as for renormalisation of planarly branched rough paths. Hence we must contract the same type of subtrees as in that case. It remains to give the trees the correct decorations. Edges going out of a subtree are dually described as one subtree being deformed-grafted onto another, hence we have to sum over ways to increase the decoration. Dually, the vertex we are inserting the subtree into, has its decoration spread out over all vertices in the subtree, hence we have to sum over all ways to reverse this.
\end{proof}

\begin{corollary}
Restricting $\Delta^-$ to trees with no vertex decorations and edge decorations $\mathcal{D}=\Xi$ recovers the coaction from section \ref{sec::PlanarRoughPathsAreRegularityStructures} for renormalising planarly branched rough paths.
\end{corollary}

\begin{example}
Let $\Xi$ denote a noise edge, then:
\begin{align*}
&\Delta^-({\color{blue} \Forest{[k_1[k_2,edge label = {node[midway,fill=white,scale=0.5]{$d_1$}}][k_3,edge label = {node[midway,fill=white,scale=0.5]{$d_2$}}[k_4,edge label = {node[midway,fill=white,scale=0.5]{$d_3$}}]][k_5,edge label = {node[midway,fill=white,scale=0.5]{$\Xi$}}]]} }  )
=\sum_{\delta=\delta_1+\dots+\delta_4} \frac{1}{ {\delta \choose \delta_1, \dots,\delta_4 } } {\color{blue} \Forest{[k_1-\delta_1[k_2-\delta_2,edge label ={node[midway,fill=white,scale=0.5]{$d_1$}}][k_3-\delta_3,edge label = {node[midway,fill=white,scale=0.5]{$d_2$}}[k_4-\delta_4,edge label = {node[midway,fill=white,scale=0.5]{$d_3$}}]][k_5,edge label = {node[midway,fill=white,scale=0.5]{$\Xi$}}]]} } \otimes {\color{blue} \Forest{[\delta]} }\\
+& \sum_{\delta,\ell_1,\ell_2} \frac{1}{ {k_1+\ell_1+\ell_2-\delta \choose \ell_1,\ell_2} } {\color{blue} \Forest{[k_1-\delta+\ell_1+\ell_2[k_5,edge label = {node[midway,fill=white,scale=0.5]{$\Xi$}}]]} } \otimes {\color{blue} \Forest{[\delta[k_2,edge label = {node[midway,fill=white,scale=0.5]{$d_1+\ell_1$}}][k_3,edge label = {node[midway,fill=white,scale=0.5]{$d_2+\ell_2$}}[k_4,edge label = {node[midway,fill=white,scale=0.5]{$d_3$}}]]]}}\\
+& \sum_{\delta=\delta_1+\delta_2,\ell_1,\ell_2} \frac{1}{{\delta \choose \delta_1,\delta_2} {k_1+\ell_1-\delta_1 \choose \ell_1} {k_2+\ell_2-\delta_2 \choose \ell_2} } {\color{blue} \Forest{[k_1-\delta_1+\ell_1[k_3-\delta_2+\ell_2,edge label = {node[midway,fill=white,scale=0.5]{$d_2$}}][k_5,edge label = {node[midway,fill=white,scale=0.5]{$\Xi$}}]]} } \otimes ({\color{blue} \Forest{[\delta[k_2,edge label = {node[midway,fill=white,scale=0.5]{$d_1+\ell_1$}}][k_4,edge label = {node[midway,fill=white,scale=0.5]{$d_3+\ell_2$}}]]} } + {\color{blue} \Forest{[\delta[k_4,edge label = {node[midway,fill=white,scale=0.5]{$d_3+\ell_2$}}][k_2,edge label = {node[midway,fill=white,scale=0.5]{$d_1+\ell_1$}}]]} }  )\\
+&\sum_{\delta=\delta_1+\dots+\delta_3,\ell_1} \frac{1}{{\delta \choose \delta_1,\dots,\delta_3} {k_1+\ell_1-\delta_1 \choose \ell_1} } {\color{blue} \Forest{[k_1-\delta_1+\ell_1[k_3-\delta_2,edge label = {node[midway,fill=white,scale=0.5]{$d_w$}}[k_4-\delta_3,edge label = {node[midway,fill=white,scale=0.5]{$d_3$}}]][k_5,edge label = {node[midway,fill=white,scale=0.5]{$\Xi$}}]]} } \otimes {\color{blue} \Forest{[\delta[k_2,edge label = {node[midway,fill=white,scale=0.5]{$d_1+\ell_1$}}]]} }\\
+&\sum_{\delta=\delta_1+\dots,\delta_3,\ell_1} \frac{1}{{\delta \choose \delta_1,\dots,\delta_3} {k_3+\ell_1-\delta_3 \choose \ell_1} } {\color{blue} \Forest{[k_1-\delta_1[k_2-\delta_2,edge label = {node[midway,fill=white,scale=0.5]{$d_1$}}][k_3-\delta_3+\ell_1,edge label = {node[midway,fill=white,scale=0.5]{$d_2$}}][k_5,edge label = {node[midway,fill=white,scale=0.5]{$\Xi$}}]]} } \otimes {\color{blue} \Forest{[\delta[k_4,edge label = {node[midway,fill=white,scale=0.5]{$d_3+\ell_1$}}]]}  }\\
+& \sum_{\delta,\ell_1, \atop \delta'=\delta'_1+\delta'_2} \frac{1}{{\delta' \choose \delta'_1,\delta'_2} {k_1+\ell_1+\ell_2-\delta_1 \choose \ell_1,\ell_2} } {\color{blue} \Forest{[k_1-\delta+\ell_1+\ell_2[k_5,edge label = {node[midway,fill=white,scale=0.5]{$\Xi$}}]]} \Forest{[k_3-\delta'_1[k_4-\delta'_2,edge label = {node[midway,fill=white,scale=0.5]{$d_3$}}]]} } \otimes {\color{blue}  \Forest{[\delta[k_3,edge label = {node[midway,fill=white,scale=0.5]{$d_1+\ell_1$}}][\delta',edge label = {node[midway,fill=white,scale=0.5]{$d_2+\ell_2$}}]]}},
\end{align*}
where we note that the sums are finite as we interpret any term in the left tensor, that has positive regularity, as zero. Note that we do not include subtrees consisting of a single vertex, as these always has positive regularity. Also note that we do not subtract decorations from the vertex $k_5$, as this would dually correspond to grafting on a noise vertex.
\end{example}

\subsection{Cointeraction}

We are now interested in the cointeraction property 
\begin{align} \label{eq::Cointeraction}
(Id \otimes \Delta^+)\Delta^- = m^{1,3}(\Delta^- \otimes \Delta^-)\Delta^+,
\end{align}
where
\begin{align*}
m^{1,3}(a \otimes b \otimes c \otimes d)=ac \otimes b \otimes d.
\end{align*}
It is this property that makes the proof of Proposition \ref{prop::RoughPathRenormalisedModelIsModel} work, i.e., to show that the renormalised model is a model for planarly branched rough paths. The cointeraction property is important for regularity structures for the same reason. Because of this, we consider \eqref{eq::Cointeraction} important also for planar regularity structures.\\
In equation \eqref{eq::Cointeraction}, the left $\Delta^-$ in the right side of the equality will take an input argument from $\mathfrak{T}^+$. This is not problematic if one sees $\mathfrak{T}^+$ as a subspace of $\mathfrak{T}$, but that is not the correct way to read the equation. We want to define $\Delta^-$ on $\mathfrak{T}^+$ in such a way that it doesn't contract any subtree containing the root. In \cite{BrunedHairerZambotti2019}, this was achieved by giving the root that is added by $\Delta^+$ a special colour that is not allowed to be in an admissible partition. In \cite{BrunedManchon2022}, the authors defined a coaction $\Delta^-_{\rm{non-root}}$ that does not contract subtrees containing the root. For (planarly) branched rough paths, this was not an issue as no subtree containing the root could be negative. In this paper, we will use the solution from \cite{BrunedManchon2022} and we define $\Delta^-_{\rm{non-root}}$ as the coaction that contracts all admissible partitions that do not contain the root.\\
This is however not sufficient to get the cointeraction property, as there is furthermore an issue with $\Delta^+$ projecting onto positive trees in its left tensor. This issue was solved for regularity structures, in \cite{BrunedHairerZambotti2019}, by considering so-called extended decorations. We will consider the extended decorations later in the section, but first we follow \cite{BrunedManchon2022} and establish the cointeraction property when we forgo the projection part of $\Delta^+$. Define $\Delta^+_0: \mathfrak{T} \to \mathfrak{T} \overline{\otimes} \mathfrak{T}$ to be the same as $\Delta^+$, except for also keeping negative trees in the left tensor. We now have to consider the completed tensor $\overline{\otimes}$, as the deformation will give an infinite sum.

\begin{proposition} \label{prop::Cointeraction}
The cointeraction property 
\begin{align} \label{eq::CointeractionNonRoot}
(Id \otimes \Delta^+_0)\Delta^- = m^{1,3}(\Delta^-_{\rm{non-root}} \otimes \Delta^-)\Delta^+_0,
\end{align}
 holds.
\end{proposition}

Before we prove the proposition, we want to demonstrate the idea of the proof with an example. We know that both sides of \eqref{eq::CointeractionNonRoot} will produce trees with the same skeletons, possibly only differing by the decorations, because the cointeraction property holds for these coactions on undecorated trees. We have to show that both sides agree also on the decorations. A skeleton from the left side of \eqref{eq::Cointeraction}, applied to the tree $\omega$, is obtained by:
\begin{enumerate}
\item Fix a $\mathfrak{T}^-$-admissible partition $\tau_1,\dots,\tau_n$ of $\omega$.
\item Fix an admissible cut of $\omega$ that does not cut any edges internal to a $\tau_i$.
\item Put $\tau_1,\dots,\tau_n$ in the leftmost tensor.
\item Contract the $\tau_i$'s and perform the cut. Put the pruned part of the middle tensor, and the trunk in the rightmost tensor.
\end{enumerate}
The corresponding skeleton in the right side of \eqref{eq::CointeractionNonRoot} is obtained by first considering the same cut, and then noting that $\tau_1,\dots,\tau_n$ can be split into one $\mathfrak{T}^-$-admissible partition for the pruned part, and one $\mathfrak{T}^-$-admissible partition for the trunk. What needs to be shown is that both orders of operations will deform the decorations in the same way.

\begin{example} \label{ex::Cointeraction}
Consider the tree
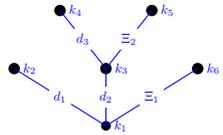
\begin{align*}
{\color{blue} \Forest{ [k_1 [k_2,edge label = {node[midway,fill=white,scale=0.5]{$d_1$}}] [k_3,edge label = {node[midway,fill=white,scale=0.5]{$d_2$}} [k_4,edge label = {node[midway,fill=white,scale=0.5]{$d_3$}}] [k_5,edge label = {node[midway,fill=white,scale=0.5]{$\Xi_2$}}] ] [k_6,edge label = {node[midway,fill=white,scale=0.5]{$\Xi_1$}}] ] } },
\end{align*}
together with the admissible cut $d_1,d_2$ and the $\mathfrak{T}^-$-admissible partition
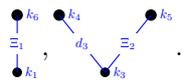
\begin{align*}
{\color{blue} \Forest{[k_1[k_6,edge label = {node[midway,fill=white,scale=0.5]{$\Xi_1$}}]]} } , {\color{blue} \Forest{[k_3[k_4,edge label = {node[midway,fill=white,scale=0.5]{$d_3$}}][k_5,edge label = {node[midway,fill=white,scale=0.5]{$\Xi_2$}}]]}  }.
\end{align*}
We now want to first perform the deformed cut, followed by contraction of the partition. And then first perform the contraction, followed by the cut. Performing the cut gives:
\begin{align*}
\sum_{\delta,\ell_1,\ell_2}  \frac{1}{{k_1+\ell_1+\ell_2 - \delta \choose \ell_1,\ell_2 } } {\color{blue} \Forest{[\delta[k_2,edge label = {node[midway,fill=white,scale=0.5]{$d_1+\ell_1$}}] [k_3,edge label = {node[midway,fill=white,scale=0.5]{$d_2+\ell_2$}} [k_4,edge label = {node[midway,fill=white,scale=0.5]{$d_3$}}] [k_5,edge label = {node[midway,fill=white,scale=0.5]{$\Xi_2$}}] ] ]} } \overline{\otimes} {\color{blue} \Forest{[k_1-\delta+\ell_1+\ell_2 [k_6,edge label = {node[midway,fill=white,scale=0.5]{$\Xi_1$}}] ]} }.
\end{align*}
Then the contraction gives the following term in the right side of \eqref{eq::CointeractionNonRoot}:
\begin{align} \label{eq::RightSide}
m^{1,3}\sum_{\delta,\ell_1,\ell_2}  \frac{1}{ {k_1+\ell_1+\ell_2 - \delta \choose \ell_1,\ell_2 } } [&(\sum_{n=n_1+n_2} \frac{1}{{n \choose n_1,n_2}} {\color{blue} \Forest{[k_3-n_1[k_4-n_2,edge label = {node[midway,fill=white,scale=0.5]{$d_3$}}] [k_5,edge label = {node[midway,fill=white,scale=0.5]{$\Xi_2$}}] ]}  } \otimes {\color{blue} \Forest{[\delta[k_2,edge label = {node[midway,fill=white,scale=0.5]{$d_1+\ell_1$}}][n,edge label = {node[midway,fill=white,scale=0.5]{$d_2+\ell_2$}}]]} })\\
\overline{\otimes}& (\sum_{m} {\color{blue} \Forest{[k_1 - \delta_1 + \ell_1 + \ell_2 - m[k_6,edge label = {node[midway,fill=white,scale=0.5]{$\Xi_1$}}]]} } \otimes {\color{blue} \Forest{[m]} })]. \nonumber
\end{align}
If we instead were to first do the contraction, we get:
\begin{align*}
\sum_{n=n_1+n_2\atop \mu,\kappa_1,\kappa_2} \frac{1}{ {n \choose n_1,n_2} {k_1+\kappa_1+\kappa_2-\mu \choose \kappa_1,\kappa_2}} {\color{blue} \Forest{ [k_3-n_1[k_2-n_2,edge label = {node[midway,fill=white,scale=0.5]{$d_3$}}] [k_5,edge label = {node[midway,fill=white,scale=0.5]{$\Xi_2$}}] ]  } \Forest{[k_1-\mu+\kappa_1+\kappa_2[k_6,edge label = {node[midway,fill=white,scale=0.5]{$\Xi_1$}}]]} } \otimes {\color{blue} \Forest{[\mu[k_2,edge label = {node[midway,fill=white,scale=0.5]{$d_1+\kappa_1$}}] [n,edge label = {node[midway,fill=white,scale=0.5]{$d_2+\kappa_2$}}] ]}   }.
\end{align*}
Then performing the cut gets us the following term in the left side of \eqref{eq::CointeractionNonRoot}:
\begin{align} \label{eq::LeftSide}
\sum_{n=n_1+n_2\atop \mu,\kappa_1,\kappa_2} \frac{1}{ {n \choose n_1,n_2} {k_1+\kappa_1+\kappa_2-\mu \choose \kappa_1,\kappa_2}} {\color{blue} \Forest{ [k_3-n_1[k_2-n_2,edge label = {node[midway,fill=white,scale=0.5]{$d_3$}}] [k_5,edge label = {node[midway,fill=white,scale=0.5]{$\Xi_2$}}] ]  } \Forest{[k_1-\mu+\kappa_1+\kappa_2[k_6,edge label = {node[midway,fill=white,scale=0.5]{$\Xi_1$}}]]} } \overline{\otimes} [\sum_{\epsilon,\kappa_1',\kappa_2'} \frac{1}{{\mu+\kappa_1'+\kappa_2'-\epsilon \choose \kappa_1',\kappa_2'}} {\color{blue} \Forest{[\epsilon[k_2,edge label = {node[midway,fill=white,scale=0.5]{$d_1+\kappa_1\atop +\kappa_1'$}}] [n,edge label = {node[midway,fill=white,scale=0.5]{$d_2+\kappa_2\atop +\kappa_2'$}}] ]} } \overline{\otimes} {\color{blue}\Forest{[\mu-\epsilon+\kappa_1'+\kappa_2']}}].
\end{align}
The claim is now that \eqref{eq::RightSide} and \eqref{eq::LeftSide} are equal. Note that we can match up the decorations by identifying:
\begin{align*}
m=&\mu-\epsilon+\kappa_1'+\kappa_2', \\
\delta=& \epsilon,\\
\ell_1=&\kappa_1+\kappa_1',\\
\ell_2=&\kappa_2+\kappa_2'.
\end{align*}
It remains to show that the combinatorial coefficients agree. For fixed $\delta,m,\ell_1,\ell_2,n_1,n_2$ The coefficient in \eqref{eq::RightSide} is:
\begin{align*}
\frac{1}{{k_1+\ell_1+\ell_2 - \delta \choose \ell_1,\ell_2 } {n \choose n_1,n_2} },
\end{align*}
and the corresponding coefficient in \eqref{eq::LeftSide} is:
\begin{align*}
\sum_{\kappa_1',\kappa_2'} \frac{1}{ {n \choose n_1,n_2} {k_1 +\ell_1+\ell_2-m-\delta \choose \ell_1-\kappa_1',\ell_2-\kappa_2'  } {m \choose \kappa_1',\kappa_2'  } }.
\end{align*}
The claim now follows from Chu--Vandermonde's identity:
\begin{align*}
\sum_{\kappa_1',\kappa_2'} {(k_1+\ell_1+\ell_2-\delta) - m \choose \ell_1-\kappa_1',\ell_2-\kappa_2'  }{m \choose \kappa_1',\kappa_2'} = {k_1+\ell_1+\ell_2-\delta \choose \ell_1,\ell_2}.
\end{align*}
\end{example}

We are now ready to generalise the above example into a proof for Proposition \ref{prop::Cointeraction}.

\begin{proof}
We have already seen that we can match the skeletons of both sides of equation \eqref{eq::CointeractionNonRoot}. Suppose that we have fixed a cut $c$ and a partition $\tau_1,\dots,\tau_n$ of the tree $\tau$, such that the cut does not include any edges internation to a $\tau_i$. Note that $\Delta^-$ will only change the decorations on vertices in each $\tau_i$, the decoration of the vertex the trees are contracted into and the decoration of edges outgoing from that vertex, also note that $\Delta^+$ will only change the decorations of the edges in $c$ and the vertices in the trunk $T^c(\tau)$. Hence we can conclude that the two sides of \eqref{eq::CointeractionNonRoot} will agree in terms of decorations and combinatorial coefficients for all $\tau_i$ that are above the cut, i.e. are in $P^c(\tau)$.\\ 
Next consider the case where a $\tau_i$ is below the cut, but not adjacent to a cut edge. First doing the cut will reduce the decoration of all non-noise vertices in $P^c(\tau)$. If the root of the cut branch has decoration $\delta$, then the vertices in $T^c(\tau)$ are reduced by $\delta_1,\dots,\delta_{m_1}$, all summing to $\delta$ and with coefficient ${{\delta \choose \delta_1,\dots,\delta_m}}^{-1}$. Say that $\delta_1,\dots,\delta_k$ belong to $\tau_i$. Then doing the contraction will again reduce the decoration of all vertices in $\tau_i$, say by $\kappa_1,\dots,\kappa_k$, and the vertex that $\tau_i$ is contracted into will have decoration $\kappa=\kappa_1+\dots+\kappa_k$, with coefficient ${ {\kappa \choose \kappa_1,\dots,\kappa_k} }^{-1}$. First doing the contraction will reduce the decorations of the vertices in $\tau_i$ by $\kappa_1',\dots,\kappa_k'$ and put the decoration $\kappa'=\kappa_1'+\dots+\kappa_k'$ on the contracted vertex. Doing the cut after this will reduce the decoration of the contracted vertex to $\kappa'-\delta'_1$, where $\delta'=\delta'_1+\dots+\delta'_{m_2}$ is the decoration of the root added by the cut. We match up the decorations of the two orders of operations by $\kappa=\kappa'-\delta'_1$, $\kappa'_j=\kappa_j+\delta_j$, $\delta'_1=\delta_1+\dots+\delta_k,\delta'_{1+q}=\delta_{k+q}$. Then we need to show that the corresponding coefficients 
$$
	\sum_{\delta_1,\dots,\delta_k} \frac{1}{{\delta \choose \delta_1,\dots,\delta_{m_1}} {\kappa \choose \kappa_1,\dots,\kappa_k} }
$$ 
and ${{\kappa' \choose \kappa_1',\dots,\kappa_k' } {\delta' \choose \delta_1',\dots,\delta'_{m_2} }}^{-1}$ are equal. This follows by the rewriting 
$$
{\delta \choose \delta_1,\dots,\delta_{m_1}} {\kappa \choose \kappa_1,\dots,\kappa_k} = {\delta' \choose \delta'_1,\dots,\delta'_{m_2}}{\delta'_1 \choose \delta_1,\dots,\delta_k }{\kappa'-\delta'_1 \choose \kappa_1-\delta_1,\dots,\kappa_k-\delta_k }
$$ 
and Chu--Vandermonde's identity.\\
Finally we need to consider the case where we cut an edge that is outgoing from a $\tau_i$. We have already seen that the deformations coming from the root decoration of the cut and from the contracted decoration will match. It remains to check the deformation of the cut edge. Let $n_v$ denote the decoration of the vertex in $\tau_i$ that the edge is outgoing from. First performing the cut will increase the decoration by $\ell$, and have a coefficient of ${{n_v + \ell \choose \ell}}^{-1}$. Then doing the contraction will not deform the edge. Instead doing the contraction first, will increase the decoration by $\kappa$ and will have a coefficient of ${{n_v + \kappa \choose \kappa}}^{-1}$. Performing the cut after the contraction will then increase the decoration by $\kappa'$ and have a coefficient of ${{n_v+\kappa+\kappa' \choose \kappa'}}^{-1}$. We can match the decoration of both cases by identifying $\ell=\kappa+\kappa'$, and the corresponding coefficients agree by Chu--Vandermonde's identity.
\end{proof}

The reason the cointeraction property fails for $\Delta^+: \mathfrak{T} \to \mathfrak{T}^+ \otimes \mathfrak{T}$, where we project onto positive trees, is because it matters if we contract a negative subtree before or after we do the projection. Consider the above Example \ref{ex::Cointeraction}. If we were to first do the cut using $\Delta^+$, we would get the term:
\begin{align*}
\sum_{\delta,\ell_1,\ell_2}  \frac{1}{{k_1+\ell_1+\ell_2 - \delta \choose \ell_1,\ell_2 } } {\color{blue} \Forest{[\delta[k_2,edge label = {node[midway,fill=white,scale=0.5]{$d_1+\ell_1$}}] [k_3,edge label = {node[midway,fill=white,scale=0.5]{$d_2+\ell_2$}} [k_4,edge label = {node[midway,fill=white,scale=0.5]{$d_3$}}] [k_5,edge label = {node[midway,fill=white,scale=0.5]{$\Xi_2$}}] ] ]} } \otimes {\color{blue} \Forest{[k_1-\delta+\ell_1+\ell_2 [k_6,edge label = {node[midway,fill=white,scale=0.5]{$\Xi_1$}}] ]} },
\end{align*}
where we only keep positive trees in the left tensor. The left tree contains the negative subtree:
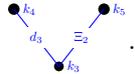
\begin{align*}
 {\color{blue} \Forest{[k_3[k_4,edge label = {node[midway,fill=white,scale=0.5]{$d_3$}}][k_5,edge label = {node[midway,fill=white,scale=0.5]{$\Xi_2$}}]]}  }.
\end{align*}
Contracting this subtree with $\Delta^-$ will improve the regularity of the tree. This improved regularity means that we could deform by bigger $\ell_1,\ell_2$ before the tree gets to negative regularity. Hence, if we contract before we cut, our sum will contain the extra terms of these bigger $\ell_1,\ell_2$ compared to if we cut before we contract.\\
We follow \cite{BrunedHairerZambotti2019} and solve this using extended decorations. Consider the space $\mathfrak{T}_{ex}$, given by trees where each node has two decorations. One of the node decorations is the usual $\mathbb{N}^d$ decoration. The other node decoration, which we call the \textit{extended decoration}, is in the set of regularities $A$. We now define a grading $|\cdot|_+$ on $\mathfrak{T}_{ex}$, given as the regularity of a tree in the usual sense plus all extended decorations. Let $\mathfrak{T}^+_{ex}$ be the subspace of trees that are positive with respect to $|\cdot|_+$. We now let $\Delta^+_{ex}: \mathfrak{T}_{ex} \to \mathfrak{T}^+_{ex} \otimes \mathfrak{T}$ be the same as $\Delta^+$ except that every node keeps its extended decoration, and we project onto positive trees according to the $|\cdot|_+$-degree. We furthermore let $\Delta^-_{ex} : \mathfrak{T}_{ex} \to \mathfrak{T}^- \otimes \mathfrak{T}_{ex}$ be the same as $\Delta^-$, except that the $|\cdot|_+$-degree of a contracted subtree gets put as the extended decoration on the vertex the tree is contracted into, and we let $\Delta^-_{\rm{ex-non-root}}:\mathfrak{T}^+_{ex} \to \mathfrak{T}^- \otimes \mathfrak{T}^+_{ex}$ be the same as $\Delta^-_{ex}$ except that we do not contract subtrees containing the root.

\begin{proposition}
The cointeraction property
\begin{align*}
(Id \otimes \Delta^+_{ex})\Delta^-_{ex}=m^{1,3}(\Delta^-_{\rm{ex-non-root}} \otimes \Delta^-_{ex})\Delta^+_{ex}
\end{align*}
holds.
\end{proposition}

\begin{proof}
We know that both sides of the equation would agree as infinite sums, and possibly only differ by which terms we keep in the projection done by $\Delta^+_{ex}$. However, every term in the right tensor of $\Delta^-_{ex}$ has the same $|\cdot|_+$-degree as the input tree. Therefore the projection done by $\Delta^+_{ex}$ will truncate the infinite sum at the same deformation parameters regardless if negative subtrees has been contracted or not.
\end{proof}

\bibliographystyle{acm}
\bibliography{RegularityStructuresReferences}
\end{document}